\documentclass[11pt, a4paper]{amsart}
\usepackage{amssymb,amsmath,amsthm,amscd}
\usepackage[shortlabels]{enumitem}
\usepackage{leftidx,mathrsfs}
\usepackage{graphicx, mathabx}
\usepackage{color,soul}
\usepackage{microtype}
\usepackage{dsfont,mathtools}
\usepackage{hyperref}
\usepackage[usenames,dvipsnames,svgnames,table]{xcolor}
\usepackage{tikz}
\usepackage{comment}
\usepackage{hhline,array}
\usepackage{caption}
\usepackage{subcaption}
\usepackage{faktor}
\usepackage{epstopdf} 
\usepackage{float}
\usepackage{tikz-cd}
\usepackage[all]{xy}
\usepackage{pgfplots}
\pgfplotsset{compat=1.16}
 
\newtheorem{lemma}{Lemma}[section]
\newtheorem{theorem}[lemma]{Theorem}

\newtheorem{mteor}{Main Theorem}

\newtheorem{proposition}[lemma]{Proposition}
\newtheorem{prop}[lemma]{Proposition}

\theoremstyle{definition}
\newtheorem{definition}[lemma]{Definition}
\newtheorem{defi}[lemma]{Definition}

\theoremstyle{remark}
\newtheorem{remark}{Remark}

\theoremstyle{definition}

\newcommand{\C}{\mathbb{C}}
\newcommand{\D}{\mathbb{D}}

\newcommand{\R}{\mathbb{R}}
\newcommand{\Z}{\mathbb{Z}}

\def\F{\mathcal{F}}
\def\G{\mathcal{G}}

\def\A{\mathcal{A}}

\newcommand{\cD}{\mathcal{D}}
\newcommand{\cE}{\mathcal{E}}
\newcommand{\cF}{\mathcal{F}}
\newcommand{\cH}{\mathcal{H}}
\newcommand{\cK}{\mathcal{K}}

\newcommand{\cU}{\mathcal{U}}
\newcommand{\cV}{\mathcal{V}}

\newcommand{\cT}{\mathcal{T}}

\DeclareMathOperator{\re}{Re}
\DeclareMathOperator{\im}{Im}
\DeclareMathOperator{\Int}{int}

\renewcommand{\epsilon}{\varepsilon}
\renewcommand{\phi}{\varphi}

\newcommand{\cc}[1]{\underset{#1}{\sim}}

\title[Mating parabolic rational maps with Hecke groups]{Mating parabolic rational maps\\ with Hecke groups}

\begin{author}[S.~Bullett]{Shaun Bullett}
\address{School of Mathematical Sciences, Queen Mary University of London, London E1 4NS, UK}
\email{s.r.bullett@qmul.ac.uk}
\end{author}
\thanks{}

\begin{author}[L.~Lomonaco]{Luna Lomonaco}
\address{Instituto de Matem{\'a}tica Pura e Aplicada, Estrada Dona Castorina 110, Jardim Bot{\^a}nico, Rio de Janeiro, RJ, CEP 22460-320, Brazil}
\email{luna@impa.br}
\thanks{L.L  was partially supported by the Serrapilheira Institute (grant number Serra-1811-26166), the FAPERJ - Fundação Carlos Chagas Filho de Amparo à Pesquisa do Estado do Rio de Janeiro (grant number JCNE - E26/201.279/2022 and JCM - E-26/210.016/2024), the ICTP through the Associates Programme and from the Simons Foundation through grant number 284558FY19}
\end{author}

\begin{author}[M.~Lyubich]{Mikhail Lyubich}
\address{Institute for Mathematical Sciences, Stony Brook University, 100 Nicolls Rd, Stony Brook, NY 11794-3660, USA}
\email{mlyubich@math.stonybrook.edu}
\end{author}
\thanks{M.L. was partially supported by NSF grants DMS-1901357 and 2247613.}

\begin{author}[S.~Mukherjee]{Sabyasachi Mukherjee}
\address{School of Mathematics, Tata Institute of Fundamental Research, 1 Homi Bhabha Road, Mumbai 400005, India}
\email{sabya@math.tifr.res.in}
\thanks{S.M. was partially supported by the Department of Atomic Energy, Government of India, under project no.12-R\&D-TFR-5.01-0500, an endowment of the Infosys Foundation, and SERB research project grant MTR/2022/000248.}
\end{author}
	
\begin{document}

\begin{abstract}
We prove that  any degree $d$ rational map having a parabolic fixed point of multiplier $1$ with a fully invariant and simply connected immediate basin of attraction is mateable with the Hecke group $\mathcal H_{d+1}$, with the mating realized by an algebraic correspondence.  
This confirms the parabolic version of a conjecture on mateability between rational maps and Hecke groups made in \cite{BF1}.
The proof is in two steps. The first is the construction of a pinched polynomial-like map which is a mating between a parabolic rational map and a parabolic circle map associated to the Hecke group. The second is lifting this pinched polynomial-like map to an algebraic correspondence via a suitable branched covering.
\end{abstract}

\maketitle

\setcounter{tocdepth}{1}
\tableofcontents

\section{Introduction}\label{intro_sec}

Algebraic correspondences which are {\it matings} between maps and groups were introduced by \cite{BP}, where a one complex parameter family $\F_a$ of holomorphic correspondences was proved to contain matings between quadratic polynomials $q_c:z \to z^2+c$ and the modular group $\mathrm{PSL}_2(\Z)$ for all real $c\in M$, where $M$ is the Mandelbrot set. It was conjectured there that the result could be extended to the case of $c$ complex. Subsequently this conjecture was proved for a large class of complex values of $c$ by the application of David homeomorphism techniques \cite{BH07}, but the full resolution of the conjecture was only achieved in \cite{BL1} with the introduction of the technique of {\it parabolic-like} mappings \cite{L1}, and the replacement of the words ``quadratic polynomials $q_c:z \to z^2+c$'' in the conjecture by the words ``parabolic quadratic rational maps $p_A: z \to z+1/z+A$''. (A relation between quadratic polynomials and parabolic quadratic rational maps was provided by Petersen and Roesch in \cite{PR21}). A related theory of algebraic correspondences realizing matings between quadratic polynomials and faithful discrete representations of the modular group in $\mathrm{PSL}_2(\C)$ was initiated in \cite{BH1}, developed in \cite{BH07}, and is the subject of recent work in \cite{RL1,RL2}. 
Matings between polynomials and Hecke groups were first studied in \cite{BF1} and \cite{BF2}. In \cite{BF1}, at the end of Section 3 it is conjectured that for every polynomial $P$ of degree $d$ with connected Julia set there exists a polynomial $Q$ of degree $d+1$ and an involution $J$, such that the $d:d$ correspondence $J\circ Cov_0^Q$ is a mating between $P$ and the Hecke group $\mathcal H_{d+1}$ (see Definition \ref{corr_defi} for the meaning of the notation $J\circ Cov_0^Q$).

In the late 2010s, a dynamical theory of Schwarz reflection maps in quadrature domains emerged, and abundant examples of matings (in a broader sense) of antiholomorphic rational maps and reflection groups were produced (for instance, see \cite{LM16,LLMM1,LLMM3,LMM23,LMMN}). In particular, the so-called \emph{cubic Chebyshev family} of anti-holomorphic correspondences was studied in \cite{LLMM3}. This is an antiholomorphic analogue of the family of correspondences from \cite{BP} mentioned above. Such correspondences were shown to arise as matings of parabolic quadratic anti-rational maps and an anti-conformal analogue of the modular group. In \cite{LMM23}, this mating phenomenon was generalized to arbitrary degree by establishing a parabolic version of the conjecture of \cite{BF1} in the antiholomorphic setting (more precisely, it was proved that all ``parabolic degree $d\geq 2$ anti-rational maps'' with connected Julia set can be mated with anti-conformal analogues of Hecke groups yielding certain correspondences).
In this paper, we will prove a parabolic version of the conjecture of \cite{BF1} in the holomorphic setting; i.e., we show that there exist algebraic correspondences realizing matings between all ``parabolic degree $d\geq 2$ rational maps'' and Hecke groups 
$\mathcal H_{d+1}$ (see Section~\ref{main_thm_subsec} for a precise statement). 

The general outline of the proof of this result is modeled on the classical Douady-Hubbard theory of polynomial-like maps that realizes such maps as matings of hybrid and external classes (or external maps) \cite{DH}, \cite[\S 3]{Lyu99}. Due to the presence of parabolic points belonging to the boundary of the domains of both the Hecke groups and the rational maps that we are going to mate, we need an extension of the theory of polynomial-like maps different from parabolic-like maps \cite{L1}. Specifically, we work with the category of \emph{pinched polynomial-like maps}, where the domain and the range are allowed to touch at a point. The hybrid classes for the relevant pinched polynomial-like maps are given by parabolic rational maps and their external maps are certain piecewise real-analytic, expansive circle coverings with a unique parabolic fixed point. 

We then extend the Douady-Hubbard mateability result for hybrid classes of polynomials with real-analytic, \emph{expanding} external maps to the current setting by demonstrating that hybrid classes of parabolic rational maps can be mated with a general class of piecewise real-analytic, \emph{expansive} maps having a \emph{unique} parabolic fixed point, yielding pinched polynomial-like maps. The proof of this result is more involved than its classical counterpart as the existence of parabolic points contributes to additional technical hurdles. 
These can be overcome using either the classical  Warschawski theorem or the invariant arcs technique of \cite{L1} .

We apply our construction to two particular maps, which we call the \emph{Hecke map} and the \emph{Farey map}. These maps can be obtained by selecting a Bowen-Series map for a suitable subgroup of the Hecke group $\mathcal H_{d+1}$, and quotienting it by a finite cyclic subgroup of order $2$ or $d+1$, respectively. The maps we obtain this way belong to the same conformal class.

The passage from these pinched polynomial-like maps (which are matings of parabolic rational maps with the Hecke/Farey map) to algebraic correspondences is based on passing to appropriate branched covering spaces. We carry it out in two ways: in a `geometric' way using 
basic branched covering theory for Riemann surfaces, and an `analytic' way based on quasiconformal surgery that yields an explicit characterization of the pinched polynomial-like map as an algebraic function.

\subsection{Statement of the main result}\label{main_thm_subsec}

Let $R: \widehat \C \rightarrow \widehat \C$ be a degree $d$ rational map with a parabolic fixed point of multiplier $1$ having a fully invariant and simply connected immediate basin of attraction $\A(R)$. Define the \emph{filled Julia set}
$$
K(R):= \widehat \C \setminus \A(R).
$$
Note that the parabolic fixed point of $R$ may have two fully invariant and simply connected immediate basins of attraction, for instance the map $B_d(z)= \frac{(d+1)z^d+ (d-1)}{(d+1)+(d-1) z^ d}$, in which case we make a choice for $\A(R)$.
\medskip

Hecke \cite{H} introduced the group
generated by the pair of M\"obius transformations $S: z\to -1/z$, $T_q: z\to z+2\cos(\pi/q)$, which he showed to be a discrete subgroup of $PSL(2,\R)$ for each integer $q\ge 3$ (the \emph{modular group} 
$PSL(2,\Z)$ is the case $q=3$). We shall work in the Poincar\'e disc rather than the upper half-plane. 
Let $\Pi$ be the regular ideal $(d+1)$-gon in the unit disc model of the hyperbolic plane with ideal vertices at the $(d+1)$-st roots of unity. Let us consider two conformal automorphisms of $\D$: $\rho$ is the rotation by $2\pi/(d+1)$ around the origin, and $\sigma$ is the involution preserving one of the sides of $\Pi$ and fixing the Euclidean mid-point of that side (see Figure~\ref{external_model_fig}). The {\it Hecke group} $\mathcal H_{d+1}$ is the group generated by $\sigma$ and $\rho$.

Define the elements $\alpha_j \in \mathcal H_{d+1}$, for $j=1,\cdots, d$ by 
$$\alpha_j:= \sigma\circ\rho^j.$$
We note that $\{\alpha_1,\alpha_2\}$ is a generating set for the Hecke group $\mathcal{H}_{d+1}$.

Before we state the Main Theorem, we introduce terminology and notation for the {\it correspondences} and {\it matings} with which we will be concerned in this article:
\begin{defi}\label{corr_defi}

By a {\it correspondence} $\F$ on the Riemann sphere $\widehat{\C}$ we will mean a multivalued map $z \to w$ (with multivalued inverse $w \to z$) whose graph is a (singular) Riemann surface in $\widehat{\C}\times\widehat{\C}$. By Chow's Theorem such a correspondence is {\it algebraic}. A consequence is the existence of a polynomial $P(z,w)$ such that
$$\F:z \to w \Leftrightarrow P(z,w)=0.$$
We say that $\F$ is {\it holomorphic} if every irreducible factor of $P(z,w)$ is of degree at least one in each of $z$ and $w$. 

Equivalently, a {\it holomorphic correspondence} on $\widehat{\C}$ is a multivalued function $\F=\pi_2\circ \pi_1^{-1}$, where $\pi_1$ and $\pi_2$ are holomorphic branched covering maps from a (closed) Riemann surface $X$ onto $\widehat{\C}$. (See Section 2.3 of \cite{BP1} for more details of these definitions, including proofs of equivalence).

\end{defi}

We note the following terminology and properties:
\begin{enumerate}[leftmargin=12mm]

\item
A holomorphic correspondence $\F:z \to w$ is of {\it bidegree} $(m:n)$ if generically each $w$ has $m$ inverse images and each $z$ has $n$ images. If $\F$ is defined by $P(z,w)=0$, and $P(z,w)$ has no repeated factors of degree $\ge1$, then $\F$ has bidegree $(m:n)$, where $m$ and $n$ are the degrees of $P(z,w)$ in $z$ and $w$ respectively.

\item
The {\it covering correspondence} of a rational map $Q$ is
$$Cov^Q: z \to w \iff Q(z)=Q(w),$$ and the {\it deleted covering correspondence} is
$$Cov_0^Q: z \to w \iff (Q(z)-Q(w))/(z-w)=0.$$

\item
The composition $\G \circ \F$ of two correspondences is defined by
$$w \in \G\circ\F(z) \iff \exists v\in\F(z)\ {\rm such\  that}\ w\in\G(v).$$ (Iteration of a correspondence $\F$ is now defined in the obvious way.)

\item
We note that both $Cov^Q$ and $Cov_0^Q$ are {\it symmetric} correspondences ($z \to w \iff w\to z$) and that $Cov^Q\circ Cov^Q=Cov^Q$.

\item
Given an involution $J$ on $\widehat{\C}$, the composition $J\circ Cov_0^Q$ is the correspondence defined by $z \to w \iff (Q(z)-Q(Jw))(z-Jw)=0$. We remark that $J=Cov_0^P$ for a rational map $P$ of degree $2$.
\end{enumerate}

\begin{defi}\label{mat} 
Let $R$ be a degree $d$ rational map with a parabolic fixed point of multiplier $1$ having a fully invariant and simply connected immediate basin of attraction.
We say that a $d:d$ holomorphic correspondence ${\mathcal F}: \widehat \C \rightarrow \widehat \C$ is a mating between a rational map $R$ and the Hecke group $\mathcal H_{d+1}$ if
\begin{enumerate}[leftmargin=12mm]
\item The dynamics of $\F$ gives rise to a partition of $\widehat \C$ into two non-empty completely invariant subsets: a connected closed set $\cK$ and an open simply connected set $\Omega$;

\item $\cK= \cK_- \cup \cK_+$, where $\cK_- \cap \cK_+=\{p\}$  is a single point, $\cK_+$ is forward invariant, $\cK_-$ is backward invariant, and $\F|_{\cK_-}$
is conformally conjugate to $\F^{-1}|_{\cK_+}$;

\item on a $d$-pinched neighbourhood of $\cK_-$ (pinched at $p$ and its preimages in $\cK_-$), a branch of $\F$ is hybrid equivalent to $R$ restricted to a $d$-pinched neighbourhood of $K(R)$ (pinched at the marked parabolic point and its preimages); and

\item when restricted to a $(d:d)$ correspondence from $\Omega$ to itself, $\mathcal F$ is conformally conjugate to the Hecke group acting on the unit disc. More precisely, there exists a conformal map $\Omega \rightarrow \D$ conjugating the $d$ branches of $\F: \Omega \rightarrow \Omega$ to the elements $\alpha_j:\D \rightarrow \D,\,\,j=1\ldots d$, of~$\mathcal{H}_{d+1}$.
\end{enumerate}
\end{defi}

\noindent (See Figure~\ref{double_covers_figure} (left).)

The principal result of this paper is the following:

\begin{mteor}\label{mainthm}
Let $R$ be a degree $d$ rational map with a parabolic fixed point of multiplier $1$ having a fully invariant and simply connected immediate basin of attraction.
Then there exists a $d:d$ holomorphic correspondence $\F$ on the Riemann sphere $\widehat{\C}$ which is a mating between $R$ and $\mathcal{H}_{d+1}$. Moreover, $\F=J \circ Cov_0^P$, 
where $J$ is a conformal involution and $P$ is a degree $d+1$ polynomial. 
\end{mteor}

Since $\mathcal H_3=PSL(2,\Z)$, we remark that in the case that $d=2$, Theorem \ref{mainthm} follows from part (i) of the Main Theorem of \cite{BL3}.

\subsection{Organization of the paper}\label{paper_org_subsec}

We give two self-contained, independent proofs of our main theorem. The first approach, which involves the Farey map and B-involutions, can be adapted for the construction of holomorphic correspondences realizing matings of other genus zero orbifold groups with complex polynomials \cite{MM23,LLM24}, and was motivated by antiholomorphic matings \cite{LMM23}.
On the other hand, the second approach involving the Hecke map and the double cover technique yields a more direct path from maps to correspondences, and is motivated by holomorphic matings \cite{BP}, \cite{BL1}.

Section~\ref{para_rat_sec} and~\ref{gen_mat_sec} are preparatory in nature, and are indispensable to both the proofs.
In Section~\ref{para_rat_sec}, we provide some background material on parameter space of parabolic rational maps. In Section~\ref{gen_mat_sec}, we give a background on the Douady-Hubbard theory of polynomial-like maps. The key part of this description is the theory of mating between hybrid and external classes/maps. Then we set the stage for a more general theory of pinched polynomial-like maps. In this setting, the objects to mate are parabolic rational maps and piecewise real-analytic, expansive circle maps with a unique parabolic fixed point. 

This point onward, the reader may select which line of arguments they would like to follow (Farey or Hecke). 
\medskip

\noindent\textbf{Road map for the first proof.}  \S\ref{para_rat_sec}$\to$~\ref{gen_mat_sec}$\to$~\ref{hecke_group_subsec}$\to$~\ref{farey_map_subsec}$\to$~\ref{qc_mating_subsec}$\to$~\ref{mating_construct_subsec}$\to$~\ref{farey_pinched_poly_subsec}$\to$~\ref{corr_from_b_inv_sec}.

In Sections~\ref{hecke_group_subsec} and~\ref{farey_map_subsec}, we define the Hecke group formally, and introduce the \emph{Farey map}, which is a piecewise analytic covering map of the circle cooked out of the Hecke group. The Farey map lives on the quotient of the hyperbolic plane by the order $d+1$ symmetry $\rho$. Sections~\ref{qc_mating_subsec},~\ref{mating_construct_subsec}, and~\ref{farey_pinched_poly_subsec} contain the mating construction between the Farey map and parabolic rational maps producing pinched polynomial-like maps. 
Finally, in Section~\ref{corr_from_b_inv_sec}, we give an algebraic description of the pinched polynomial-like maps obtained above, and use it to manufacture the desired algebraic correspondences by lifting the pinched polynomial-like maps by appropriate branched coverings.
\medskip

\noindent\textbf{Road map for the second proof.}
\S\ref{para_rat_sec}$\to$~\ref{gen_mat_sec}$\to$~\ref{hecke_group_subsec}$\to$~\ref{hecke_map_subsec}$\to$~\ref{qc_mating_subsec}$\to$~\ref{mating_construct_subsec}$\to$~\ref{hecke_pinched_poly_subsec}$\to$~\ref{RtoF}.

Sections~\ref{hecke_group_subsec} and~\ref{hecke_map_subsec} are devoted to a description of the Hecke group and the associated piecewise analytic circle covering called the \emph{Hecke map}. The Hecke map lives on the quotient of the hyperbolic plane by the order two symmetry $\sigma$. Sections~\ref{qc_mating_subsec},~\ref{mating_construct_subsec}, and~\ref{hecke_pinched_poly_subsec} contain the mating construction between the Hecke map and parabolic rational maps producing pinched polynomial-like maps. 
Finally, in Section~\ref{RtoF}, we obtain the desired algebraic correspondences by lifting these pinched polynomial-like maps by suitable double coverings. 

The curious reader who wishes to examine both proofs may find it helpful to look at Section~\ref{same_ext_class_subsec} (which describes an explicit conjugacy between the Farey and the Hecke maps) and Appendix~\ref{dictionary_sec} (containing a `dictionary' between the two different constructions of the correspondences).

\subsection*{Acknowledgment.} L.L would like
to thank ICTP, and S.B. would like to thank both ICTP and IMPA for their hospitality at different stages of this project.
Part of this work was done during S.M.'s visit to the IMS at Stony Brook, as well as during visits by M.L. and S.M. to the Fields Institute and the Centre for Nonlinear Analysis and Modeling (CNAM) in Toronto (March 2024), to Urgench State University, Uzbekistan (August 2023), and to IISER Pune, India (January 2024).
M.L. and S.M. gratefully acknowledge these institutions for their hospitality and support.

\section{Parabolic rational maps}\label{para_rat_sec}

Recall that for a degree $d$ rational map $R$ with a parabolic fixed point of multiplier $1$ having a (marked) fully invariant and simply connected immediate basin of attraction $\A(R)$, the filled Julia set $K(R)$ is defined to be $\widehat{\C}\setminus \mathcal{A}(R)$. By the Riemann-Hurwitz formula, $R$ has $d-1$ critical points (counted with multiplicity) in $\A(R)$.
One can employ a standard quasiconformal surgery argument (cf. \cite[Proposition~6.8]{McM}) to merge all the $d-1$ critical points of $R$ in $\A(R)$ into a single critical point of multiplicity $d-1$.
It follows from the above consideration that it suffices to prove Theorem~\ref{mainthm} for rational maps belonging to the following class:

\begin{definition}
Let $\pmb{\mathcal{B}}_d$ be the collection of degree $d\geq 2$ rational maps $R$ satisfying the following properties.
\begin{enumerate}[leftmargin=12mm]
\item $R$ has a parabolic fixed point of multiplier $1$ with a fully invariant and simply connected immediate basin of attraction $\A(R)$.
\item $\A(R)$ contains a unique critical point, which has multiplicity $d-1$.
\end{enumerate}
\end{definition}

Note that for $R\in\pmb{\mathcal{B}}_d$,
$$
R_{|\A(R)}\ \cc{\mathrm{conf}}\ B_{d|\D},
$$
where 
$$
B_d(z)= \frac{z^d+ \frac{d-1}{d+1}}{1+\frac{d-1}{d+1} z^ d}.
$$
Indeed, the conformal conjugacy is given by the Riemann uniformization of $\A(R)$ that sends the unique critical point of multiplicity $d-1$ to the origin and sends the parabolic fixed point to $1$. 
For convenience, we will normalize maps in $\pmb{\mathcal{B}}_d$ so that the marked parabolic fixed point of multiplier $1$ is at $\infty$.

The following supplementary result shows that like the connectedness locus of polynomials (of a fixed degree), the moduli space of the above parabolic rational maps (of a fixed degree) is also compact.

\begin{lemma}
The moduli space $\pmb{\mathcal{B}}_d/\mathrm{PSL_2(\mathbb C)}$ is compact.
\end{lemma}

\begin{proof}
Let $(R_n,\mathcal{A}(R_n))$ be a sequence of rational maps in $\pmb{\mathcal{B}}_d$ and their corresponding marked basins. By definition, there exist conformal maps $\psi_n: \mathbb D \to \mathcal{A}(R_n)$ that conjugate $B_d$ to $R_n$. After possibly conjugating $R_n$ by an affine map, we may assume that $\psi_n(0) = 0, \psi_n'(0) =1$. By compactness of normalized univalent functions, there exists a subsequence $\{\psi_{n_k}\}$ converging to some univalent map $\psi_\infty$. Thus the pointed domains $(\mathcal{A}(R_{n_k}),0)$ converge to the pointed domain $(\psi_\infty(\mathbb D),0)$ in the Carath\'eodory topology.
As $R_{n_k}\vert_{\mathcal{A}(R_{n_k}) }= \psi_{n_k} \circ B_d \circ \psi^{-1}_{n_k}$, these rational maps converge locally uniformly to some holomorphic map $R_\infty\colon \psi_\infty(\D) \to \widehat{\C}$. The fact that each $R_{n_k}$ is a rational map of degree $d$ implies that $R_\infty$ extends to a rational map of $\widehat{\C}$, of degree at most $d$. Finally, the conformal map $\psi_\infty^{-1}$ conjugates $R_\infty\vert_{\psi_\infty(\D)}$ to $B_d\vert_{\D}$, and hence $R_\infty$ must have degree $d$.
It is easy to see that $R_\infty$ has a parabolic point at $\infty$ and that $\psi_\infty(\mathbb D)$ is the desired marked immediate basin of $\infty$.
\end{proof}

We illustrate the dynamical behaviour of $B_d$ for $d=2,3$ and $4$ in Figure \ref{dyn_B}. In each plot the innermost curve bounds an attracting petal $\mathcal P$ containing the critical value of $B_d$. We 
choose this curve, $\partial \mathcal P$, to subtend a non-zero angle at the parabolic point $1\in \overline\D$ (and to be non-tangential to the unit circle there). Only the first few inverse image of $\partial\mathcal P$
are plotted, as otherwise the parabolic nature of the fixed point makes for a very congested picture. We remark that the dynamics of each $B_d$ on $\widehat\C\setminus\overline{\D}$ is an exact copy of its behaviour 
inside $\D$, by Schwarz reflection.

\begin{figure}
\captionsetup{width=0.96\linewidth}
\centering
\includegraphics[width=3cm]{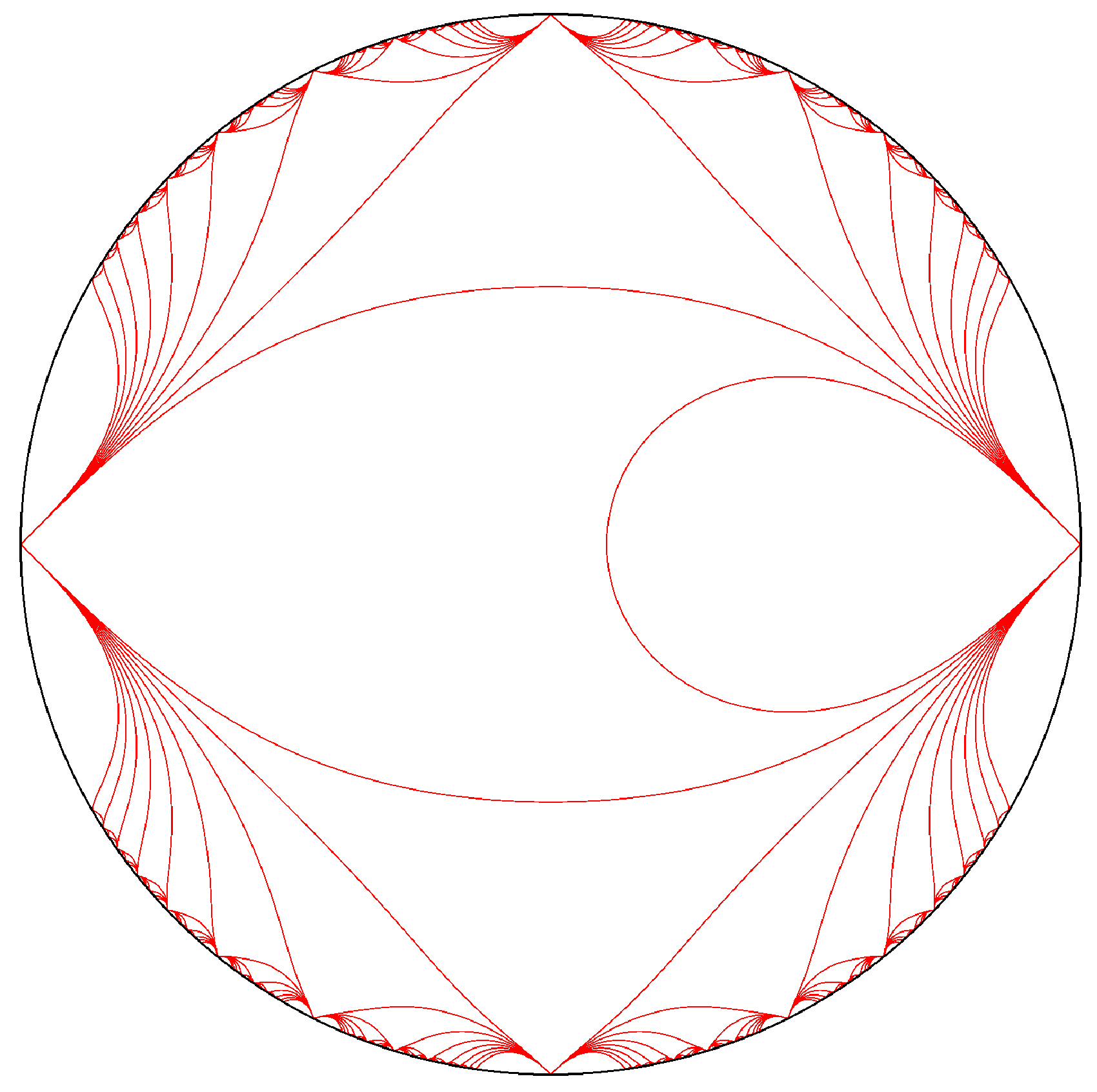} \hspace{1cm} \includegraphics[width=3cm]{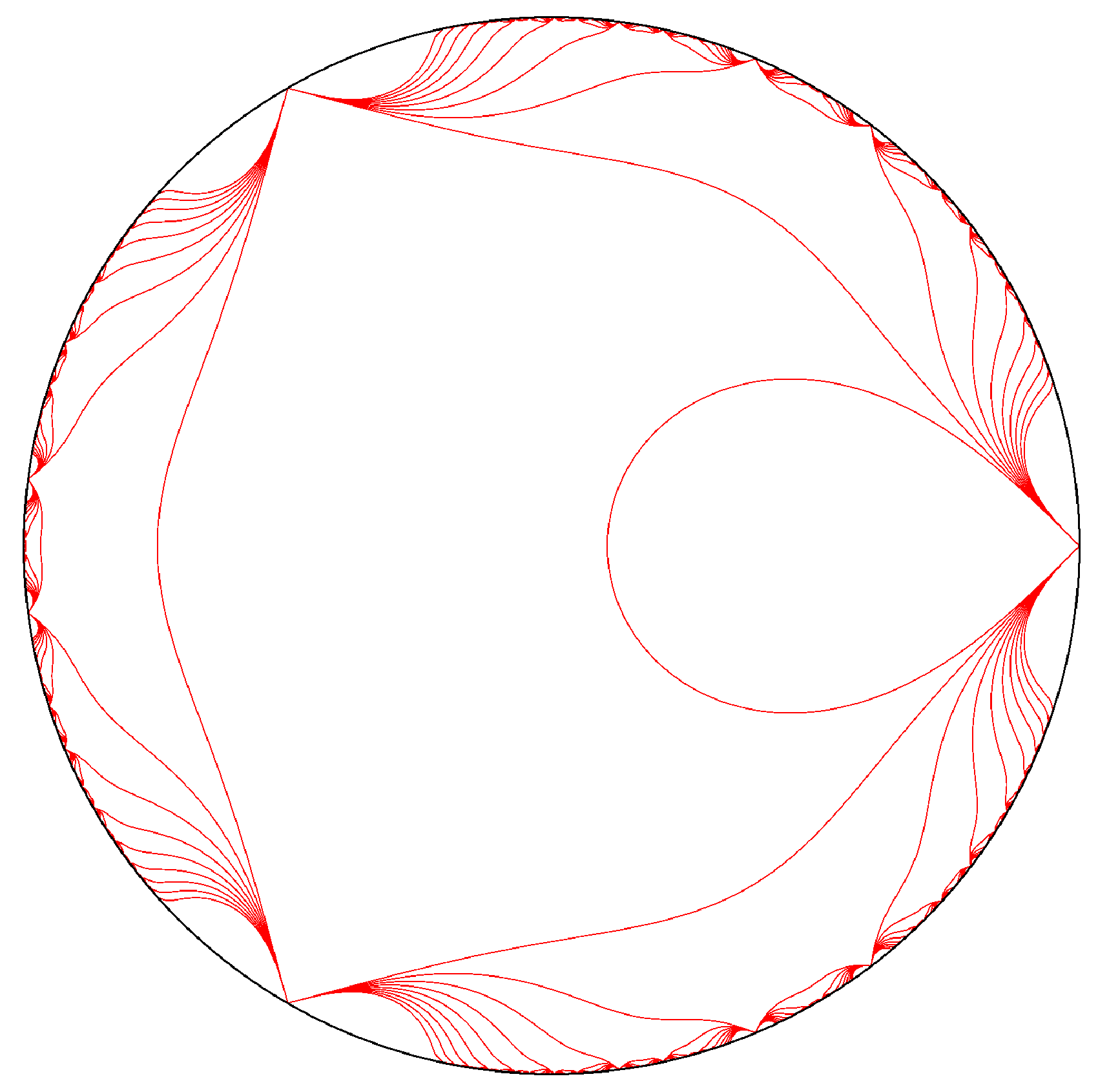} \hspace{1cm} \includegraphics[width=3cm]{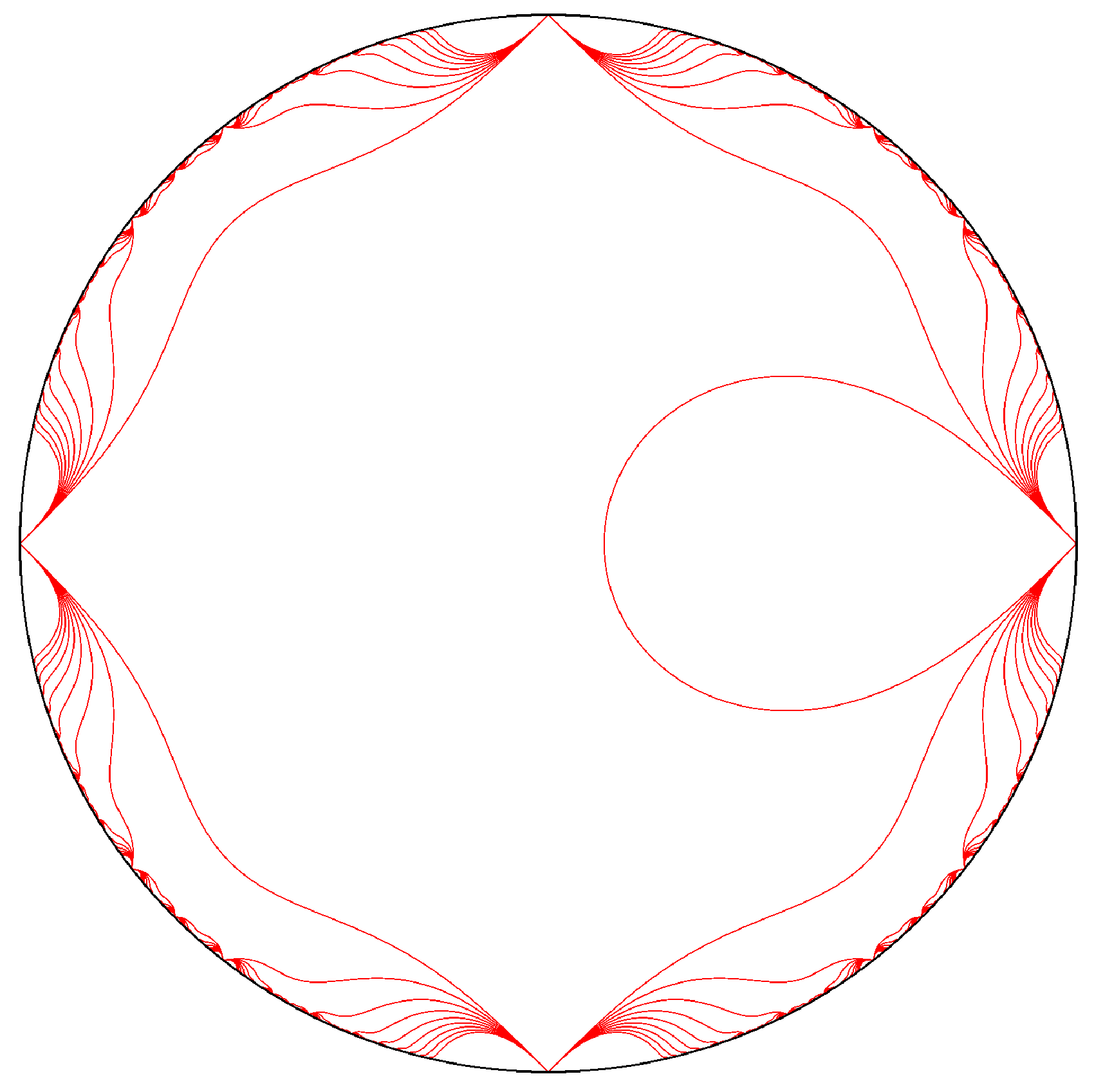} 
\caption{Dynamical behaviour of the Blaschke product maps $B_2$, $B_3$ and $B_4$ on the closed unit disc $\overline\D$.}\label{dyn_B}
\end{figure}

\section{Polynomial-like maps and pinched  polynomial-like maps}\label{gen_mat_sec}

As mentioned in the introduction, a first step in the construction of algebraic correspondences that are matings of parabolic rational maps and the Hecke group is to manufacture a partially defined holomorphic map on the sphere. This map captures certain dynamical features of the rational map as well as of the Hecke group, and can itself be regarded as a mating in a suitable sense. In fact, the mating structure of this intermediate holomorphic map can be explicated using an appropriate generalization of the classical notion of polynomial-like maps. In this section, we recall some generalities from the classical theory of polynomial-like maps, and set the stage for our main constructions.

\subsection{Quadratic-like maps: hybrid class and external map}

\begin{definition}\label{ql_map_def}
A \emph{quadratic-like map} $f:U\to V$ is a holomorphic double branched covering between two conformal discs $U$ and $V$ in $\C$ such that $U\Subset V$. The set of non-escaping points
$$
K(f):=\{z:f^{\circ n}(z)\in U,\ n=0,1,2,\cdots\}
$$
is called the \emph{filled Julia set} of $f$. Its boundary $J(f):=\partial K(f)$ is called the \emph{Julia set} of $f$
\end{definition}

The dynamics of a quadratic-like map can be decomposed into \emph{internal} and \emph{external} dynamics. In other words, a quadratic-like map can be regarded as the \emph{mating} of its internal and external dynamics in a precise way which we now formulate.

Two quadratic-like maps $f_j:U_j\to V_j$, $j\in\{1,2\}$, are said to be \emph{quasiconformally conjugate} if there exists a quasiconformal homeomorphism $h:(V_1, U_1)\to (V_2,U_2)$ such that 
$$
h\circ f_1=f_2\circ h\quad \textrm{on}\quad U_1.
$$ 
The maps $f_1, f_2$ are called \emph{hybrid conjugate/equivalent} if they are quasiconformally conjugate by a map $h$ with $\overline{\partial} h = 0$ a.e. on the filled Julia set $K(f_1)$.
The equivalence classes of hybrid conjugate quadratic-like maps are called \emph{hybrid classes}.
 
The hybrid class of a quadratic-like map captures its `internal' dynamics. By the Douady-Hubbard Straightening theorem \cite{DH}, each hybrid class of quadratic-like maps with connected Julia set contains a unique quadratic polynomial (up to affine conjugacy) with connected Julia set, or equivalently, a unique parameter in the Mandelbrot set. Thus, the Mandelbrot set is a catalog of all possible internal dynamics displayed by quadratic-like maps with connected Julia set.

If a quadratic-like map $f:U\to V$ has connected Julia set, one can capture its `external dynamics' by a real-analytic expanding circle map as follows. We consider a conformal map $\psi_f:(\widehat{\C}\setminus K(f),\infty)\to(\widehat{\C}\setminus\overline{\D},\infty)$, and set
$$
\Omega:= \psi_f(U\setminus K(f)),\quad \Omega':=\psi_f(V\setminus K(f)).
$$
Conjugating $f$ by $\psi_f$, we obtain a holomorphic double covering $\cE:\Omega\to\Omega'$. Since $\Omega, \Omega'$ are conformal annuli with common inner boundary $\mathbb{S}^1$, the Schwarz Reflection Principle allows us to extend $\cE$ to an annular neighbourhood of $\mathbb{S}^1$ such that the extended map $\cE$ restricts to a real-analytic double covering of the circle. Moreover, the fact that the outer boundary of $\Omega$ is contained in $\Omega'$, translates to the expanding property of $\cE$. We note that as $\psi_f$ is unique up to post-composition with a rotation, the circle endomorphism $\cE$ is also unique up to conjugation by a rotation. The map $\cE$ is called the \emph{external map} of $f$.

Equivalence classes of real-analytic expanding double coverings of the circle under the equivalence relation induced by real-analytic circle conjugacy are called \emph{external classes}. Although two conformally conjugate quadratic-like maps are dynamically indistinguishable, for many purposes (for instance, in renormalization theory), it is more convenient to consider the space of quadratic-like maps up to affine conjugacy instead of conformal conjugacy. This perspective, which was introduced in \cite{Lyu99}, will also be important in the current paper. Hence, we remark that while the conformal conjugacy class of a quadratic-like map is uniquely determined by its hybrid class and external class, one needs to specify the hybrid class and an external map (i.e., a specific representative from the external class) to determine a quadratic-like map uniquely up to affine conjugacy.

By a quasiconformal surgery argument (similar to the one used in the proof of the Straightening Theorem), one can show that any hybrid class of quadratic-like maps can be \emph{mated} with any quadratic external map; i.e, there exists a quadratic-like map with prescribed hybrid class and external map (cf. \cite[Section I.4]{DH}, \cite{Lyu24}). Further, if the hybrid class has connected Julia set, then the mating is unique up to affine conjugacy.
 
Finally, an \emph{external fiber} or \emph{vertical fiber} in the space of quadratic-like maps (up to affine conjugacy) is the collection of all quadratic-like maps admitting a given real-analytic expanding double covering of the circle as their external map.

\subsection{Polynomial-like maps}

With minor adjustment, everything that was said about quadratic-like maps above holds for arbitrary degree.

\begin{definition}\label{pl_map_def}
A \emph{polynomial-like map} $f:U\to V$ is a holomorphic branched covering of degree $d\geq 2$ between two conformal discs $U, V\subset \C$ with $U\Subset V$. 
\end{definition}

The notion of hybrid classes and external maps can be defined verbatim for general polynomial-like maps. However, to make sense of certain uniqueness statements, it is convenient to mark an invariant external access to $J(f)$. 

The Straightening Theorem then asserts that a degree $d$ polynomial-like map $f$ with connected Julia set is hybrid equivalent to a unique monic, centered, degree $d$ polynomial $P$ with connected Julia set (i.e., a unique member of the connectedness locus of monic centered polynomials) such that the hybrid conjugacy carries the marked invariant access to $J(f)$ to the access to $J(P)$ defined by the $0-$ray (where the B{\"o}ttcher coordinate of $P$ is taken to be tangent to the identity at $\infty$) \cite{DH}, \cite[Section~1]{IK}.

Note also that an invariant external access to $J(f)$ determines a marked fixed point for the external map of $f$.
The mechanism of mating hybrid classes with external maps yields the following statement: given an externally marked hybrid class $[f]$ of degree $d$ polynomial-like maps with connected Julia set and an expanding, real-analytic degree $d$ circle covering $\cE$ with a marked fixed point, there exists a degree $d$ polynomial-like map $F$, unique up to affine conjugacy, satisfying the following properties:
\begin{enumerate}
\item $F$ is hybrid equivalent to $f$;
\item $F$ has $\cE$ as its external map; and 
\item the hybrid conjugacy between $f$ and $F$ carries the invariant external access to $J(f)$ to an invariant external access to $J(F)$ such that this invariant external access to $J(F)$ determines the marked fixed point of~$\cE$.
\end{enumerate}

As in the quadratic setting, an \emph{external fiber} in the space of (marked) degree $d$ polynomial-like maps (up to affine conjugacy preserving the marking) is the sub-collection of all maps admitting a given real-analytic expanding degree $d$ circle covering as their external map.

\subsection{Parabolic-like maps}\label{para_like_maps_subsec}
Polynomial-like maps are objects that locally behave as polynomials about their filled Julia set. As we said in the previous subsection, a polynomial-like map of degree $d$ is determined up
to holomorphic conjugacy by its (marked) internal and external classes. In particular the external class is a degree $d$ real-analytic orientation
preserving and strictly expanding self-covering of the unit circle: the expanding feature of such a
circle map implies that all the periodic points are repelling, and in particular
not parabolic. So, polynomial-like maps do not model parabolic rational maps in $\pmb{\mathcal{B}}_d$.
To extend the Douady-Hubbard theory to parabolic settings, the second author introduced  \textit{parabolic-like maps}, local objects encoding the dynamics of rational maps in $\pmb{\mathcal{B}}_d$ in a neighbourhood of the filled Julia set (see  \cite{L1} for a precise definition), and proved that any degree $2$ parabolic-like map is hybrid equivalent to a parabolic rational map of the form $P_A(z)=z+1/z+A$, a unique such member if the filled Julia set is connected (see \cite{L1} and \cite{L2}).

A parabolic-like mapping is thus similar to, but different from, a polynomial-like mapping.
The similarity resides in the fact that a parabolic-like map is a local concept,
characterized by its internal and external class. The difference resides in the fact that
the external map of a degree $d$ parabolic-like mapping
is a degree
$d$ real-analytic orientation preserving self-covering of
the unit circle topologically expanding (this is, expansive) with parabolic fixed points (see \cite{LPS}).

\subsection{Pinched polynomial-like maps}\label{pinched_poly_like_subsec}
Parabolic-like maps are objects defined on a neighbourhood of a parabolic fixed point. It is useful to consider objects which can have `parabolic points' \textit{on the boundary} of the domain. This motivates us to introduce another class of objects, which we now formalize. We note that similar notions already appeared in various contexts in the literature \cite{DH,Mak93,BF2,LMM23}.

A \textit{polygon} is a closed Jordan disc in $\widehat{\C}$ with a piecewise smooth boundary. The points of intersection of these smooth boundary curves are called the \emph{break-points} or \emph{corners} of the polygon.

A \textit{pinched polygon} is a set in $\widehat{\C}$ which is homeomorphic to a closed disc quotiented by a finite geodesic lamination, and which has  a piecewise smooth boundary. The cut-points of a pinched polygon will be called its \emph{pinched points}, and the non-cut points of intersection of the smooth boundary curves are called the \emph{corners} of the pinched polygon. 

\begin{definition}\label{pinched_poly_def}
Let $P_2\subset \widehat{\C}$ be a polygon, and let $P_1\subset P_2$ be a pinched polygon such that $\partial P_1\cap \partial P_2$ is the set of corners of $P_2$ and is contained in the set of corners of $P_1$. 

Suppose that there is a holomorphic map $f\colon \Int{P_1}\to \Int{P_2}$ such that
\begin{enumerate}[leftmargin=10mm]
\item $f$ is a branched cover from each component of $\Int{P_1}$ onto $\Int{P_2}$;
\item $f$ extends continuously to the boundary of $\Int{P_1}$; and
\item the corners and pinched points of $P_1$ are the preimages of the corners of $P_2$.
\end{enumerate}

We then call the triple $(f,P_1, P_2)$ a \emph{pinched polynomial-like map}.
\end{definition}
(See Figure~\ref{pinched_poly_fig}.)

\begin{remark}
We assume that $P_1$ is a pinched polygon (instead of requiring it to be a polygon) in order to allow some of the corners of $P_2$ to be critical values of $f$.  
\end{remark}

\begin{figure}
\captionsetup{width=0.96\linewidth}
\begin{tikzpicture}
\node[anchor=south west,inner sep=0] at (1,0) {\includegraphics[width=0.6\textwidth]{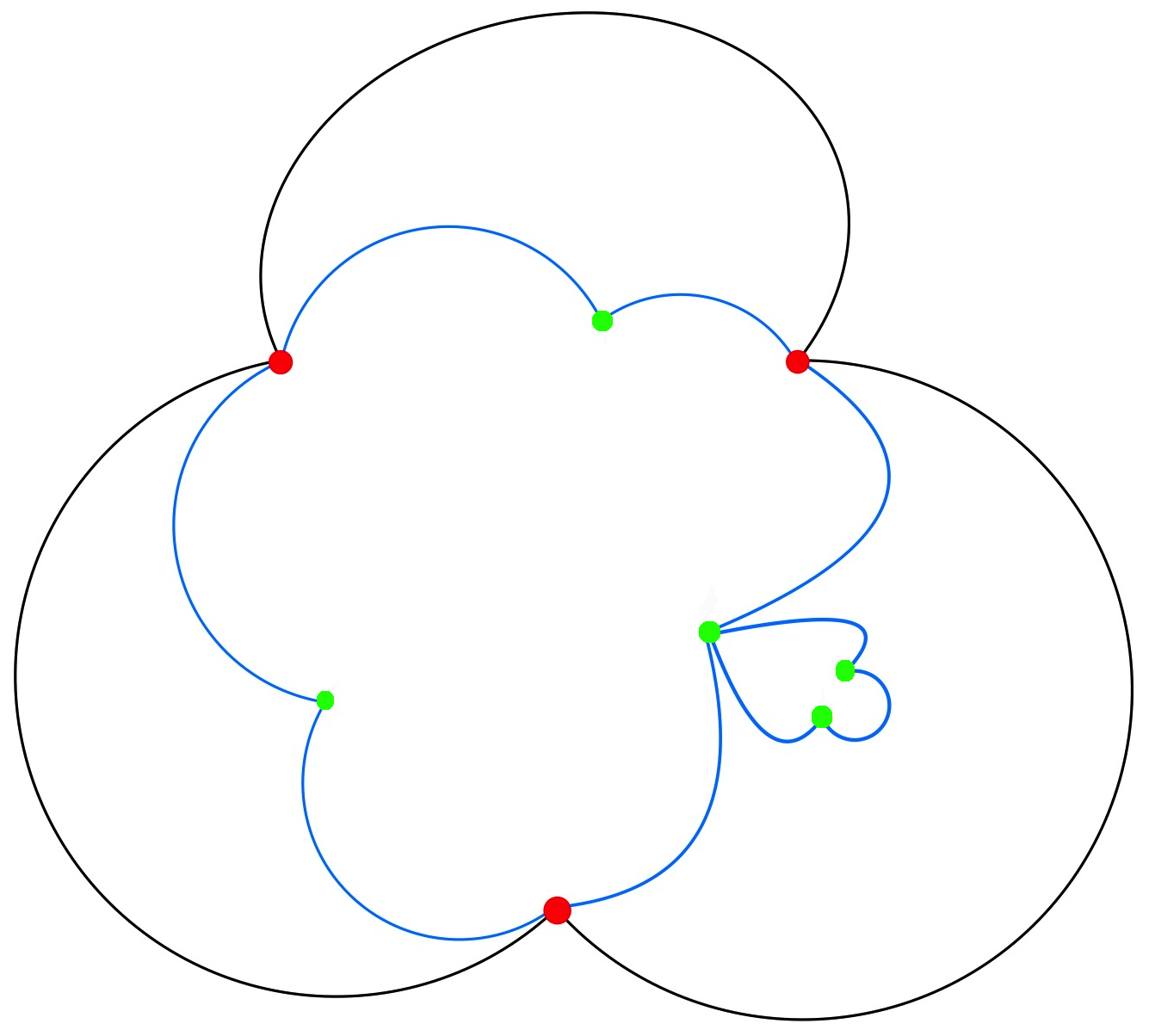}}; 
\node at (1.36,4.2) {\begin{Large}$P_2$\end{Large}};
\node at (2.7,1.4) {\begin{Large}$P_1$\end{Large}};
\end{tikzpicture}
\caption{Pictured is the domain and codomain of a pinched polynomial-like map. Here, $\Int{P_2}$ is the interior of the black polygon with three corners (marked in red). The interior of the blue pinched polygon is $\Int{P_1}$. The pinched point and the additional corner points of $P_1$ are marked in green.}
\label{pinched_poly_fig}
\end{figure}

As with polynomial-like maps, we define the \textit{filled Julia set} $K(f)$ (also called the \textit{non-escaping set}) of a pinched polynomial-like map to be $K(f):=\bigcap_{n\geq 0} f^{-n}(P_1)$. Similar to classical polynomial-like maps, the filled Julia set $K(f)$ of a pinched polynomial-like map is connected if and only if it contains all of the critical values of $f$.

\begin{definition}
1) Let $Z$ be a closed set in $\widehat{\C}$. A subset $\widehat{Z}$ of $\widehat{\C}$ is said to be a \emph{pinched neighbourhood} of $Z$, pinched at finitely many points $z_1,\cdots,z_n\in\partial Z$, if $Z\subset\widehat{Z}$, each point of $Z\setminus\{z_1,\cdots,z_n\}$ is an interior point of $\widehat{Z}$, but no $z_i$ is an interior point of $\widehat{Z}$. We will sometimes refer to $\widehat{Z}$ as an $n$-pinched neighbourhood of $Z$.
\smallskip

2) For a pinched neighborhood $N$ of $\mathbb{S}^1$, the intersection $N\cap\overline{\D}$ is called a \emph{one-sided pinched neighbourhood} of $\mathbb{S}^1$. 
\end{definition}

\begin{definition}\label{hybrid_def}
Let $(f_1,P_1, P_2)$ and $(f_2,Q_1,Q_2)$ be two pinched polynomial-like maps. We say that $f_1$ and $f_2$ are \textit{hybrid equivalent} if there exists a quasiconformal map $\Phi\colon \widehat{\C}\to \widehat{\C}$ such that:

\begin{enumerate}
	\item $\Phi$ sends the corners of $P_2$ onto the corners of $Q_2$,
	\item $\Phi$ conjugates $f_1$ to $f_2$ on a pinched neighbourhood of $K(f_1)$ pinched at the corners and pinched points of $P_1$,
	\item $\overline \partial \Phi\equiv 0$ almost everywhere on $K(f_1)$.
\end{enumerate}
\end{definition}

Every polynomial-like map is a pinched polynomial-like map. Moreover, every rational map $R\in\pmb{\mathcal{B}}_d$ admits a pinched polynomial-like restriction. Indeed, if $D\subsetneq \mathcal{A}(R)$ is an attracting petal containing a fully ramified critical value such that $P_2 = \widehat{\C}\setminus D$ is a polygon, then $(R|_{R^{-1}(P_2)}, P_1=R^{-1}(P_2), P_2)$ is a pinched polynomial-like map with filled Julia set $\mathcal{K}(R)$.

Hybrid equivalence is clearly an equivalence relation on the space of pinched polynomial-like maps. We refer to the corresponding equivalence classes as \emph{hybrid classes} of pinched polynomial-like maps.

\begin{figure}[h!]
\captionsetup{width=0.96\linewidth}
\begin{tikzpicture}
\node[anchor=south west,inner sep=0] at (0,0) {\includegraphics[width=0.99\textwidth]{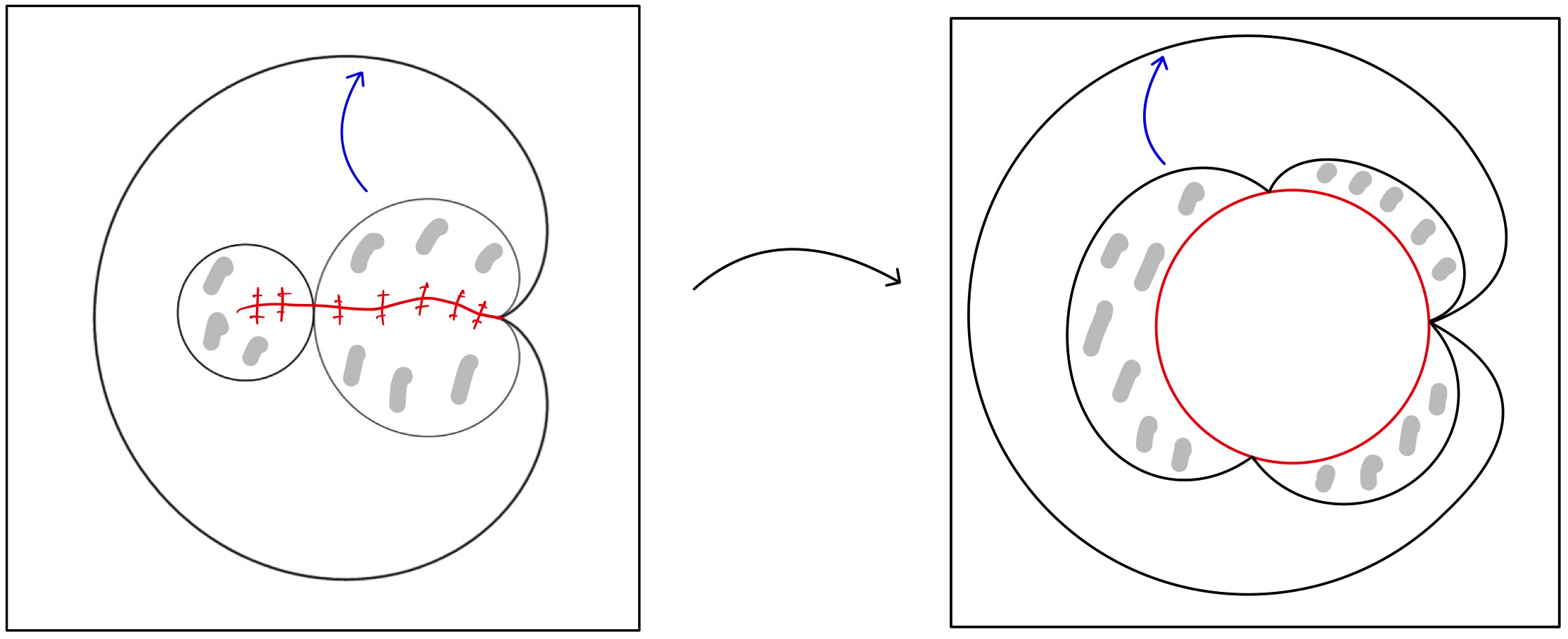}}; 
\node at (10.32,2.56) {$\overline{\D}$};
\node at (9,4.2) {$\cE$};
\node at (4,0.6) {$P_2$};
\node at (3,1.45) {$P_1$};
\node at (2.56,4) {$f$};
\node at (3.3,2.4) {\begin{tiny}$K(f)$\end{tiny}};
\node at (6.36,3.42) {$\psi_f$};
\end{tikzpicture}
\caption{The Riemann map $\psi_f:\widehat{\C}\setminus K(f)\to\widehat{\C}\setminus\overline{\D}$ conjugates the action of $f$ (outside $K(f)$) to a piecewise analytic map on the interior of} a one-sided pinched neighbourhood of $\mathbb{S}^1$. By the Schwarz reflection principle, this map can be extended to a piecewise analytic map $\cE$ on a pinched neighbourhood of $\mathbb{S}^1$ that restricts to a piecewise real-analytic circle covering.
\label{ext_class_fig}
\end{figure}

For a pinched polynomial-like map $(f,P_1, P_2)$ with connected filled Julia set, one can define the notion of an \emph{external map} following the classical construction of external maps of polynomial-like maps (or parabolic-like maps, see Figure~\ref{ext_class_fig}). However, unlike in the classical polynomial-like case, the external map of a pinched polynomial-like map is in general a \emph{piecewise} real-analytic circle covering that is not necessarily expanding. 

Thus, the dynamics of a pinched polynomial-like map also decomposes into internal and external dynamics, which are captured by its hybrid class and external map (respectively). 
However, in this generality, one cannot expect a straightening result for pinched polynomial-like maps, nor can one expect to be able to mate an arbitrary hybrid class with a piecewise real-analytic circle covering.

\smallskip

\textbf{A standing assumption:} In what follows, an \emph{external map} will stand for a piecewise real-analytic, expansive, degree $d\geq 2$ circle covering with the additional regularity condition that the two branches of the map at each break-point of its piecewise definition admit local analytic extensions. 

We remark that this class of external maps enlarges the class of parabolic external maps considered in \cite{LPS} (which are real-analytic expansive degree $d\geq 2$ circle coverings) to piecewise real-analytic circle maps.

\begin{definition}\label{ppl_ext_conj_def}
Let $\cE_1,\cE_2$ be two degree $d$ external maps. We say that they are \emph{piecewise analytically conjugate} if there exists a homeomorphism $\phi:\overline{\D}\to\overline{\D}$ such that
\begin{enumerate}
\item $\phi$ is quasiconformal on $\D$;
\item $\phi$ carries the break-points of the piecewise definition of $\cE_1$ onto those of $\cE_2$;
\item $\phi$ is conformal in the interior of a one-sided pinched neighbourhood of $\mathbb{S}^1$ pinched at the break-points of $\cE_1$; and
\item $\phi\circ \cE_1=\cE_2\circ\phi$ on $\mathbb{S}^1$.
\end{enumerate}
\end{definition}

\begin{remark}
One immediately observes that the last condition of Definition~\ref{ppl_ext_conj_def} implies that the maps $\cE_1$ and $\cE_2$ are conformally conjugate in a pinched neighbourhood of the circle.
\end{remark}

As in the classical situation, we equip a pinched polynomial-like map (with connected Julia set) with an invariant external access to its Julia set, and consider the space of degree $d$ pinched polynomial-like maps with connected Julia set up to M{\"o}bius conjugacy preserving the marking. An \emph{external fiber} in this space is the sub-collection of all maps with a given fixed external map.

\begin{lemma}\label{homeo_fibers_lem}
Let $\cE_1$ and $\cE_2$ be piecewise analytically conjugate external maps. Then the `straightening map' between the corresponding external fibers is a bijection. 
\end{lemma}
\begin{proof}[Sketch of proof]
The piecewise analytic conjugacy (which extends quasiconformally to $\D$) between $\cE_1$ and $\cE_2$ can be used to pass from a pinched polynomial-like map in the external fiber of $\cE_1$ to a map in the external fiber of $\cE_2$ with the same (marked) hybrid class, and vice versa. The fact that these `straightening maps' between the two fibers are inverses of each other follows from the fact that a (marked) pinched polynomial-like map is uniquely determined by its (marked) hybrid class and external map.
\end{proof}

Our goal in the next few sections is to extract appropriate external maps from the Hecke group $\mathcal{H}_{d+1}$ and to justify that these external maps can indeed be mated with `parabolic' hybrid classes given by $R\in\pmb{\mathcal{B}}_d$. The resulting pinched polynomial-like maps, whose internal dynamics are given by maps in $\pmb{\mathcal{B}}_d$ and whose external dynamics exhibit features of the Hecke group, are important for the construction of the desired algebraic correspondences.

\section{Hecke group and the associated external class}\label{hecke_ext_map_sec}

In this section, we introduce the Hecke group and induce two piecewise real-analytic, expansive circle coverings; namely the \emph{Farey map} and the \emph{Hecke map}. These maps are produced by two different instantiations of a common underlying mechanism. In fact, these two maps will turn out to be piecewise analytically conjugate; i.e., they are two different representatives of the same external class (see Section~\ref{pinched_poly_like_subsec}).
These maps will play a crucial role in the two different constructions of correspondences as matings of rational maps and the Hecke group.

\subsection{The Hecke group}\label{hecke_group_subsec}

We recall the definition of the Hecke group from Section~\ref{main_thm_subsec}, and explain how it arises as an index two Fuchsian subgroup of a triangle reflection group. 

Let $\Pi$ be the closed ideal polygon in $\D$ with vertices at the $(d+1)$-st roots of unity. Label the sides of $\Pi$ as $C_1,\cdots, C_{d+1}$, where $C_j$ connects $\omega^{j-1}$ to $\omega^j$, with $\omega:=\exp{(\frac{2\pi i}{d+1})}$. For $j\in\{1,\cdots,d+1\}$, the connected component of $\D\setminus \Pi$ having $C_j$ on its boundary is denoted by $\D_j$. Further, let $a_1,a_2$ be the radial line segments from $0$ to $1$, and $0$ to $\omega$, respectively. Finally, let $\beta$ be the perpendicular bisector of the Euclidean arc $C_1$ (see Figure~\ref{external_model_fig}). The geodesics $C_1, a_1$ and $\beta$ form a triangle $\Delta$ in $\D$ with vertex angles $\pi/(d+1), \pi/2$ and $0$. 
The anti-M{\"o}bius reflection maps $r_{a_1}, r_{C_1}$ and $r_\beta$ in the sides of $\Delta$ generate a triangle reflection group $\Delta(d+1,2,\infty)$. The Hecke group $\mathcal{H}_{d+1}$ is the Fuchsian subgroup of index two of $\Delta(d+1,2,\infty)$. The group $\cH_{d+1}$ is generated by the M{\"o}bius maps $\sigma:=r_{\beta}\circ r_{C_1}$ and $\rho:=r_\beta\circ r_{a_1}$. We note that $\sigma$ is a rotation by angle $\pi$ around the point of intersection of $\beta$ and $C_1$, and $\rho$ is a rotation by angle $2\pi/(d+1)$ around the origin (which is the point of intersection of $\beta$ and $a_1$). In particular, $\sigma,\rho$ are of order $2,d+1$, respectively. It is a standard fact that there is no other relation in the group $\mathcal{H}_{d+1}$; i.e.,
$$
\mathcal{H}_{d+1}=\langle \sigma,\rho:\sigma^2=\rho^{d+1}=1\rangle\cong \Z/2\Z\ast\Z/(d+1)\Z.
$$
(cf. \cite[\S 11.3]{Bea95}).

For future reference we recall the notation $\alpha_j=\sigma\circ\rho^j$, $j\in\{1,\cdots, d\}$, and note that
\begin{enumerate}
\item $\alpha_j$ carries $\D_{d+2-j}$ onto $\D\setminus\overline{\D_1}$ for each $j$ (indeed, $\rho^j$ maps $\D_{d+2-j}$ onto $\D_1$, and $\sigma$ maps $\D_1$ onto $\D\setminus\overline{\D_1}$);
\item $\sigma$ conjugates $\alpha_j$ to $\alpha_{d+1-j}^{-1}$ for each $j$;
\item $\alpha_1$ and $\alpha_d$ are parabolic with unique fixed points $1$ and $\omega$ respectively; 
\item $\alpha_j$, $j\in\{2,\cdots, d-1\}$, is hyperbolic with its axis intersecting $C_{d+2-j}$ and $C_1$ (at their Euclidean mid-points) orthogonally. 
\end{enumerate}

\begin{figure}[h!]
\captionsetup{width=0.96\linewidth}
\begin{tikzpicture}
\node[anchor=south west,inner sep=0] at (0,0) {\includegraphics[width=0.45\textwidth]{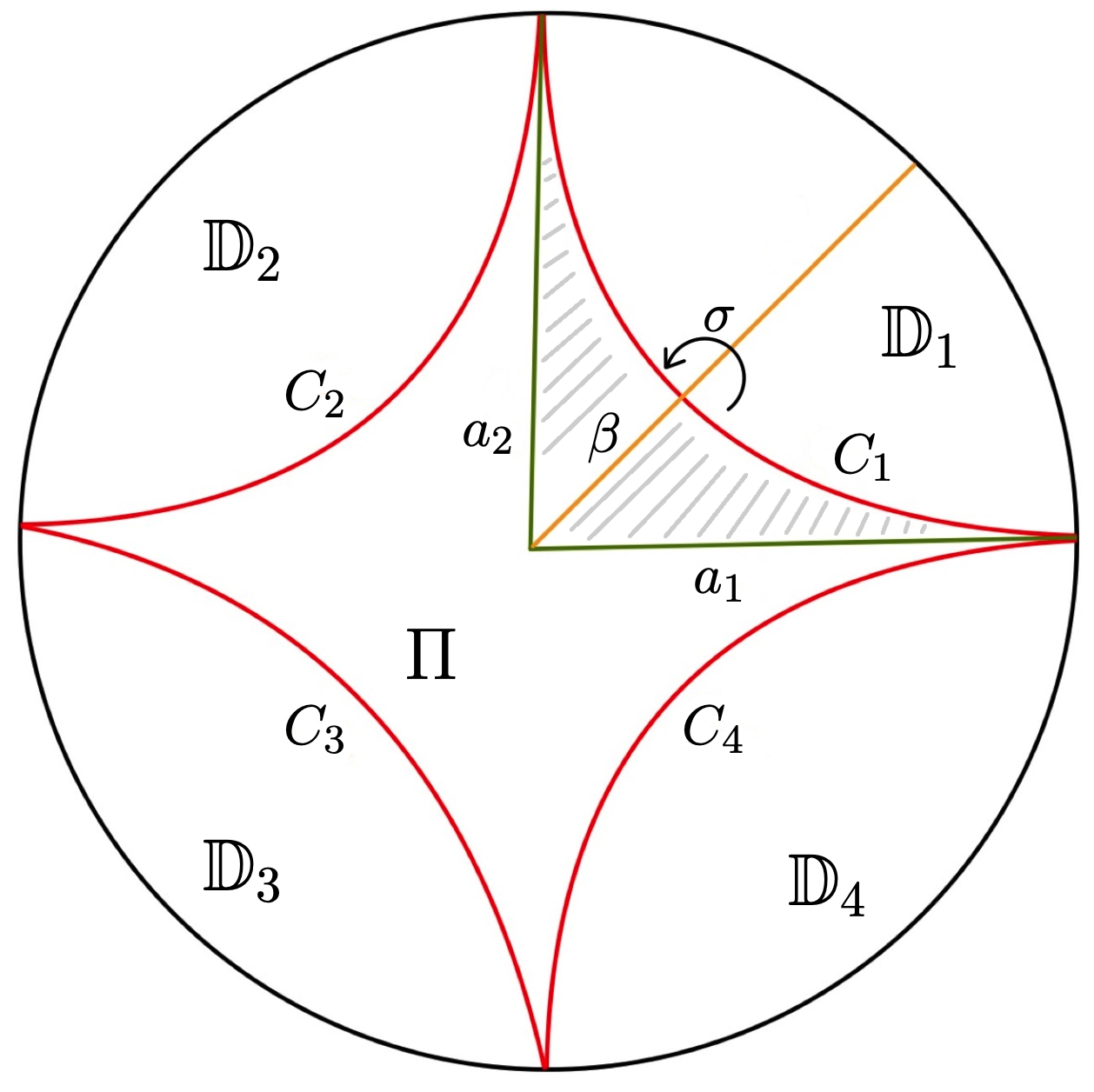}}; 
\node[anchor=south west,inner sep=0] at (6.4,0) {\includegraphics[width=0.5\textwidth]{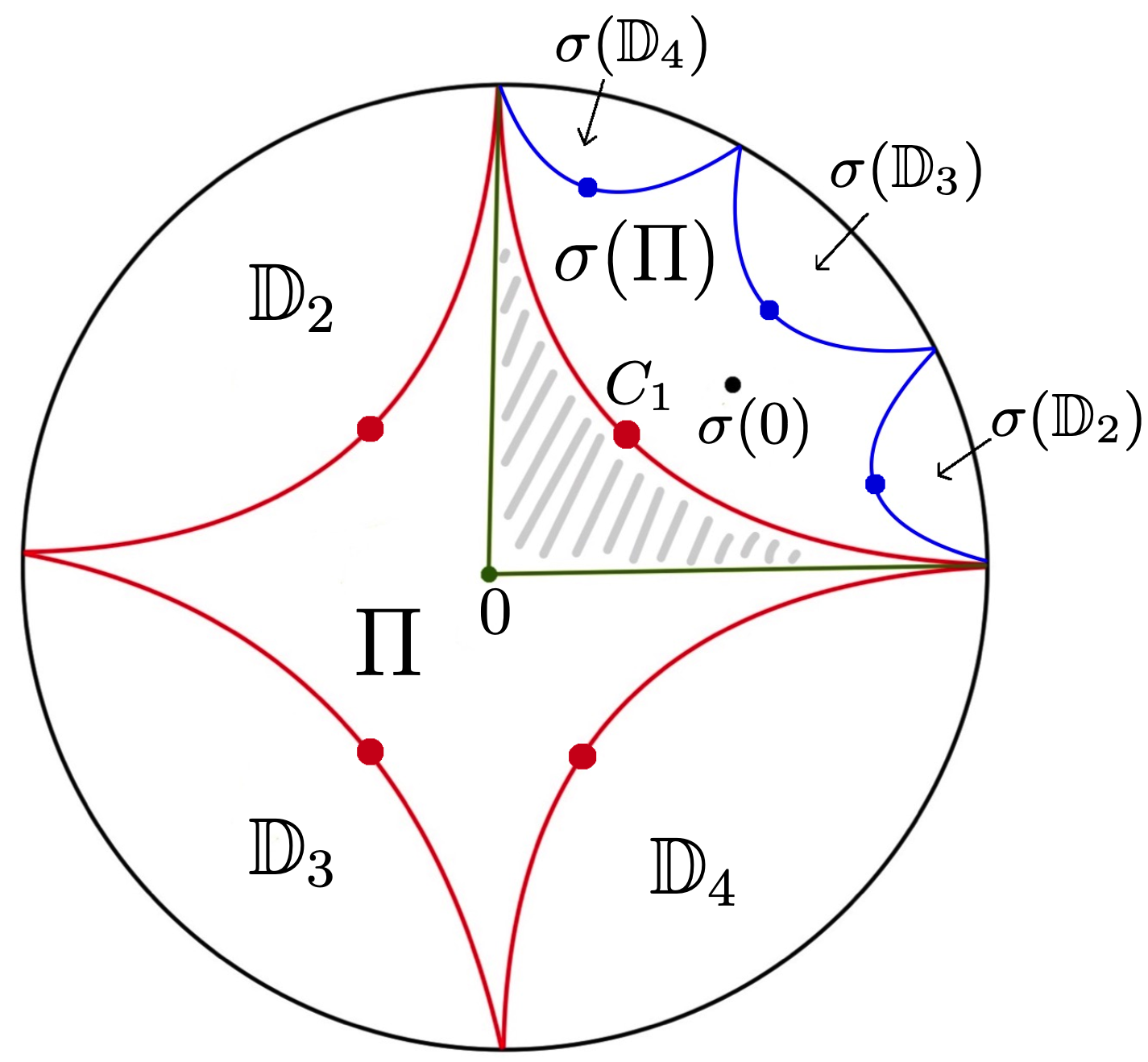}}; 
\end{tikzpicture}
\caption{The fundamental domain of the Hecke group $\mathcal{H}_4$ is shaded and the polygon $\Pi$ is depicted.}
\label{external_model_fig}
\end{figure}

A fundamental domain for the $\mathcal{H}_{d+1}$-action on $\D$ is given by the triangle bounded by $C_1,a_1$ and $a_2$. Moreover, $\D/\mathcal{H}_{d+1}$ is isomorphic to a sphere with one puncture, an orbifold point of order $2$ and an orbifold point of order $d+1$. Note that $\mathcal{H}_3$ is M{\"o}bius conjugate to the modular group $\mathrm{PSL}_2(\Z)$.

The axes of all the involutions (i.e., the invariant geodesics under the involutions) in $\cH_{d+1}$ produce a tessellation of $\D$ by ideal $(d+1)-$gons. Specifically, these polygons are the translates of $\Pi$ under the action of the Fuchsian group generated by $\rho^j\circ\sigma\circ\rho^{-j}$, $j\in\{0,\cdots,d\}$. Since the rotation $\rho$ acts transitively on the sides of $\Pi$, it follows that given any polygon $\Pi'$ of the above tessellation and given any two sides $\ell_1,\ell_2$ of $\Pi'$, there exists a unique group element in $\cH_{d+1}$ that carries $\ell_1$ to $\ell_2$ in an orientation-reversing fashion (here we put counter-clockwise orientation on the polygon boundary $\partial\Pi'$). 
This observation will play an important role in the definition of the Farey/Hecke maps.

\subsection{The Farey map}\label{farey_map_subsec}
In this subsection, we will construct a piecewise analytic circle covering, called the \emph{Farey map}, from the Hecke group.

\subsubsection{The discontinuous pre-Farey map $\widetilde{F}_d$}
We first define a piecewise M{\"o}bius map on an arc of $\mathbb{S}^1$.
The sides of $\sigma(\Pi)$ other than $C_1$ are mapped onto $C_1$ (as orientation-reversing maps) by the group elements 
$$
\beta_j:= \rho^j\circ\sigma,\ j\in\{1,\cdots, d\};
$$ 
indeed, $\beta_j$ maps $\sigma(C_{d+2-j})$ onto $\sigma(C_1)=C_1$.
Thus, the maps $\beta_j$ can be patched together to produce a piecewise M{\"o}bius self-map $\widetilde{F}_d$ of degree $d$ of the arc $[1,\omega]\subset\mathbb{S}^1$. The map $\widetilde{F}_d$ also extends naturally to a piecewise M{\"o}bius map
$$
\widetilde{F}_d:\bigcup_{j=2}^{d+1}\sigma(\overline{\D_j})\longrightarrow \overline{\D_1}.
$$
If we orient the boundary of the domain (respectively, the image) of $\widetilde{F}_d$ in such a way that the domain (respectively, the image) is on the right side of the boundary, then the map $\widetilde{F}_d$ is orientation-preserving. The map $\widetilde{F}_d$ has discontinuities at the break-points of its piecewise definition. To turn this (mildly) discontinuous piecewise M{\"o}bius map to actual circle coverings, we need to identify the points $1$ and $\omega$.

\subsubsection{Passing to a $(d+1)$-fold quotient to eliminate discontinuity}
To this end, we construct an orbifold quotient of the disc by the order $d+1$ rotation~$\rho$:
$$
\mathcal{Q}_1:=\faktor{\D}{\langle \rho\rangle}. 
$$ 
We also define the bordered orbifold $\overline{\mathcal{Q}_1}:=\faktor{\overline{\D}}{\langle \rho\rangle}$. The map $\widetilde{F}_d$ descends to a piecewise analytic covering map on the topological circle $\partial\mathcal{Q}_1$.
A (closed) fundamental domain $\mathfrak{F}_1$ for the action of $\langle \rho\rangle$ on $\overline{\D}$ is given by the closed sector (in $\overline{\D}$) formed by the radial lines at angles $0, \frac{2\pi}{d+1}$ (see top left of Figure~\ref{farey_fig})
Thus, $\overline{\mathcal{Q}_1}$ is biholomorphic to the (bordered) surface obtained from $\mathfrak{F}_1$ by identifying the radial line segments at angles $0, \frac{2\pi}{d+1}$ under $\rho$. This endows $\mathcal{Q}_1$ with a preferred choice of complex coordinates. 
We denote the quotient map from $\overline{\D}$ to $\overline{\mathcal{Q}_1}$ by $\pi_1$.
Note that the map $z\mapsto z^{d+1}$ of $\D$ induces a biholomorphism $\xi_1:\mathcal{Q}_1\to\D$, and yields a homeomorphism $\xi_1:\partial\mathcal{Q}_1\to\mathbb{S}^1$. We define the branched covering
$$
\theta_1:=\xi_1\circ\pi_1:\D\to\D,
$$
of degree $d+1$, and define the sets
$$
\mathfrak{H}_1:=\theta_1(\Pi),\quad \cD_1:=\overline{\D}\setminus\Int{\mathfrak{H}_1}. 
$$
The image of the radial lines at angles $0,\frac{2\pi}{d+1}$ under $\theta_1$ is the segment $[0,1)$.

\subsubsection{The Farey map $F_d$}
Let $\underline{\theta_1}^{-1}$ be the inverse branch of $\theta_1$ from $\overline{\D}\setminus[0,1]$ to $\mathfrak{F}_1$. 
With the above setup at our disposal, we now define the map
$$
F_d: \bigcup_{j=2}^{d+1} \theta_1(\sigma(\overline{\D_j}))\longrightarrow\overline{\D}
$$
as
$$
F_d\equiv \theta_1\circ \widetilde{F}_d\circ\underline{\theta_1}^{-1}= \theta_1\circ\beta_{d+2-j}\circ\underline{\theta_1}^{-1}\ \textrm{on}\  \theta_1(\sigma(\overline{\D'_j})).
$$
An important feature of the map $F_d$ is that it extends analytically to the larger region $\cD_1$ as 
$$
F_d:\cD_1\to\overline{\D},\quad F_d\equiv \theta_1\circ \sigma\circ\underline{\theta_1}^{-1}.
$$
This map $F_d$ is called the \emph{$d$-th Farey map}.

The map $F_d$ restricts to a piecewise analytic, orientation-preserving degree $d$ covering of $\mathbb{S}^1$ with a unique neutral fixed point (at $1$). It also enjoys the regularity property that the two branches of the map at each break-point of its piecewise definition admit local analytic extensions.
It is also easy to see that $F_d\vert_{\mathbb{S}^1}$ is expansive (for instance, by \cite[Lemma~3.7]{LMMN}). Hence, the map $F_d$ is an external map in the sense of Section~\ref{pinched_poly_like_subsec}.

We list some additional properties of the map $F_d$ (see Figure~\ref{farey_fig}).
\begin{enumerate}
\item All points in $\Int{\cD_1}$ eventually escape to $\mathfrak{H}_1$ under iterates of $F_d$.
\item The map $F_d$ has a critical point of multiplicity $d$ at $\theta_1(\sigma(0))$ with associated critical value $0$. 
\item The restriction $F_d:\partial\mathfrak{H}_1=\theta_1(C_1\cup\{1\})\to\partial\mathfrak{H}_1$ is an orientation-reversing involution.
\end{enumerate}

\begin{remark}
For an interpretation of the map $F_d$ as a factor of a certain \emph{Bowen-Series map}, see \cite{MM23}. 
\end{remark}

\subsubsection{Parabolic asymptotics of the Farey map}\label{para_asymp_subsec}

An explicit computation shows that the M{\"o}bius involution $\sigma$ is given by
$$
\sigma(z)= \frac{2z\omega -\omega(1+\omega)}{z(1+\omega)-2\omega},
$$
while the map $\rho$ is clearly given by $\rho(z)=\omega z$.
The Farey map $F_d$ does not admit an analytic continuation in a neighbourhood of the fixed point $1$. However, a straightforward computation using the maps $\sigma, \rho$ shows that the two branches of $F_d$ at $1$ can be extended analytically in a neighbourhood of $1$ as different parabolic germs. Specifically, these germs have the following power series expansion:
\begin{equation}
1+\zeta\mapsto 1+\zeta-ia\zeta^2+O(\zeta^3),\ \textrm{for}\ \zeta\approx 0\ \textrm{with}\ \im(\zeta)>0, 
\label{farey_asymp_top_eqn}
\end{equation}
and 
\begin{equation}
1+\zeta\mapsto 1+\zeta+ib\zeta^2+O(\zeta^3),\ \textrm{for}\ \zeta\approx 0\ \textrm{with}\ \im(\zeta)<0.
\label{farey_asymp_bottom_eqn}
\end{equation}
Here $a$ and $b$ are positive constants. In particular, $F_d$ has repelling directions along the imaginary axis at~$1$.

We remark, for future reference, that the change of coordinate $\zeta\mapsto -\frac{i}{a(\zeta-1)}$ (respectively, $\zeta\mapsto \frac{i}{b(\zeta-1)}$) carries the positive (respectively, negative) imaginary axis at $1$ to the negative real axis near $\infty$, and conjugates the asymptotics \eqref{farey_asymp_top_eqn} (respectively, \eqref{farey_asymp_bottom_eqn}) to maps of the form $z\mapsto z+1+O(1/z)$ near $\infty$.

\subsection{The Hecke map}\label{hecke_map_subsec}
We will now associate another piecewise analytic circle covering, called the \emph{Hecke map}, with the Hecke group.

\subsubsection{The discontinuous pre-Hecke map}
The sides of $\Pi$ other than $C_1$ are mapped onto $C_1$ by the group elements 
$$
\alpha_j= \sigma\circ \rho^j,\ j\in\{1,\cdots, d\};
$$ 
indeed, $\alpha_j$ maps $C_{d+2-j}$ onto $C_1$. 
Thus, the maps $\alpha_j$ can be patched together to produce a piecewise M{\"o}bius self-map $\widetilde{H}_d$ of degree $d$ of the arc $[\omega,1]\subset\mathbb{S}^1$. Note that $\widetilde{H}_d$ extends naturally to the region $\bigcup_{j=2}^{d+1}\overline{\D_j}$, and gives a piecewise M{\"o}bius map
$$
\widetilde{H}_d: \bigcup_{j=2}^{d+1}\overline{\D_j}\longrightarrow \overline{\D\setminus\D_1}.
$$
If we orient the boundary of the domain (respective, the image) of $\widetilde{H}_d$ in such a way that the domain (respectively, the image) is on the right side of the boundary, then the map $\widetilde{H}_d$ is orientation-preserving.
We remark that $\widetilde{H}_d$ is discontinuous at the break-points of its piecewise definition. We now pass to a quotient of $\D$, identifying the points $1$ and $\omega$ in the process, such that $\widetilde{F}_d$ descends to a piecewise analytic circle covering on the quotient.

\subsubsection{Passing to a two-fold quotient to eliminate discontinuity}

We define the (bordered) orbifolds
$$
\mathcal{Q}_2:=\faktor{\D}{\langle\sigma\rangle}, \quad \overline{\mathcal{Q}_2}:=\faktor{\overline{\D}}{\langle \sigma\rangle}.
$$
The map $\widetilde{H}_d$ descends to a piecewise analytic covering map on the topological circle $\partial\mathcal{Q}_2$.
To visualize the quotient orbifold $\mathcal{Q}_2$, and to derive asymptotics of the map induced by $\widetilde{H}_d$ on this orbifold, it will be convenient to work with a M{\"o}bius conjugate copy of $\cH_{d+1}$ such that the fixed point of the order two rotation $\sigma$ is placed at the origin. To this end, let $M$ be a M{\"o}bius automorphism of the disc that carries $C_1$ to the vertical geodesic connecting $\pm i$ and sends the fixed point of $\sigma$ on $C_1$ to the origin. After possibly  post-composing $M$ with $z\mapsto -z$, we can assume that $\Pi':=M(\Pi)$ is contained in the left half-disc (see Figure~\ref{farey_fig}). We will use the group $M\circ \cH_{d+1}\circ M^{-1}$ with generators $\rho^M:=M\circ\rho\circ M^{-1}, \sigma^M:= M\circ\sigma\circ M^{-1}$ to put complex coordinates on $\mathcal{Q}_2$. We also set $\D'_j:=M(\D_j)$ and $C'_j:=M(C_j)$, $j\in\{1,\cdots,d+1\}$.

A (closed) fundamental domain $\mathfrak{F}_2$ for the action of $\langle \sigma^M\rangle$ on $\overline{\D}$ is given by the closure (in $\overline{\D}$) of the left half-disc (see top right of Figure~\ref{farey_fig}).
Evidently, $\overline{\mathcal{Q}_2}$ is biholomorphic to the (bordered) surface obtained from $\mathfrak{F}_2$ by identifying the vertical geodesic with itself under $\sigma^M$. This endows $\mathcal{Q}_2$ with a preferred choice of complex coordinates. 
The quotient map from $\overline{\D}$ to $\overline{\mathcal{Q}_2}$ is denoted by $\pi_2$.
\begin{figure}[h!]
\captionsetup{width=0.96\linewidth}
\begin{tikzpicture}
\node[anchor=south west,inner sep=0] at (0,0) {\includegraphics[width=0.41\textwidth]{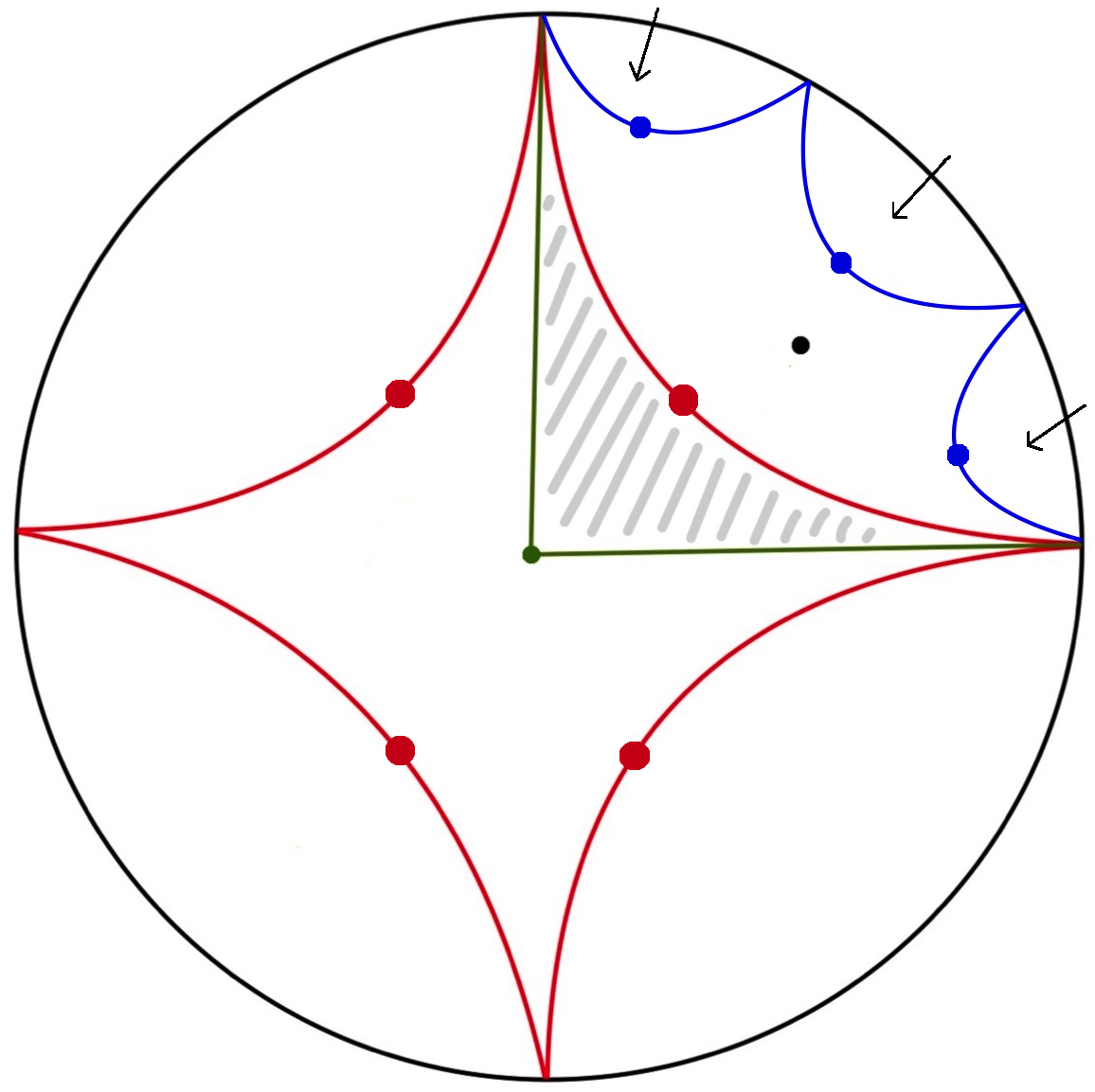}}; 
\node[anchor=south west,inner sep=0] at (6,0) {\includegraphics[width=0.41\textwidth]{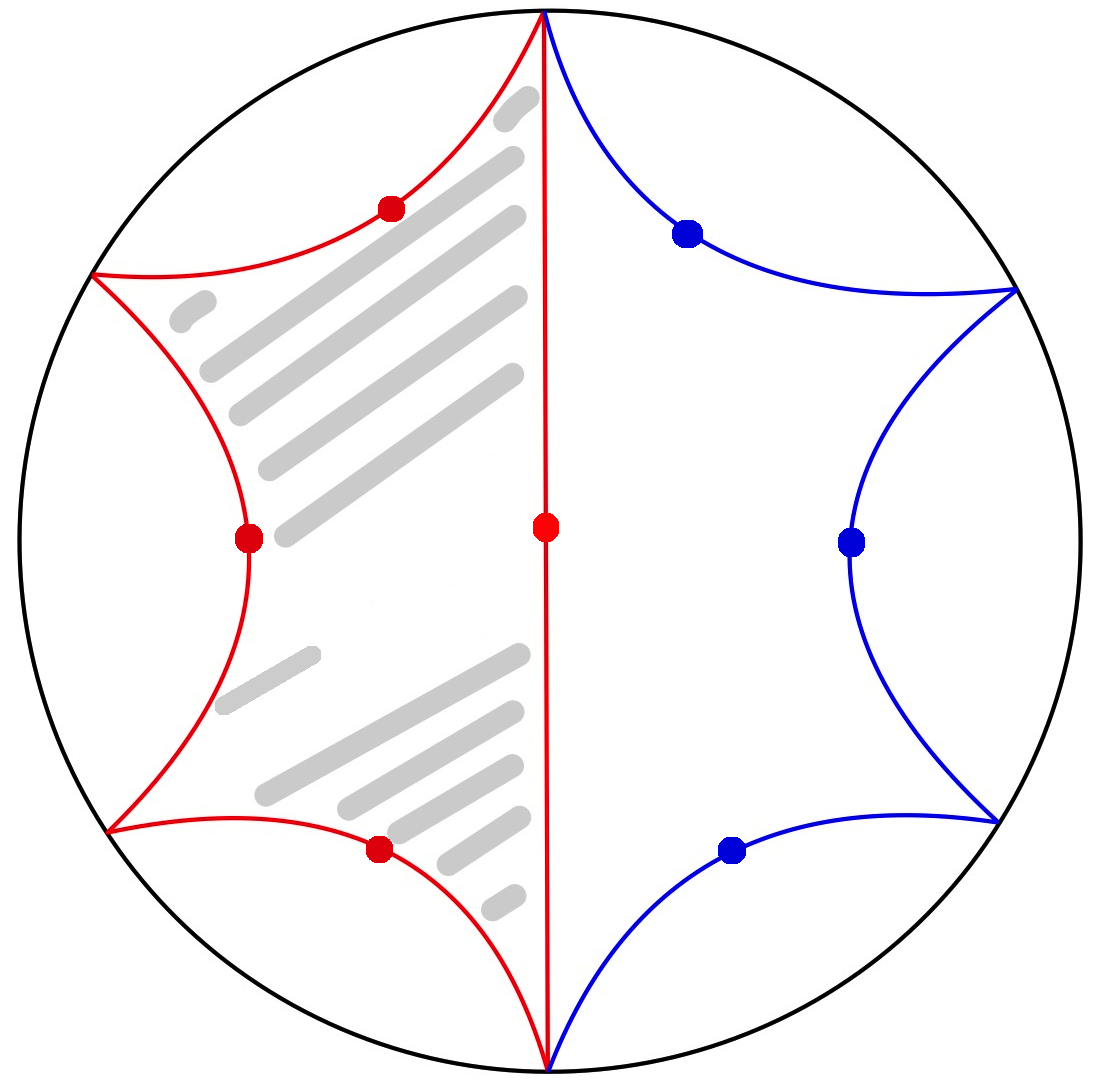}};
\node[anchor=south west,inner sep=0] at (0,-5.5) {\includegraphics[width=0.41\textwidth]{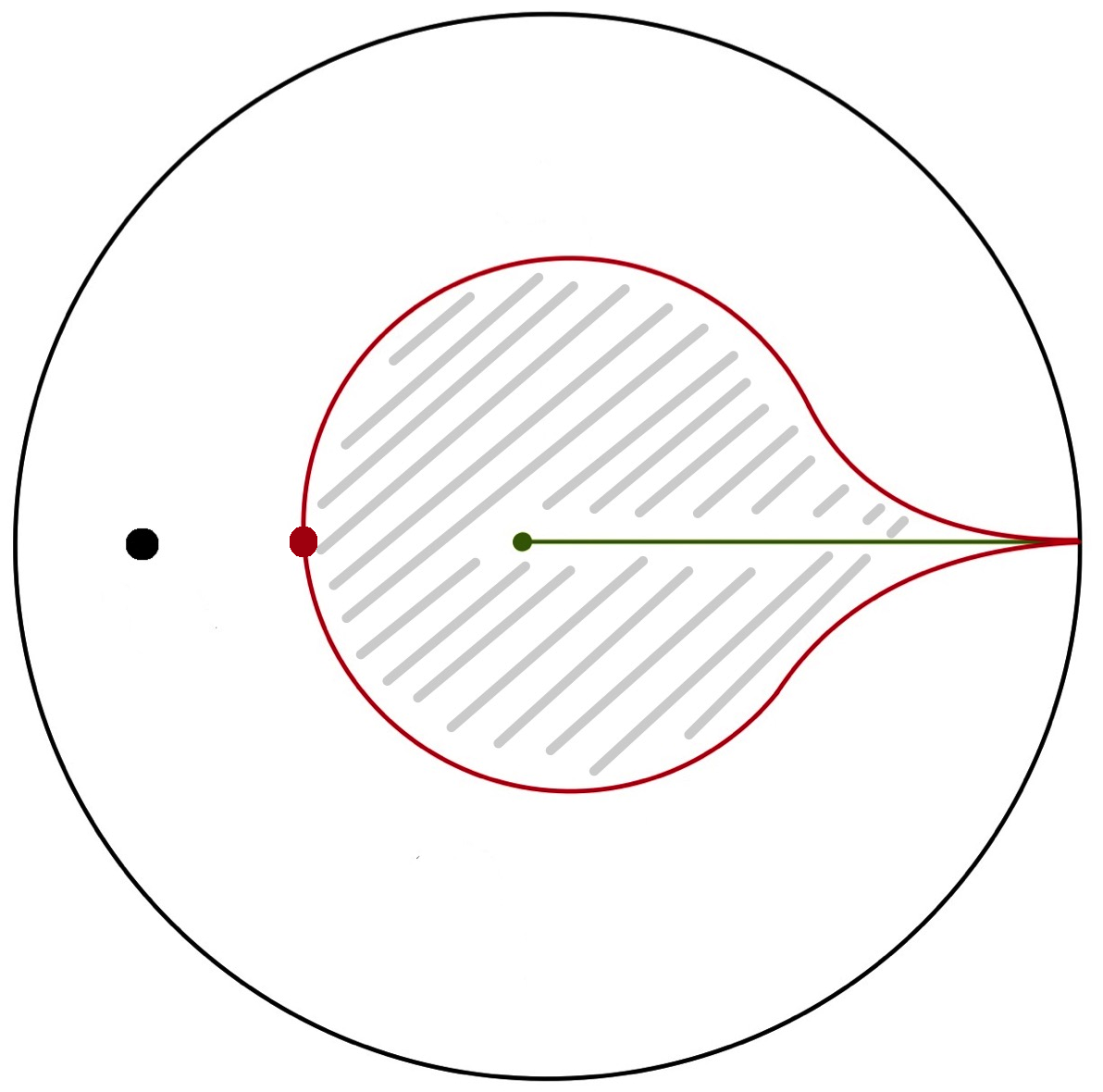}}; 
\node[anchor=south west,inner sep=0] at (6,-5.5) {\includegraphics[width=0.41\textwidth]{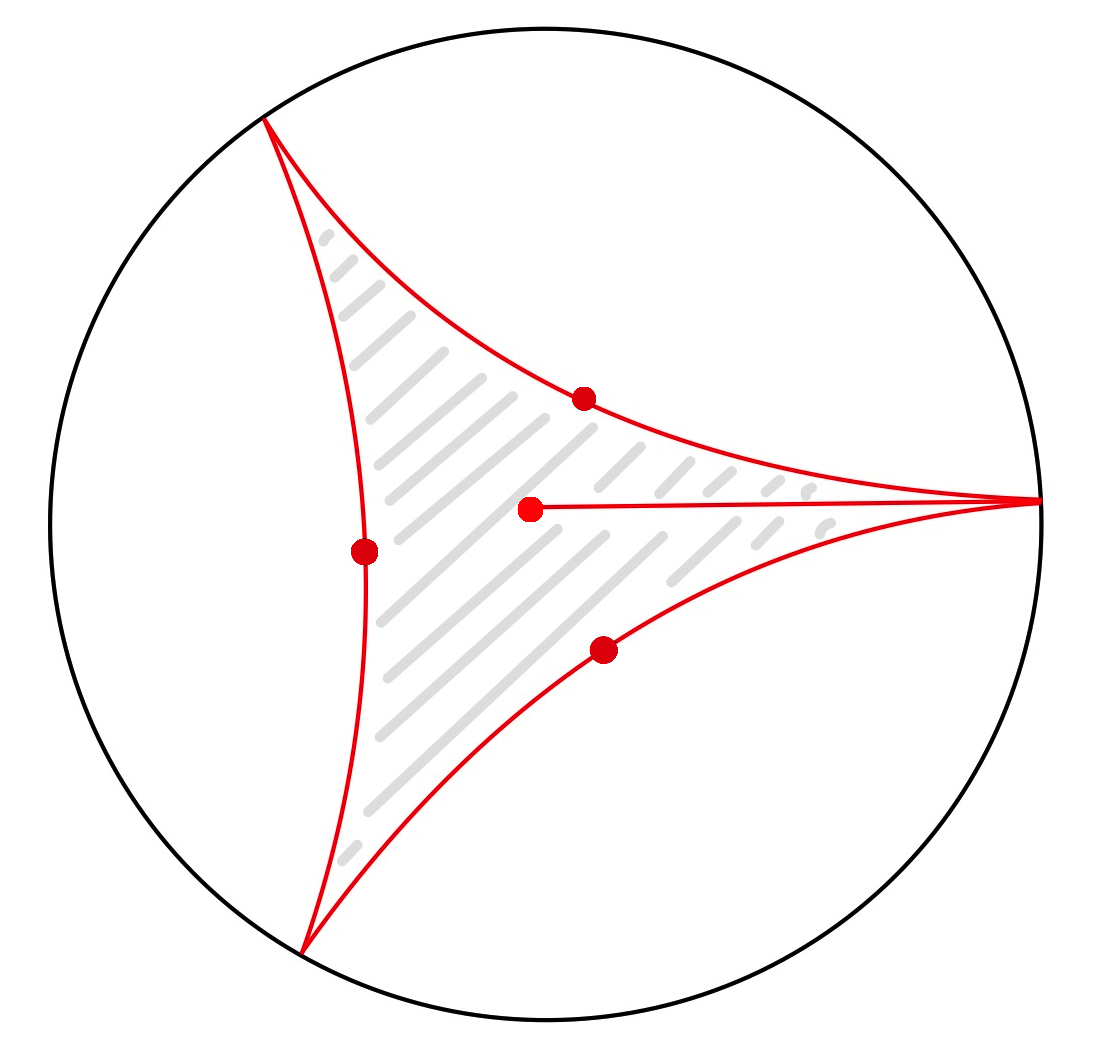}}; 
\node at (3.3,4.2) {\begin{large}$\sigma(\Pi)$\end{large}};
\node at (1.5,4) {\begin{large}$\D_2$\end{large}};
\node at (1.5,1.1) {\begin{large}$\D_3$\end{large}};
\node at (3.6,1) {\begin{large}$\D_4$\end{large}};
\node at (2.56,2.36) {\begin{small}$0$\end{small}};
\node at (2,2.2) {\begin{Large}$\Pi$\end{Large}};
\node at (3.92,3.32) {\begin{footnotesize}$\sigma(0)$\end{footnotesize}};
\node at (3.25,3.6) {\begin{footnotesize}$C_1$\end{footnotesize}};
\node at (3.28,5.32) {\begin{footnotesize}$\sigma(\D_4)$\end{footnotesize}};
\node at (4.72,4.66) {\begin{footnotesize}$\sigma(\D_3)$\end{footnotesize}};
\node at (5.56,3.36) {\begin{footnotesize}$\sigma(\D_2)$\end{footnotesize}};
\node at (9.28,2) {\begin{small}$\sigma^M(\Pi')$\end{small}};
\node at (7.5,4.36) {$\D_2'$};
\node at (6.6,2.62) {$\D_3'$};
\node at (7.5,0.72) {$\D_4'$};
\node at (8.8,2.56) {\begin{small}$0$\end{small}};
\node at (8.88,3.32) {\begin{small}$C'_1$\end{small}};
\node at (7.88,2.36) {\begin{small}$\Pi'$\end{small}};
\node at (1.25,-1.6) {\begin{tiny}$\theta_1(C_1)$\end{tiny}};
\node at (2.4,-2.45) {\begin{small}$\mathfrak{H}_1$\end{small}};
\node at (2.32,-4.9) {\begin{large}$\cD_1$\end{large}};
\node at (0.8,-3.3) {\begin{small}$\theta_1(\sigma(0))$\end{small}};
\draw [->,line width=0.5pt] (1.25,-1.75) to (1.6,-2.1);
\node at (2.4,-3.1) {\begin{small}$0$\end{small}};
\node at (8.48,-3.1) {\begin{small}$0$\end{small}};
\node at (8.12,-2.4) {$\mathfrak{H}_2$};
\node at (9,-1.36) {$\theta_2(\D'_2)$};
\node at (7,-3.2) {$\theta_2(\D'_3)$};
\node at (9.2,-4.6) {$\theta_2(\D'_4)$};
\node at (0.1,-0.2) {$\theta_1$};
\node at (11.1,-0.2) {$\theta_2$};
\draw [->, line width=0.6pt] (0.32,1.1) to (0.32,-1.28);
\draw [->, line width=0.6pt] (10.8,1.1) to (10.8,-1.44);
\draw [->, line width=0.6pt] (5.16,-2.5) to (6.12,-2.5);
\node at (5.64,-2.2) {\begin{large}$\mathfrak{p}$\end{large}};
\node at (5.66,2.12) {$M$};
\draw [->, line width=0.6pt] (5.16,1.8) to (6.1,1.8);
\node at (9.6,0.72) {\begin{tiny}$\sigma^M(\D'_2)$\end{tiny}};
\node at (9.6,4.36) {\begin{tiny}$\sigma^M(\D'_4)$\end{tiny}};
\node at (10.55,2.66) {\begin{tiny}$\sigma^M(\D'_3)$\end{tiny}};
\end{tikzpicture}
\caption{The red quadrilateral in the top left figure is mapped by the M{\"o}bius automorphism $M$ of the disc to the red quadrilateral in the top right figure, such that the geodesic $C_1$ is mapped to the vertical line segment. Left bottom: The map $F_d=\mathfrak{p}_2\circ\mathfrak{p}_1$ is defined on the closed set $\cD_1=\overline{\D}\setminus\Int{\mathfrak{H}_1}$ bounded by the unit circle (in black) and the monogon $\theta_1(C_1)$ (in red). Right bottom:  The map $H_d=\mathfrak{p}_1\circ\mathfrak{p}_2$ is defined on $\cD_2=\overline{\D}\setminus\Int{\mathfrak{H}_2}=\bigcup_{j=2}^4\theta_2(\overline{\D'_j})$.}
\label{farey_fig}
\end{figure}
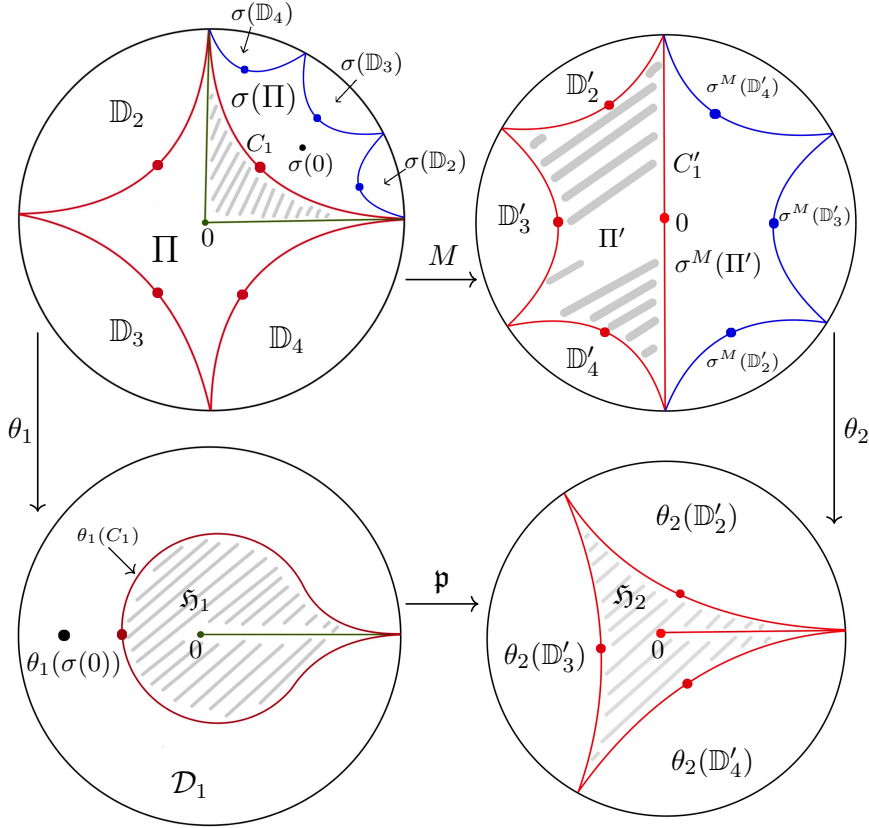
The map $z\mapsto -z^2$ of $\D$ induces a biholomorphism $\xi_2:\mathcal{Q}_2\to\D$, and yields a homeomorphism $\xi_2:\partial\mathcal{Q}_2\to\mathbb{S}^1$. We define the quadratic branched covering
$$
\theta_2:=\xi_2\circ\pi_2:\D\to\D,
$$
and introduce the sets
$$
\mathfrak{H}_2:=\theta_2(\Pi'),\quad \cD_2:=\overline{\D}\setminus\Int{\mathfrak{H}_2}. 
$$
Note that the image of the vertical geodesic under $\theta_2$ is the segment $[0,1)$.

\subsubsection{The Hecke map $H_d$}
Let $\underline{\theta_2}^{-1}$ be the inverse branch of $\theta_2$ from $\overline{\D}\setminus[0,1]$ to $\mathfrak{F}_2$. 
Let
$$
H_d: \cD_2=\bigcup_{j=2}^{d+1} \theta_2(\overline{\D'_j})\longrightarrow\overline{\D}
$$
be defined as
$$
H_d\equiv \theta_2\circ M\circ \widetilde{H}_d\circ M^{-1}\circ\underline{\theta_2}^{-1}= \theta_2\circ\alpha_{d+2-j}^M\circ\underline{\theta_2}^{-1}\ \textrm{on}\  \theta_2(\overline{\D'_j}).
$$
The map $H_d:\cD_2\to\overline{\D}$ will be called the \emph{$d$-th Hecke map}.

The map $H_d$ restricts to a piecewise analytic, expansive, orientation-preserving degree $d$ covering of $\mathbb{S}^1$ with a unique neutral fixed point (at $1$). Further, the two branches of $H_d$ at each break-point of its piecewise definition admit local analytic extensions.
Consequently, the map $H_d$ is an external map in the sense of Section~\ref{pinched_poly_like_subsec}.

Some additional properties of $H_d$ are listed below (see Figure~\ref{farey_fig}).
\begin{enumerate}
\item All points in $\Int{\cD_2}$ eventually escape to $\mathfrak{H}_2$ under iterates of $H_d$.
\item The map $H_d$ maps each $\theta_2(C'_j)$ two-to-one onto the segment $[0,1)$, $j\in\{2,\cdots,d+1\}$.
\end{enumerate}

The arguments of Section~\ref{para_asymp_subsec} apply to the Hecke map $H_d$ as well, yielding similar parabolic asymptotics for $H_d$ at the neutral fixed point $1$. 
This allows one to construct for the map $H_d$ the so-called \textit{dividing arcs}: two smooth forward invariant arcs $\gamma_+^{\mathcal F}, \gamma_-^{\mathcal F}$ emerging from the parabolic fixed point tangentially to the unit circle (to construct these it is sufficient to take partial horocycles, see for example \cite{BL3} or alternatively preimages of straight lines under repelling Fatou coordinates, see for example \cite{L1} and \cite{LPS}).

\subsection{The Farey and Hecke maps define the same external class}\label{same_ext_class_subsec}

Clearly, the elements $\alpha_j$ and $\beta_j$ are conjugate by the involution $\sigma$, $j\in\{1,\cdots, d\}$. Thus, the maps $\widetilde{H}_d$ and $\widetilde{F}_d$ are also conjugate by $\sigma$.
Hence, the map
$$
\mathfrak{p}:=\theta_2\circ M\circ\underline{\theta_1}^{-1}:F_d^{-1}(\cD_1)=\bigcup_{j=2}^{d+1} \theta_1(\sigma(\overline{\D_j}))\longrightarrow\cD_2=\bigcup_{j=2}^{d+1} \theta_2(\overline{\D'_j})
$$
is a piecewise conformal conjugacy between $F_d$ and $H_d$. Moreover, it follows from the definition of $\mathfrak{p}$ that it is asymptotic to a power map near $1$, and admits conformal extensions near the other pinched points of its domain. It follows that $\mathfrak{p}$ can be extended quasiconformally to $\D$. In particular, $\mathfrak{p}$ is a quasisymmetric conjugacy between the external maps $F_d$ and $H_d$ on the~circle.
Therefore, $F_d$ and $H_d$ are piecewise analytically conjugate external maps in the sense of Definition~\ref{ppl_ext_conj_def}.

\section{Mating parabolic rational maps with Farey and Hecke maps}\label{conf_mating_sec}

The goal of this section is to construct pinched polynomial-like maps as matings of hybrid classes coming from rational maps $R\in\pmb{\mathcal{B}}_d$ with the Hecke and Farey maps. The construction of such pinched polynomial-like maps only depends on certain qualitative properties of the Farey and Hecke maps, and not on the specific maps themselves. Hence, we carry out these mating constructions in a unified setting. To this end, we will introduce a family of external maps that contain the Farey and Hecke maps, possibly up to restriction to a smaller domain.
These maps, which are termed \emph{Farey-like maps}, will be shown to admit matings with the above hybrid classes.

Recall that all rational maps in $\pmb{\mathcal{B}}_d$ have the parabolic Blaschke product $B_d$ as their external map. To construct the desired pinched polynomial-like maps, we need to show that Farey-like maps are quasiconformally conjugate to the external map $B_d$ in a pinched neighbourhood of the circle. 
Quasiconformal compatibility results of this flavor (between parabolic maps) was first established by the second author \cite{L1}, and similar techniques were extended to other settings in \cite{LLMM3,BL1,LMM23}.

\subsection{Quasiconformal compatibility of Farey-like maps and the parabolic Blaschke product $B_d$}\label{qc_mating_subsec}

\subsubsection{Farey-like maps}\label{farey_like_subsubsec}

 Recall that all external maps are assumed to satisfy the standing hypothesis of Section~\ref{pinched_poly_like_subsec}.

\begin{definition}\label{farey_like_def}
Let $\cE$ be an external map such that $\cE(1)=1$, and $\cE$ is analytic on $\mathbb{S}^1\setminus \cE^{-1}(1)$.
We say that $\cE$ is \emph{Farey-like} if it satisfies the following properties (see Figure~\ref{farey_like_fig} (left)).
\begin{enumerate}
\item\label{dom_range_cond} There exist one-sided, pinched, closed neighbourhoods $X_1, X_2$ of $\mathbb{S}^1$ in $\overline{\D}$ such that $X_2$ is pinched only at $1$, $X_1$ is pinched only at the $d$ points of $\cE^{-1}(1)$, $X_1\subset X_2$, and $\partial X_1\cap\partial X_2=\mathbb{S}^1$; 

\item\label{extension_cond} $\cE$ extends to a continuous map $\cE:X_1\to~X_2$ such that $\cE:\partial X_1\to\partial X_2$ is a $d-$fold covering and each component of $\Int{X_1}$ maps conformally onto $\Int{X_2}$; and

\item\label{para_cond} the branches of $\cE$ at $1$ extend locally as simple parabolic germs.
\end{enumerate}
\end{definition}
\begin{figure}[h!]
\captionsetup{width=0.96\linewidth}
\begin{tikzpicture}
\node[anchor=south west,inner sep=0] at (0,0) {\includegraphics[width=0.48\textwidth]{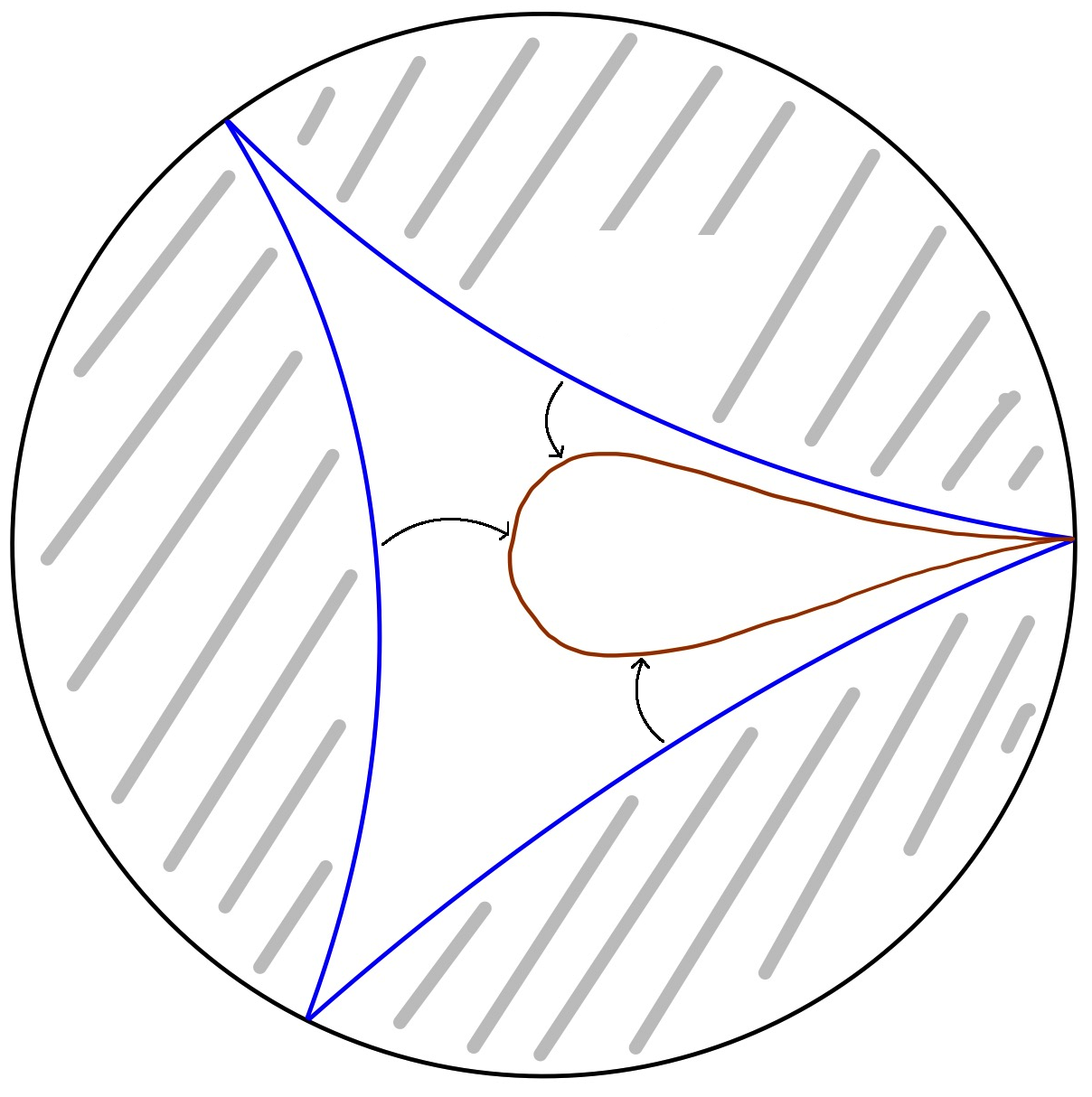}}; 
\node[anchor=south west,inner sep=0] at (6.4,0) {\includegraphics[width=0.48\textwidth]{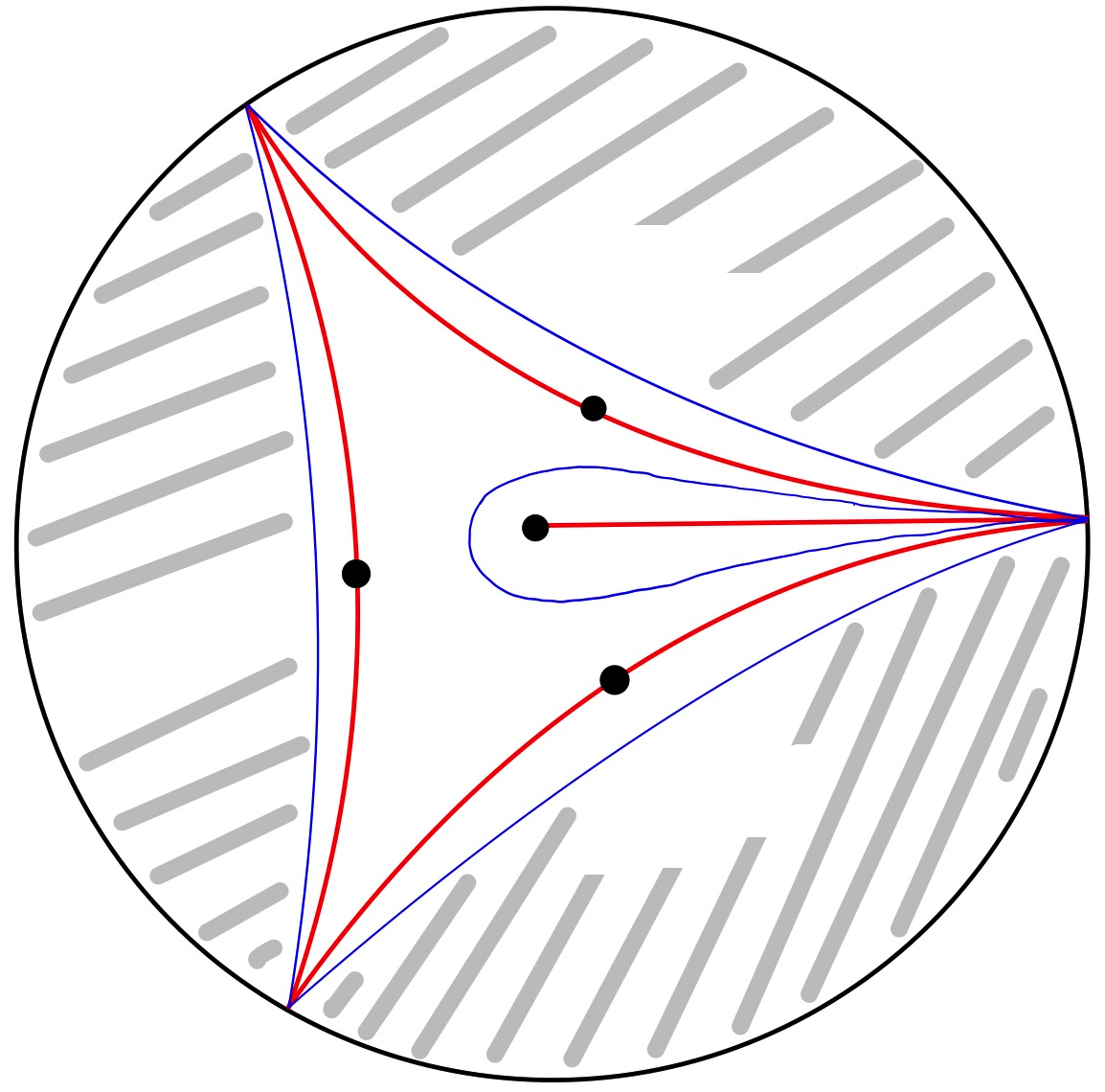}}; 
\node at (2.78,2.48) {$\gamma$};
\node at (4.8,0.4) {$\mathbb{S}^1$};
\node at (3.54,4.28) {\begin{footnotesize}$\cE^{-1}(\gamma)$\end{footnotesize}};
\node at (9.4,2.88) {\begin{footnotesize}$0$\end{footnotesize}};
\node at (8.84,2.88) {$\gamma$};
\node at (12.64,3.08) {$1$};
\node at (9.84,4.4) {\begin{footnotesize}$H_3^{-1}(\gamma)$\end{footnotesize}};
\node at (10.18,1.64) {\begin{footnotesize}$H_3^{-1}(\gamma)$\end{footnotesize}};
\node at (7.56,2.58) {\begin{footnotesize}$H_3^{-1}(\gamma)$\end{footnotesize}};
\end{tikzpicture}
\caption{Left: The domain and image of a Farey-like map is displayed. There is a simple closed curve $\gamma\subset\overline{\D}$ (in brown) that intersects $\mathbb{S}^1$ only at $1$ (which we call a \emph{monogon}) such that $X_2$ is the closed region bounded by the monogon $\gamma$ and the unit circle $\mathbb{S}^1$}. The closed region bounded by the blue curve $\cE^{-1}(\gamma)$ and $\mathbb{S}^1$ is $X_1$. Right: A Farey-like restriction of the Hecke map $H_3$ is illustrated. The curve $\gamma$, which meets $\mathbb{S}^1$ only at $1$, surrounds the arc $[0,1]$. The preimage of $\gamma$ under $H_3$, along with the circle $\mathbb{S}^1$, bounds the domain of definition of the Farey-like restriction of~$H_3$.
\label{farey_like_fig}
\end{figure}

\begin{remark}
i) The set $\gamma :=(\partial X_2\cap\D)\cup\{1\}$ is a monogon; i.e., a simple closed curve in $\overline{\D}$ that intersects $\mathbb{S}^1$ only at $1$. The set $X_2$ is bounded by $\gamma$ and~$\mathbb{S}^1$. The set $X_1$ is bounded by $\cE^{-1}(\gamma)$ and $\mathbb{S}^1$. 

ii) Since $\cE$ preserves $\mathbb{S}^1$ and is expansive on the circle, the parabolic germ obtained as the extension of the top (respectively, bottom) branch of $\cE$ at $1$ is tangent-to-identity and has the positive (respectively, negative) imaginary axis at $1$ as its repelling direction.

iii) After possibly shrinking the domain and the range, we may assume that $\partial X_1\cap\D$ and $\partial X_2\cap\D$ are smooth curves.
\end{remark}

By the Schwarz reflection principle, a Farey-like map $\cE:X_1\to X_2$ extends to a covering map $\cE:\widehat{X_1}:=X_1\cup\iota(X_1)\to\widehat{X_2}:=X_2\cup\iota(X_2)$ (where $\iota(w)=1/\overline{w}$) that maps each of the $d$ components $U_1,\cdots, U_d$ of $\Int{\widehat{X_1}}$ conformally onto $\Int{\widehat{X_2}}$. Here, we enumerate the components of $\Int{\widehat{X_1}}$ counter-clockwise such that $1\in\partial U_1$ and $U_1$ is contained in the upper half-plane.
Due to the touching structure of $\widehat{X_1}$ and $\widehat{X_2}$, the maps $\cE\vert_{U_j}$, $j\in\{2,\cdots,d-1\}$, are uniformly expanding, while $\cE\vert_{U_j}$, $j\in\{1,d\}$, are uniformly expanding away from $1$ with respect to the hyperbolic metric on $\Int{\widehat{X_2}}$.
This fact, along with the parabolic asymptotics of $\cE$ at $1$, implies that the diameters of the components of $\cE^{-n}(\Int{\widehat{X_1}})$ (which are topological discs) shrink to $0$ as $n\to+\infty$. It follows that the non-escaping set of $\cE:\widehat{X_1}\to\widehat{X_2}$ is precisely the unit circle $\mathbb{S}^1$, and hence every $z\in X_1\cap\D$ eventually maps to $X_2\setminus X_1$. In particular, $X_2\setminus X_1$ is a fundamental domain for $\cE$ (in $\D$), and $X_1\cap\D$ admits a tessellation by $(d+1)-$gons $\cE^{-n}(X_2\setminus X_1)$, $n\geq 1$.

\subsubsection{Quasiconformal conjugacy between Farey-like maps and the parabolic Blaschke product}

\begin{lemma}\label{qc_conj_fd_bd_lem}
Let $\cE$ be a Farey-like external map. Then there exists a homeomorphism $\mathfrak{g}:\overline{\D}\to\overline{\D}$ that is quasiconformal on $\D$, sends $1$ to $1$, and conjugates the restriction of $\cE$ to a one-sided pinched neighbourhood of $\mathbb{S}^1$ (pinched at the points of $\cE^{-1}(1)$)   to the restriction of $B_d$ to a one-sided pinched neighbourhood of $\mathbb{S}^1$ (pinched at the points of $B_d^{-1}(1)$).
\end{lemma}
\begin{figure}[h!]
\captionsetup{width=0.96\linewidth}
\begin{tikzpicture}
\node[anchor=south west,inner sep=0] at (0,0) {\includegraphics[width=0.66\textwidth]{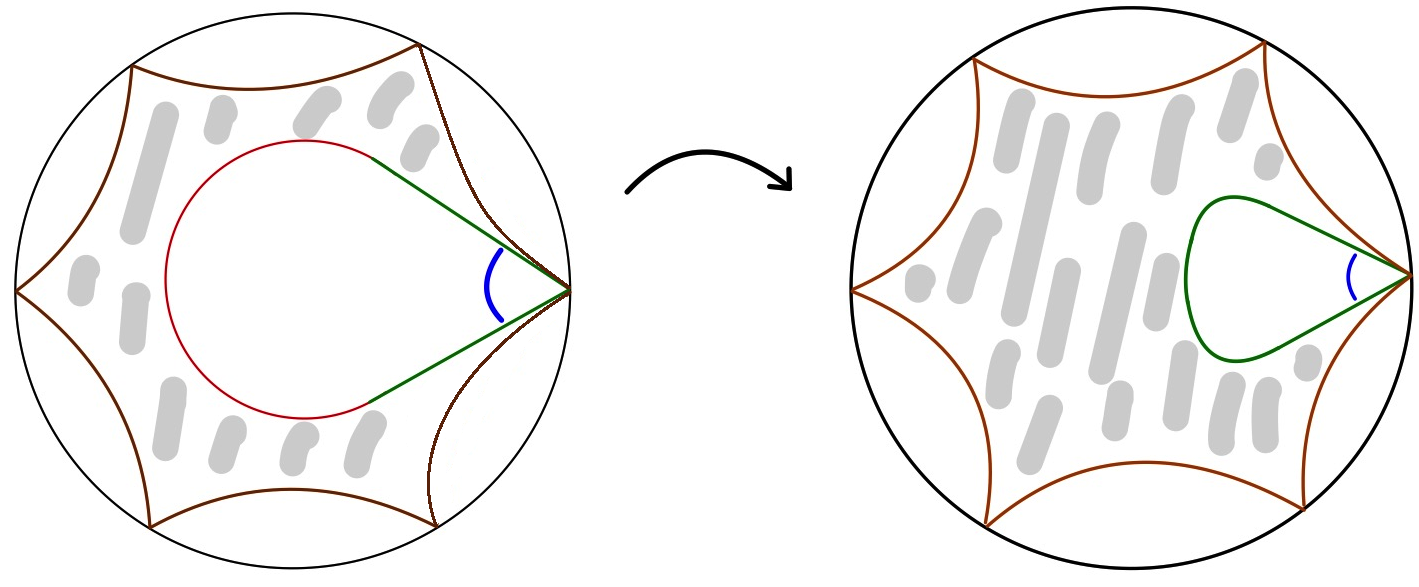}}; 
\node at (1.64,1.75) {$\mathfrak{T}$};
\node at (3.54,1.75) {\begin{small}$1$\end{small}};
\node at (8.48,1.84) {\begin{small}$1$\end{small}};
\node at (2.7,1.75) {\begin{small}$\theta_0$\end{small}};
\node at (7.32,1.75) {\begin{small}$\mathcal{P}$\end{small}};
\node at (4.16,2.8) {$\mathfrak{g}$};
\node at (7.78,1.82) {\begin{small}$\frac{\theta_0}{2}$\end{small}};
\end{tikzpicture}
\caption{Depicted is a degree $6$ Farey-like external map. The map $\mathfrak{g}$ carries $\mathfrak{T}$ onto the petal $\overline{\mathcal{P}}$ (of $B_6$) conformally and behaves like a square root map near $1$. The shaded regions are the fundamental pinched annuli $Y_{\cE}$ and $Y_{B_6}$.}
\label{rd_bd_qc_conj_fig}
\end{figure}

We present two proofs of the existence of the desired quasiconformal conjugacy for our class of maps. They apply the key ingredient, the Fatou coordinate at the parabolic point, in different ways: one proof is based on a classical theorem of Warschawski, the other is based on the dividing arcs method of \cite{L1}. The choice of which proof to consult is left to the reader's discretion.

\begin{proof}
A priori, the curve $\gamma:=(\partial X_2\cap\D)\cup\{1\}$ may have a cusp at $1$. Without loss of generality, we may open up this cusp to form a wedge of angle $\theta_0>0$ by replacing $\gamma$ (near $1$) with a pair of straight line segments. The fact that the branches of $\cE$ at $1$ extend locally as simple parabolic germs with repelling directions along the imaginary axis (at $1$) implies that if we redefine $X_1$ to be the closed region bounded by $\mathbb{S}^1$ and $\cE^{-1}(\gamma)$, then $\cE:X_1\to X_2$ is Farey-like. We also set $\mathfrak{T}:=\overline{\D\setminus X_2}$, so $\partial\mathfrak{T}=\gamma$ (see Figure~\ref{rd_bd_qc_conj_fig}).

In the $B_d-$plane, we choose an attracting petal $\mathcal{P}\subset\D$ of $B_d$ at the parabolic point $1$ such that 
\begin{enumerate}
\item $\mathcal{P}$ contains the critical value of $B_d$ in $\D$, 
\item the critical point $0$ lies on $\partial\mathcal{P}$ or outside $\overline{\mathcal{P}}$, and 
\item near $1$, the boundary $\partial\mathcal{P}$ is the union of two straight lines that meet at an angle $\theta_0/2$. 
\end{enumerate}
One can construct the homeomorphism $\mathfrak{g}:\overline{\D}\to\overline{\D}$ with the claimed properties in two ways.
\medskip

\noindent\textbf{First proof.}
Let us choose a homeomorphism $\mathfrak{g}: \mathfrak{T}\to \overline{\mathcal{P}}$ that is conformal on the interior and sends $1$ to $1$. By \cite[Theorem~3.11]{Pom}, the map $\mathfrak{g}$ has the asymptotics 
\begin{equation}
\zeta\mapsto 1+b(\zeta-1)^{1/2}+o((\zeta-1)^{1/2}),
\label{g_asymp_formula}
\end{equation} 
for some $b\in\C^*$, near $1$ (for a suitable branch of square root).
Note that both $\cE: \cE^{-1}(\gamma)\to \gamma$ and $B_d: B_d^{-1}(\partial\mathcal{P})\to\partial\mathcal{P}$ are degree $d$ orientation-preserving covering maps. We lift the map $\mathfrak{g}:\gamma\to \partial\mathcal{P}$ via the above coverings to get a homeomorphism from $\cE^{-1}(\gamma)$ onto $B_d^{-1}(\partial\mathcal{P})$, which we also denote by $\mathfrak{g}$. 

By design, the curves $\cE^{-1}(\gamma)$ and $\gamma$ bound a pinched annulus $Y_{\cE}$ that is a fundamental domain for the action of $\cE$ (in $\D$). Similarly, the curves $B_d^{-1}(\partial\mathcal{P})$ and $\partial\mathcal{P}$ bound a pinched annulus $Y_{B_d}$ that is a fundamental domain for the action of $B_d$. These fundamental pinched annuli are shaded in grey in Figure~\ref{rd_bd_qc_conj_fig}.

We claim that there exists a quasiconformal homeomorphism $\mathfrak{g}$ from the pinched fundamental annulus $Y_{\cE}$ onto the pinched fundamental annulus $Y_{B_d}$ which continuously agrees with $\mathfrak{g}$ already defined.
Once this claim is established, we can use the equivariance property of $\mathfrak{g}$ on the boundaries of the pinched annuli (more precisely, the fact that $\mathfrak{g}$ conjugates the action of $\cE$ to that of $B_d$ on the boundaries of their fundamental pinched annuli) to lift it under the iterates of $\cE$ and $B_d$, and obtain a quasiconformal homeomorphism of $\D$ with a quasisymmetric extension to $\mathbb{S}^1$. By construction, this map $\mathfrak{g}$ would conjugate the restriction of $\cE$ to a one-sided pinched neighbourhood of $\mathbb{S}^1$ (pinched at the points of $\cE^{-1}(1)$) to the restriction of $B_d$ to a one-sided pinched neighbourhood of $\mathbb{S}^1$ (pinched at the points of $B_d^{-1}(1)$).

We now proceed to prove the quasiconformal interpolation claim. Note that the boundaries of $Y_{\cE}, Y_{B_d}$ are piecewise smooth. Moreover, since the branches of $\cE$ at the break-points of its domain of definition admit local analytic extensions, it follows that the smooth pieces of the outer boundary curves of $Y_{\cE}, Y_{B_d}$ meet at positive angles except at the point $1$. Quasiconformality of $\mathfrak{g}: \mathfrak{T}\to \overline{\mathcal{P}}$ at $1$ now translates to the fact that $\mathfrak{g}:\partial Y_{\cE}\to\partial Y_{B_d}$ is quasisymmetric away from $1$. Thus, the existence of the desired quasiconformal interpolating map outside a neighbourhood of $1$ follows by the Ahlfors-Beurling extension theorem. It remains to justify that a similar quasiconformal interpolation can also be performed near $1$, where the boundaries of the pinched fundamental annuli subtend zero angles (i.e., they form cusps). By symmetry of the situation, it suffices to demonstrate this for the top cusps. To facilitate the construction of this interpolating map, we will now perform changes of coordinates that carry the point $1$ (both in the $\cE-$plane and in the $B_d-$plane) to the point at $\infty$.

To this end, we use the change of coordinate $\pmb{j}_1:\zeta\mapsto -\frac{i}{a(\zeta-1)}$ (see Section~\ref{para_asymp_subsec}) to send $1$ to $\infty$ and the positive imaginary axis at $1$ to the negative real axis near $\infty$. The map $\pmb{j}_1$ conjugates the Farey-like map $\cE$ on $\{\zeta\in\C:\vert\zeta-1\vert<\epsilon,\ \im(\zeta)>0\}$ (for $\epsilon>0$ small enough) to a map of the form $z\mapsto z+1+O(1/z)$ near $\infty$, and sends the top cusp of $Y_{\cE}$ at $1$ to an unbounded curvilinear strip $S_1$ of width $1$ bounded by a pair of smooth curves. 
Analogously, we use the map $\pmb{j}_2:\zeta\mapsto \frac{c}{(\zeta-1)^2}$ (for some $c\in\C^*$) to send $1$ to $\infty$ and the negative real axis at $1$ to the positive real axis near $\infty$. The map $\pmb{j}_2$ conjugates the Blaschke product $B_d$ on $\{\zeta\in\C:\vert\zeta-1\vert<\epsilon,\ \arg(\zeta)\in\{\pi/2+\delta,3\pi/2-\delta\}\}$ (for $\epsilon,\delta>0$ sufficiently small) to a map of the form $z\mapsto z+1+O(1/z)$ near $\infty$, and sends the top cusp of $Y_{B_d}$ at $1$ to an unbounded curvilinear strip $S_2$ of width $1$ bounded by a pair of smooth curves. Thanks to the asymptotic development~\eqref{g_asymp_formula}, the map $\pmb{j}_2\circ\mathfrak{g}\circ\pmb{j}_1^{-1}$ is easily seen to be asymptotically linear near $\infty$.

Next, we map the curvilinear infinite strips $S_1,S_2$ to the straight infinite strip $\mathscr{S}:=\{z=x+iy\in\C: x>0,\ y\in\left(-\pi/2,\pi/2\right)\}$ by conformal maps $\pmb{k}_1,\pmb{k}_2$. By \cite{War42}, the map $\pmb{k}_i$, $i\in\{1,2\}$ has asymptotics $\pmb{k}_i(z) = c_i z + o(z)$ as $z\to\infty$, for some $c_i\in\C^*$. A straightforward computation (as in \cite[Lemma~5.3]{LLMM3}) shows that the map 
$$
\pmb{k}_2\circ\pmb{j}_2\circ\mathfrak{g}\circ\pmb{j}_1^{-1}\circ\pmb{k}_1^{-1}:\partial\mathscr{S}\to\partial\mathscr{S}
$$ 
is of the form $\lambda z+o(z)$ as $\re(z)\to+\infty$, and the maps on the upper and lower boundaries are a bounded distance from each other. Therefore, linear interpolation yields a quasiconformal homeomorphism of $\mathscr{S}$ that continuously extends the above boundary maps (once again, we refer the reader to \cite[Lemma~5.3]{LLMM3} for an explicit formula for such an interpolating quasiconformal homeomorphism). Pulling it back to the top cusps of $Y_{\cE}, Y_{B_d}$ via the conformal maps $\pmb{j}_i, \pmb{k}_i$, $i\in\{1,2\}$, produces the desired quasiconformal map $\mathfrak{g}:Y_{\cE}\to Y_{B_d}$.
\medskip

\noindent\textbf{Second proof.}
This proof uses the method of \textit{dividing arcs}, introduced by the second author in \cite{L1}. We use these to augment the pinched annuli defined above by adding topological discs in such a way that there are no longer cusps at the pinch points: the existence of the desired quasiconformal homeomorphism $\mathfrak{g}$ then follows from standard theory of extendability of 
quasisymmetric maps on piecewise smooth boundaries (without cusps) to quasiconformal homeomorphisms on interiors. 

\begin{figure}
\captionsetup{width=0.96\linewidth}
\centering
\includegraphics[width=5cm]{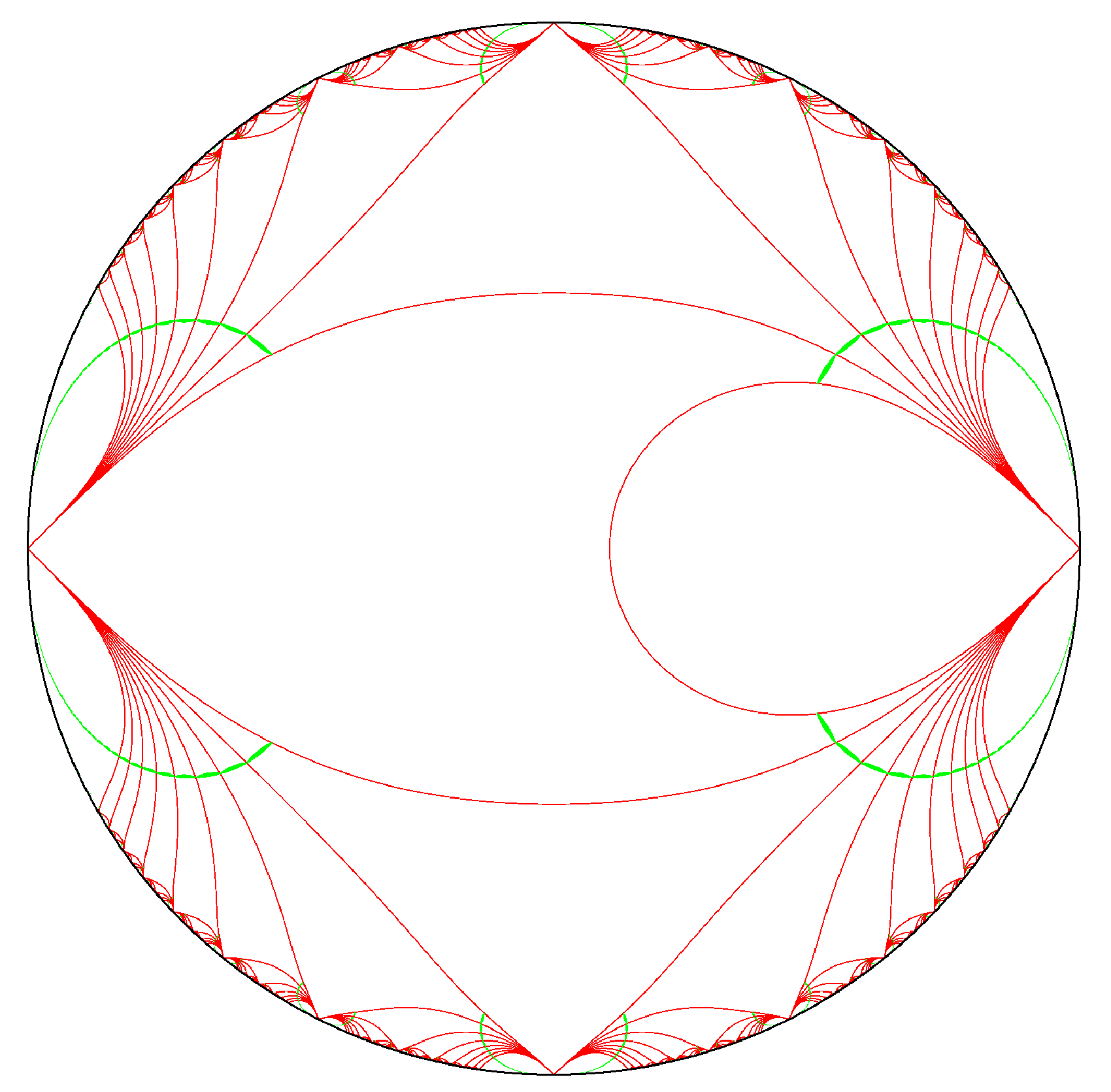}
\caption{Invariant arcs for $B_2$, and their inverse images}
\label{arcsfig}
\end{figure}

Our dividing arcs for the map  $B_d$ are a pair of smooth forward invariant arcs $\gamma_B^+, \gamma_B^-$, parametrized by $[0,1]$, lying in the repelling petals of the parabolic fixed point $1\in \D$,
emerging from it  tangentially to the unit circle, and meeting the boundary $\partial \mathcal{P}$ of our chosen attracting petal $\mathcal{P}$ transversely (Figure \ref{arcsfig}): we can construct 
$\gamma_B^+$ and $\gamma_B^-$ as pre-images of repelling horizontal straight lines in the Fatou coordinate, as in \cite{L1}.  Similarly we may construct
dividing arcs for the Farey-like map $\cE$ at its parabolic point, meeting the boundary curve of $X_2$ transversely. (We remark that for the Farey-like maps that will concern us, the Hecke 
map and the Farey map, the existence of dividing arcs on both sides of the parabolic point is self-evident, since the elements $\alpha_1$ and 
$\alpha_d$ of the Hecke group are parabolic, and so have horocycles emanating from their fixed points.)

Let $d_B^+$ denote the topological disc trapped between $\gamma_B^+$ and $\partial \mathcal{P}$, and $d_B^-$ that trapped by $\gamma_B^-$ and $\partial \mathcal{P}$ (see Figure \ref{arcsfig}). 
Define the augmented pinched annulus $U$ by setting $$U:= Y_{B_d}\cup d_B^+ \cup d_B^-.$$ 
We next augment the pinched annulus $Y_\cE=X_2 \setminus X_1$ by adding suitable part-horodiscs to it.
Let $d_\cE^+$ denote the subset of $\D$ trapped between $\gamma$ and $\gamma_\cE^+$,
and let $d_\cE^-$ be the subset trapped between $\gamma$ and $\gamma_\cE^-$. 
We define the augmented pinched annulus $V$ by setting $$V: =Y_\cE \cup d_\cE^+ \cup d_\cE^-.$$

Via Fatou coordinates on both sides we can construct analytic conjugacies $\psi^+: \gamma_\cE^+ \rightarrow \gamma_B^+$
and $\psi^-: \gamma_\cE^- \rightarrow \gamma_B^-$, which 
extend to quasiconformal maps $\psi_d^+: d_\cE^+ \rightarrow d_B^+$ and 
$\psi_d^-: d_\cE^- \rightarrow d_B^-$
respectively.

Let $h: \gamma \rightarrow \partial \mathcal{P}$ be a $C^1$ diffeomorphism with $h(\gamma_\cE^+)=\psi^+(\gamma_\cE^+),\,\,\,h(\gamma_\cE^-)=\psi^-(\gamma_\cE^-)$, and 
let $h_1: \cE^{-1}(\gamma) \rightarrow B_d^{-1}(\partial\mathcal{P})$ be a lift.
Note that $h: \gamma \rightarrow \partial \mathcal{P}$ extends to a quasiconformal map $H: \mathfrak{T} \rightarrow \mathcal{P}$, where $\partial \mathfrak{T}=\gamma$.
There also exists a quasiconformal homeomorphism  
$$\psi_2: \overline {V} \rightarrow \overline {U}$$
such that $\psi_{2|d_\cE^{\pm}}=\psi_d^{\pm}$, $\psi_{2|\gamma}=h$, and $\psi_{2|\cE^{-1}(\gamma)}=h_1$.

Where an inverse image of the augmented annulus $U$ overlaps with an earlier inverse image, their intersection is always a subset of some 
$d^\pm_\cE$ (or an inverse image thereof), and we simply remove this overlap by cutting along the relevant part of invariant curve $\gamma^\pm_\cE$ (or inverse image) to
obtain a partition of $\D\setminus X_2$, and a similar partition of  $\D\setminus \mathcal P$. The dynamics matches on the partition boundaries
which are segments of the invariant curves
(since $\psi^+$ and $\psi^-$ are both analytic conjugacies), so we can lift the quasiconformal map $\psi_2$ via the dynamics of $\cE$ and the corresponding dynamics of $B_d$, 
to obtain a quasiconformal map 
$\psi_3:\D\setminus (V\cup \mathfrak{T})\rightarrow  \D\setminus (U\cup \mathcal{P})$ 
and finally define
$\mathfrak{g}:\D\to\D$
by:
$$ \mathfrak{g} =\left\{
\begin{array}{cl}
H  &\mbox{on  } \mathfrak{T} \\
\psi_2 &\mbox{on  } V \\
\psi_3 &\mbox{on  } \D\setminus (V\cup \mathfrak{T}) \\
\end{array}\right.
$$
As $\mathfrak{g}$ is quasiconformal, it extends to a quasiconformal map  $\mathfrak{g}:\overline{\D}\to\overline{\D}$.
\end{proof}

\subsection{Mating parabolic rational maps with Farey-like maps}\label{mating_construct_subsec}

Throughout the rest of this section, we will fix an $R\in\pmb{\mathcal{B}}_d$.

\begin{lemma}\label{conf_mating_lem}
Let $\cE:X_1\to X_2$ be a Farey-like external map.
Then, there exists a pinched polynomial-like map $(f,P_1,P_2)$ that is hybrid equivalent to $R$ and has $\cE$ as its external map.
\end{lemma}

\begin{proof}
There exists a conformal map $\psi_R:\D\to \mathcal{A}(R)$ that conjugates $B_d$ to $R$, and extends continuously to $1$ with $\psi_R(1)=\infty$.
Recall that the quasiconformal homeomorphism $\mathfrak{g}:\D\to\D$ constructed in Lemma~\ref{qc_conj_fd_bd_lem} conjugates $\cE$, restricted to a one-sided pinched neighbourhood of $\mathbb{S}^1$ (pinched at the points of $\cE^{-1}(1)$) to $B_d$, restricted to a one-sided pinched neighbourhood of $\mathbb{S}^1$ (pinched at the points of $B_d^{-1}(1)$).

We now define a map on a subset of $\widehat{\C}$ as follows:
$$
\widetilde{f}:=
\begin{cases}
\left(\psi_R\circ \mathfrak{g}\right)\circ \cE\circ\left(\mathfrak{g}^{-1}\circ\psi_R^{-1}\right)\ {\rm on\ } A,\\
R \quad {\rm on\ } K(R).
\end{cases}
$$
where $A=\psi_R(\mathfrak{g}(X_1\cap\D))$ if one uses the first proof of Lemma~\ref{qc_conj_fd_bd_lem}, and $A=\psi_R(\mathfrak{g}(X_1\cap\D\cup (d_{\mathcal F}^+ \cup d_{\mathcal F}^-))$ if one uses the second proof.
By the equivariance properties of $\psi_R$ and $\mathfrak{g}$, the map $\widetilde{f}$ agrees with $R$  on the closure of a neighbourhood of $K(R)\setminus R^{-1}(\infty)$. By quasiconformal removability of finitely many points, the map $\widetilde{f}$ is a quasiregular map on the interior of its domain of definition (see Figure~\ref{top_mating_fig}). 

Let $\mu$ be the Beltrami coefficient on $\widehat{\C}$ defined as $\mu\vert_{\mathcal{A}(R)}=(\mathfrak{g}^{-1}\circ\psi_R^{-1})^*(\mu_0)$, where $\mu_0$ is the standard complex structure on $\D$, and $\mu\vert_{K(R)}=0$. Since $\cE$ is holomorphic, it follows that $\mu$ is $\widetilde{f}$-invariant. Quasiconformality of $\mathfrak{g}^{-1}\circ\psi_R^{-1}$ implies that $\vert\vert\mu\vert\vert_\infty<1$. Let $\kappa$ be a quasiconformal homeomorphism of $\widehat{\C}$ that solves the Beltrami equation with coefficient $\mu$. Then the conjugated map $f:=\kappa\circ\widetilde{f}\circ\kappa^{-1}$ is holomorphic on the interior of a pinched polygon 
$$
P_1:= \kappa(\mathrm{Dom}(\widetilde{f})).
$$
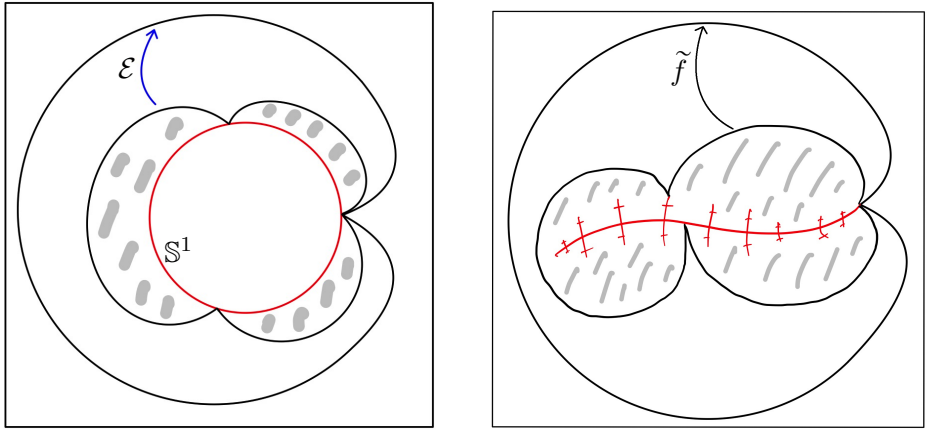
\begin{figure}[h!]
\captionsetup{width=0.96\linewidth}
\begin{tikzpicture}
\node[anchor=south west,inner sep=0] at (0,0) {\includegraphics[width=0.45\textwidth]{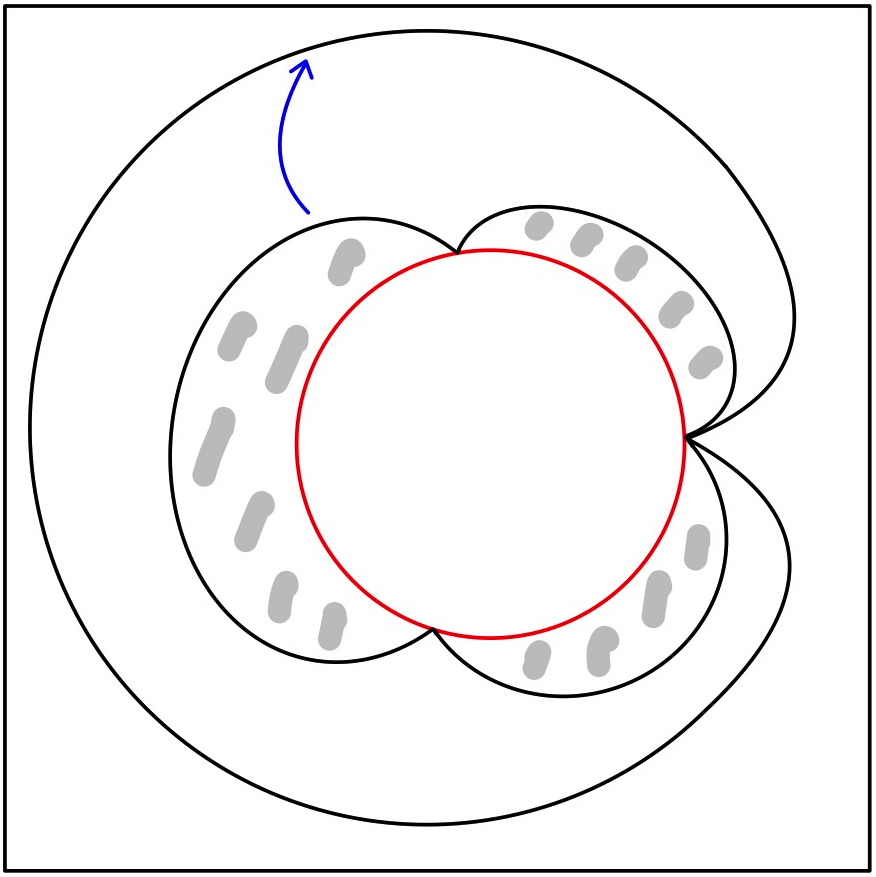}};
\node[anchor=south west,inner sep=0] at (6.4,0) {\includegraphics[width=0.455\textwidth]{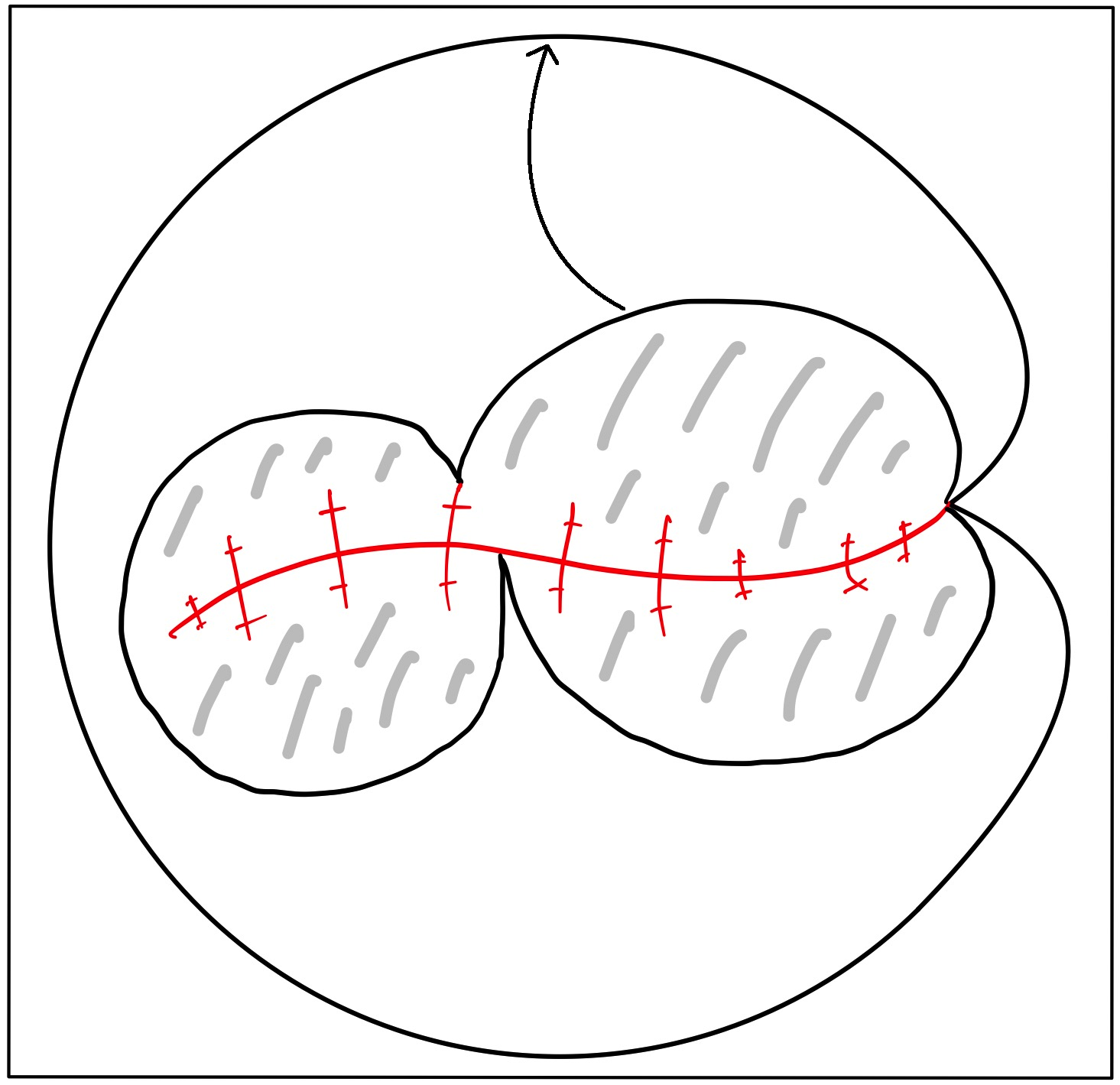}};
\node at (1.64,4.8) {$\cE$};
\node at (2.32,2.4) {$\mathbb{S}^1$};
\node at (8.9,4.8) {$\widetilde{f}$};
\end{tikzpicture}
\caption{The domain of definition for the topological mating $\widetilde{f}$.}
\label{top_mating_fig}
\end{figure}
It is easy to see from the construction that $f:P_1\to P_2:=f(P_1)$ is a degree $d$ pinched polynomial-like map with filled Julia set $K(f)=\kappa(K(R))$. Since $\kappa^{-1}$ is conformal a.e. on $K(f)$, it hybrid conjugates the above pinched polynomial-like map to a pinched polynomial-like restriction of $R$ (with filled Julia set $K(R)$).

Finally, the restriction $f: P_1\setminus K(f)\to P_2\setminus K(f)$ is conformally conjugate to $\cE:X_1\cap\D\to X_2\cap\D$ via the map $(\kappa\circ\psi_R\circ\mathfrak{g})^{-1}$. 
\end{proof}

\subsection{Pinched polynomial-like maps as matings of $R$ and $F_d$}\label{farey_pinched_poly_subsec}

The Farey map $F_d:F_d^{-1}(\cD_1)\to\cD_1$ is an example of a Farey-like map in the sense of Definition~\ref{farey_like_def}. This observation, combined with Lemma~\ref{farey_pinched_poly_subsec}, allows us to construct a pinched polynomial-like map as the mating of a parabolic rational map $R\in\pmb{\mathcal{B}}_d$ and the map $F_d$.

\begin{theorem}\label{farey_mating_thm}
There exists a pinched polynomial-like map $R_F$ that is hybrid equivalent to $R$ and has $F_d$ as its external map. Moreover, the map $R_F$ admits an analytic continuation with the following properties.

\begin{enumerate}[leftmargin=8mm]
    \item There exist a Jordan domain $\cU$ and a continuous extension $R_F:\overline{\cU}\to\widehat{\C}$ which is meromorphic on $\cU$ such that $R_F:R_F^{-1}(\overline{\cU})\to\overline{\cU}$ is a degree $d$ pinched polynomial-like map hybrid conjugate to a pinched polynomial-like restriction of $R$ (with filled Julia set $K(R)$), and $R_F:\overline{\cU}\setminus K(R_F)\to\widehat{\C}\setminus K(R_F)$ is conformally conjugate to $F_d:\cD_1\cap\D\to\D$.

    \item The Jordan curve $\partial\cU$ is non-singular real-analytic except possibly at the fixed point $\pmb{s}$ of $R_F$ that corresponds to the parabolic fixed point $1$ of~$F_d$.
\end{enumerate}
\smallskip

(Here, $K(R_F)$ denotes the filled Julia set of the pinched polynomial-like restriction of $R_F$, and $\cD_1$ is defined in Section~\ref{farey_map_subsec}).
\end{theorem}
\begin{proof}
Since the map $F_d:F_d^{-1}(\cD_1)\to\cD_1$ is Farey-like, Lemma~\ref{conf_mating_lem} gives a pinched polynomial-like map $R_F$ that is a mating of $R$ and $F_d:F_d^{-1}(\cD_1)\to\cD_1$ (where the parabolic fixed point $\infty$ of $R$ is glued with the parabolic fixed point $1$ of $F_d$). 

By construction, there exists a conformal map $\psi_{R_F}:\widehat{\C}\setminus K(R_F)\to\D$ that conjugates $R_F$ to $F_d$, wherever defined. We extend the pinched polynomial-like map $R_F$ as the conjugate of $F_d:\cD_1\cap\D\to\D$ via $\psi_{R_F}$. We define $\cU$ to be the interior of the domain of definition $K(R_F)\cup\psi_{R_F}^{-1}(\cD_1\cap\D)$ of the extended map $R_F$.
Since the inner boundary of $\cD_1$ meets $\mathbb{S}^1$ at a unique point, it follows that $\cU$ is a Jordan domain. 

The final statement is a consequence of the fact that $\partial\cU\setminus\{\pmb{s}\}$ is the image of the non-singular real-analytic curve $\partial\cD_1\cap\D$ under the conformal map~$\psi_{R_F}^{-1}$.
\end{proof}

\subsection{Pinched polynomial-like maps as matings of $R$ and $H_d$}\label{hecke_pinched_poly_subsec} 

In order to construct a Farey-like restriction of the Hecke map, let us thicken the arc $[0,1]$ to a monogon $\gamma$ that starts and ends at $1$ (and is disjoint from $\mathbb{S}^1$ otherwise) and surrounds the arc $[0,1]$ (see Figure~\ref{farey_like_fig} (right)). We denote the component of $\D\setminus\gamma$ containing the origin by $\mathbb{L}$. It is easy to see from the parabolic dynamics of $H_d$ at $1$ that $H_d:H_d^{-1}(\overline{\D}\setminus\mathbb{L})\to\overline{\D}\setminus\mathbb{L}$ is Farey-like.

\begin{theorem}\label{hecke_mating_thm}
There exists a pinched polynomial-like map $R_H$ that is hybrid equivalent to $R$ and has $H_d$ as its external map. Moreover, there exist a pinched polygon $\cV$ and a continuous extension $R_H:\cV\to\widehat{\C}$ which is meromorphic on $\Int{\cV}$ such that $R_H:R_H^{-1}(\cV)\to\cV$ is a degree $d$ pinched polynomial-like map hybrid conjugate to a pinched polynomial-like restriction of $R$ (with filled Julia set $K(R)$), and $R_H:\cV\setminus K(R_H)\to\widehat{\C}\setminus K(R_H)$ is conformally conjugate to $H_d:\cD_2\cap\D\to\D$.
\smallskip

\noindent (Here, $K(R_H)$ denotes the filled Julia set of the pinched polynomial-like restriction of $R_H$, and $\cD_2$ is defined in Section~\ref{hecke_map_subsec}.)
\end{theorem}
\begin{proof}
By the discussion preceding this corollary, the Hecke map $H_d$ admits a Farey-like restriction. Hence, the existence of the desired pinched polynomial-like map $R_H$ (which is hybrid equivalent to $R$ and has the above Farey-like restriction of $H_d$ as its external map) follows from Lemma~\ref{conf_mating_lem}.

Note also that by construction, there exists a conformal map $\psi_{R_H}:\widehat{\C}\setminus K(R_H)\to\D$ that conjugates $R_H$ to $H_d$, wherever defined. One can now extend the pinched polynomial-like map $R_H$ as the conjugate of $H_d:\cD_2\cap\D\to\D$ via $\psi_{R_H}$. The pinched polygon $\cV$ is then the domain of definition $K(R_H)\cup\psi_{R_H}^{-1}(\cD_2\cap\D)$ of the extended map $R_H$.
\end{proof}

\begin{figure}
\captionsetup{width=0.96\linewidth}
\includegraphics[width=4cm]{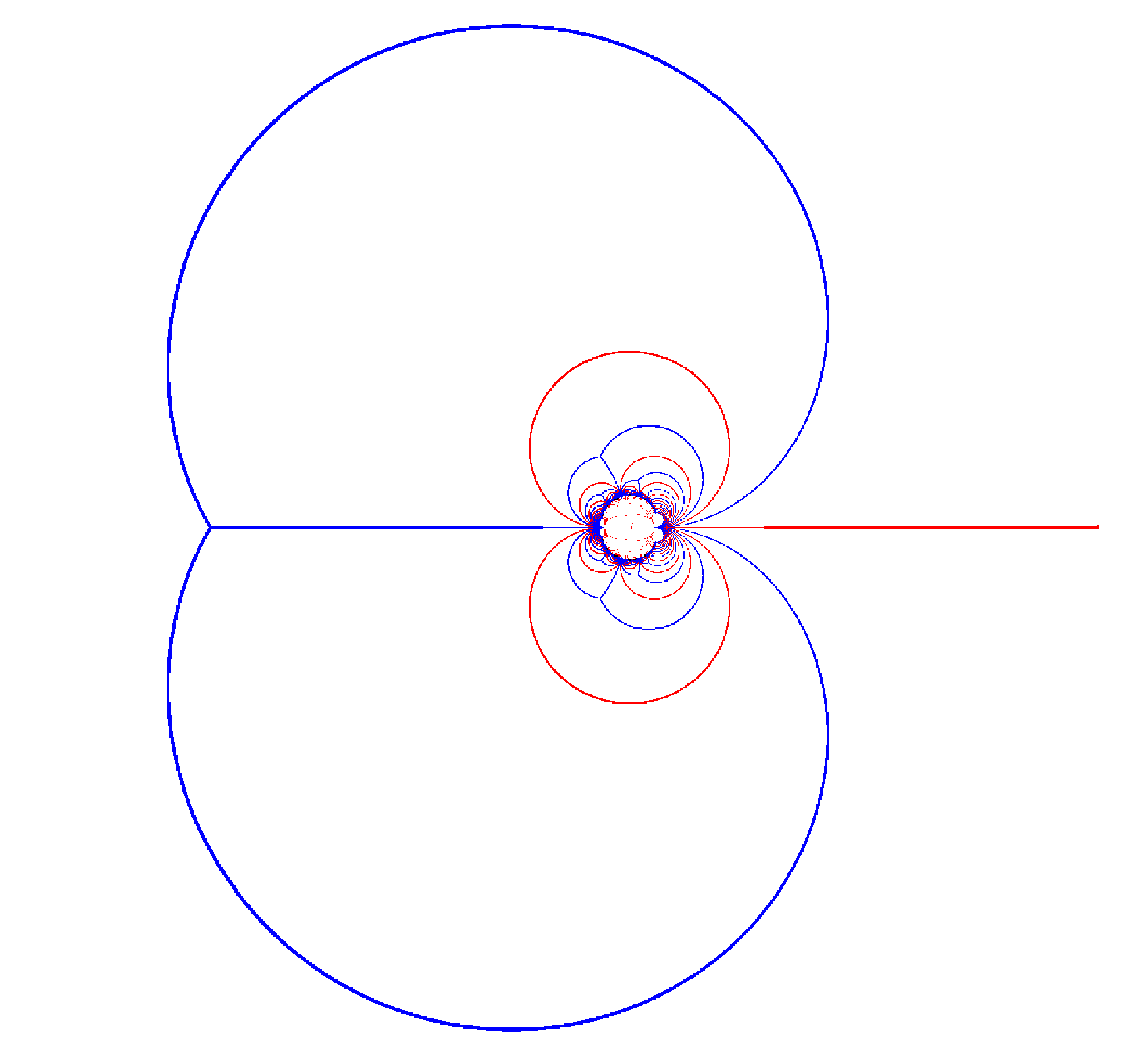}\hspace{1cm}\includegraphics[width=4cm]{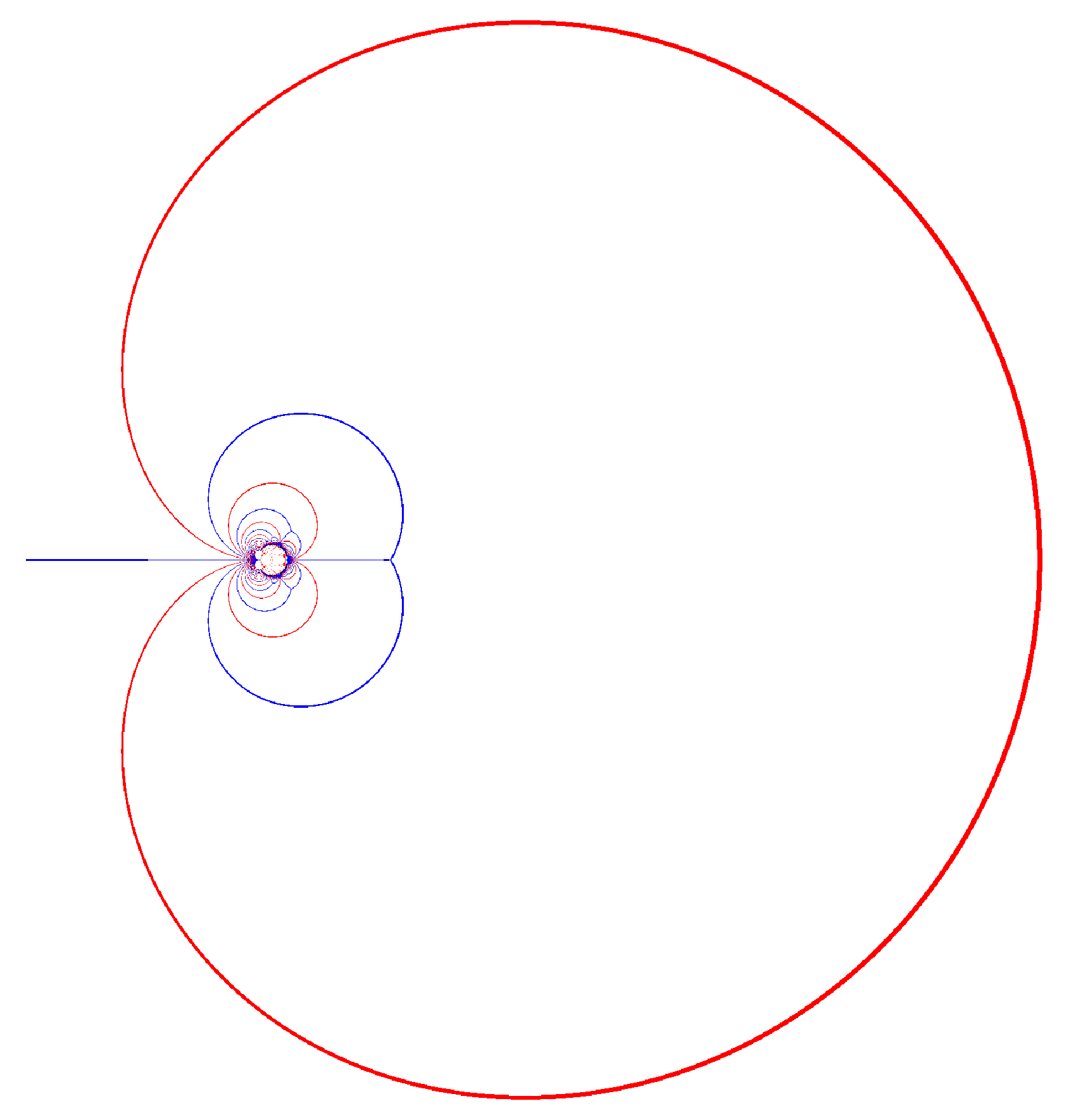}

\includegraphics[width=4cm]{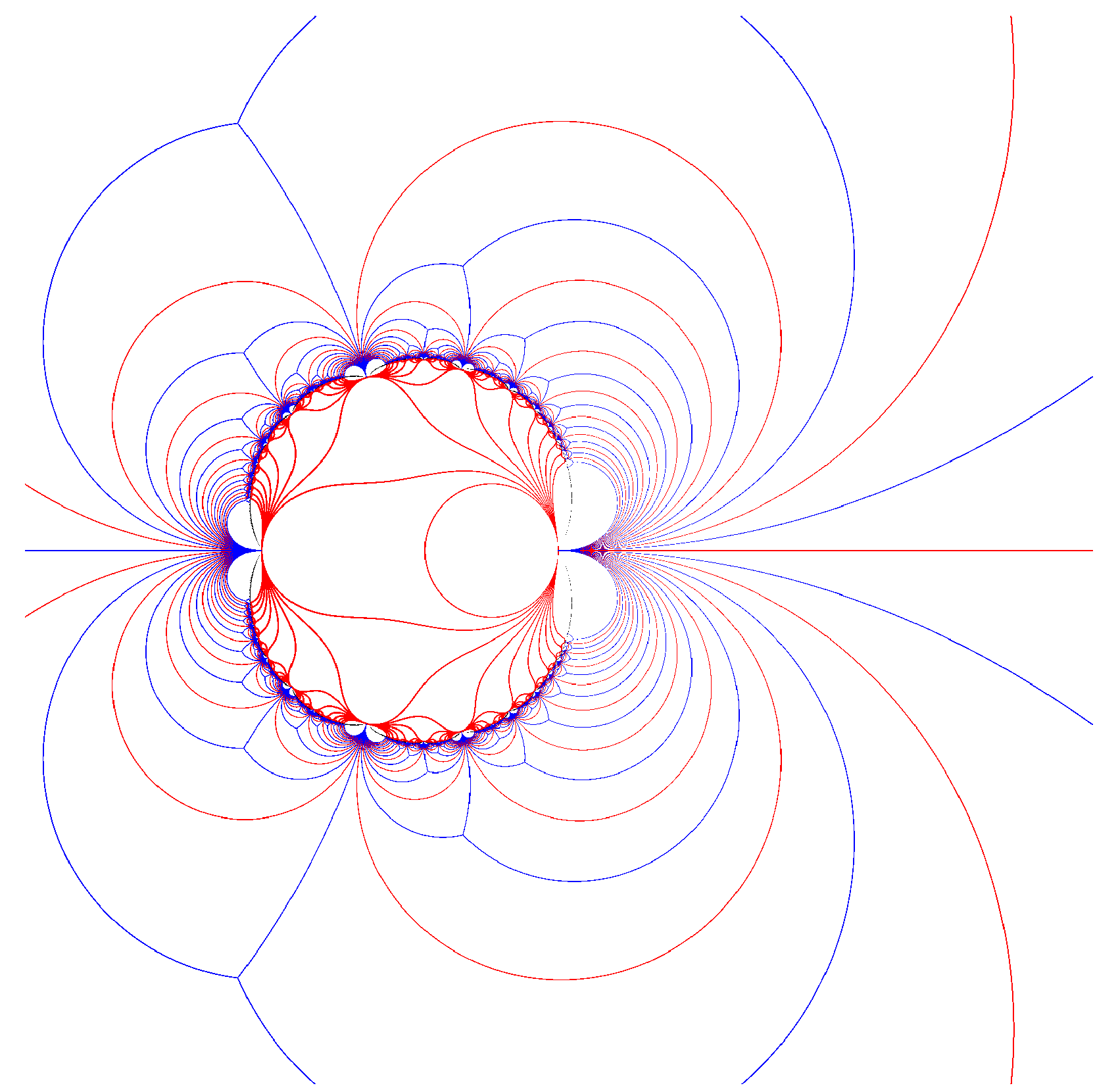}\hspace{1cm}\includegraphics[width=4cm]{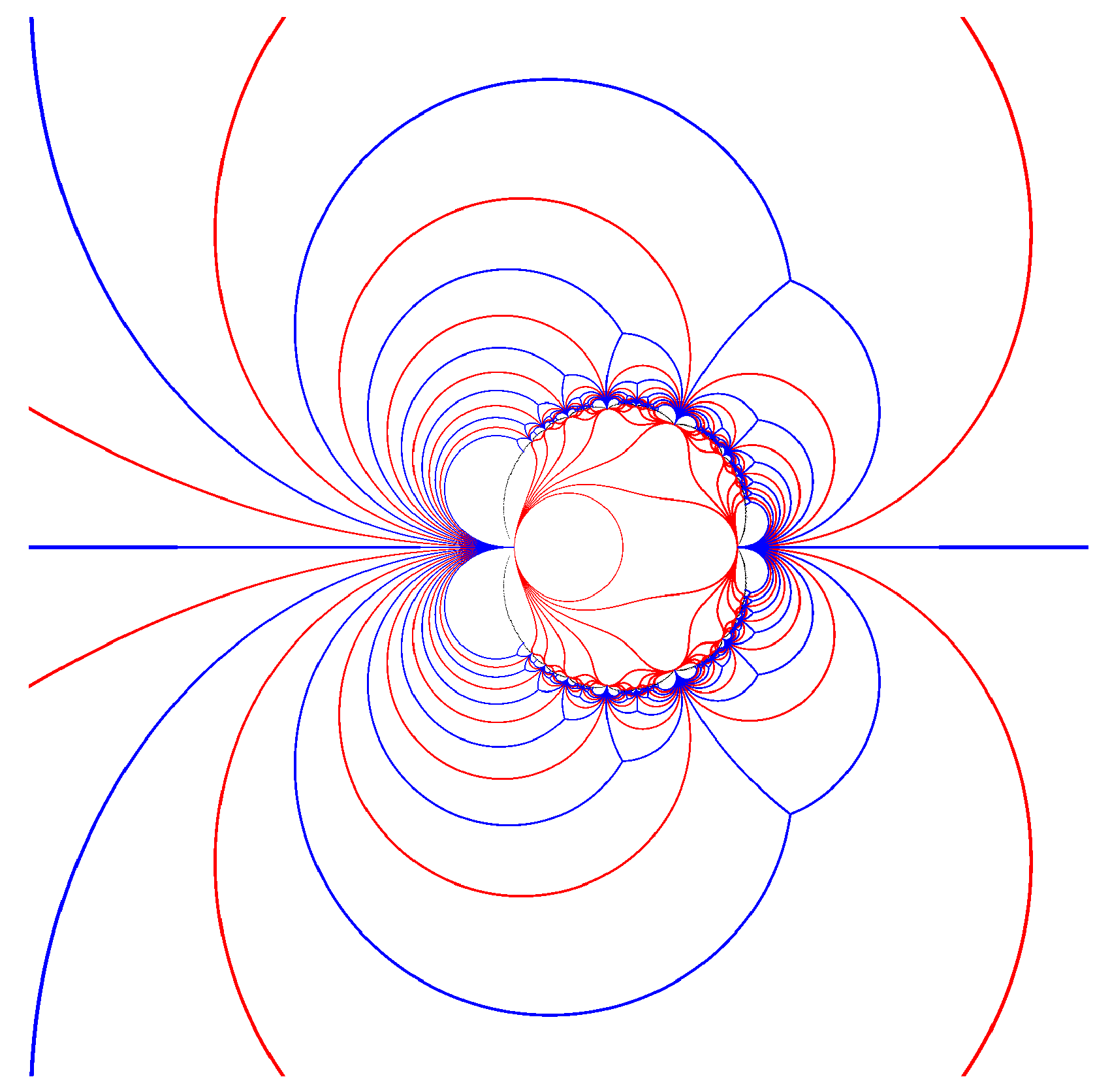}

\includegraphics[width=4cm]{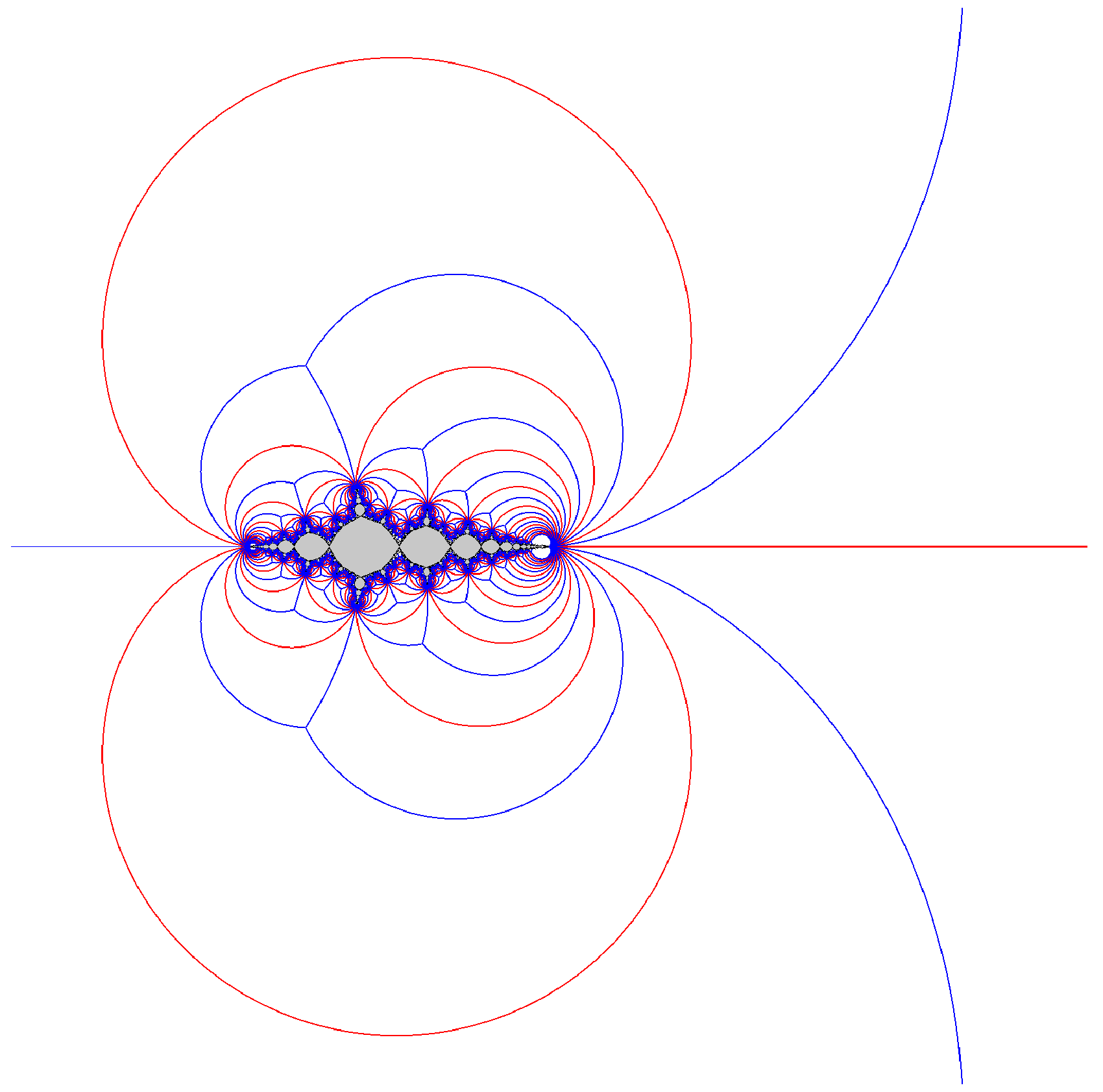}\hspace{1cm}\includegraphics[width=4cm]{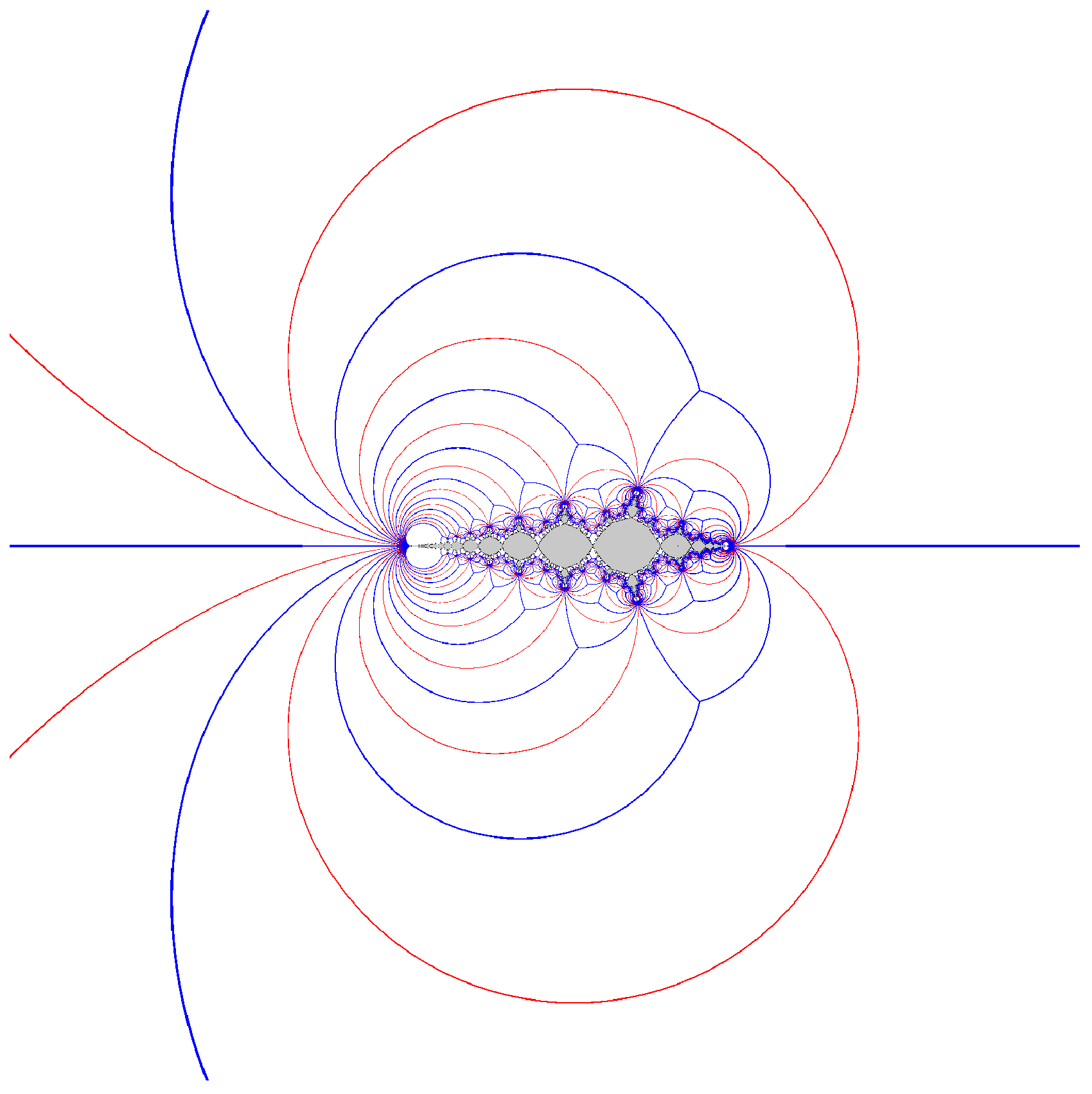}

\caption{Computer plots of $R_H$ (left-hand column) and $R_F$ (right-hand column): see text for details.}\label{RHandRF}
\end{figure}

In Figure \ref{RHandRF} we exhibit computer plots of $R_H$ and $R_F$ for two examples of rational maps $R\in\pmb{\mathcal{B}}_2$, firstly for $R$ the Blaschke product map $B_2$ 
(top row zoomed out, and middle row near their Julia sets), and secondly (bottom row) for $R$ the basilica, 
the map $R \in\pmb{\mathcal{B}}_2$ determined by the condition 
that the critical point of $R$ in the filled Julia set $K(R)$ is periodic of period two. On the complement of $K(R_H)$ (respectively, $K(R_F)$) can be seen a copy of the tiling of 
$\theta_2(\D)=\D/\langle\sigma\rangle$ (respectively, the tiling
of $\theta_1(\D)=\D/\langle\rho\rangle$). We note that the tilings of $\theta_1(\D), \theta_2(\D)$ are the images of the tessellation of the unit disk induced by the Hecke group $\cH_{d+1}$ (see Section~\ref{hecke_group_subsec}) under $\theta_1, \theta_2$, respectively. Notice that in the case of $R=B_2$, the Julia set is a quasicircle: the mating construction  
has quasiconformally replaced the $B_2$ action on $\widehat \C \setminus K(B_2)$, the Schwarz reflection of the action of $B_2$ on the round disc $K(B_2)$, by the action of the Hecke map 
on $\theta_2(\D)$ (respectively the Farey map on $\theta_1(\D)$).

\section{From pinched polynomial-like maps to algebraic correspondences}\label{ppl_like_to_corr_sec}

In the previous sections, we extracted two piecewise analytically conjugate external maps from the Hecke group, and showed that they can be mated with parabolic rational maps in $\pmb{\mathcal{B}}_d$ as pinched polynomial-like maps. For a given $R\in\pmb{\mathcal{B}}_d$, the resulting two pinched polynomial-like maps $R_F$ and $R_H$ are also (piecewise) analytically conjugate. However, the globalizations of these semi-global maps (defined on proper subsets of the sphere) to holomorphic correspondences on the Riemann sphere are carried out in different ways. Specifically, since the Farey (respectively, Hecke) map was defined on a quotient of $\D$ by an order $d+1$ (respectively, order two) group element, one needs to pass to a $(d+1)-$fold cover of the $R_F-$plane (respectively, a double cover of the $R_F-$plane) to construct the desired correspondences.

We carry out the construction of the correspondence from the pinched polynomial-like maps $R_F$ and $R_H$ in the next two subsections that are independent of each other. Once again, readers are invited to select the proof they find more appealing.

\subsection{Algebraic correspondence from the pinched polynomial-like map $R_F$}\label{corr_from_b_inv_sec}

In this subsection, we will promote the pinched polynomial-like map $R_F$ constructed in Theorem~\ref{farey_mating_thm} (for $R\in\pmb{\mathcal{B}}_d$) to an algebraic correspondence on the Riemann sphere that is a mating of $R$ and the Hecke group $\mathcal{H}_{d+1}$.

The first step in the passage from the map $R_F$ to the correspondence is to recognize the map $R_F$ as an explicit algebraic function. To this end, we introduce a class of algebraic maps which we term \emph{B-involutions}, and show that the maps $R_F$ indeed belong to this class. Roughly speaking, a B-involution is semi-conjugate to the M{\"o}bius involution $J(z)= 1/z$ via a rational map $Q$. It then turns out that the dynamics of the correspondence $\G$ generated by the involution $J$ and the deck transformations of the rational map $Q$ can be studied profitably from the action of the map $R_F$. In fact, the correspondence $\G$ can be regarded as the lift of the map $R_F$ by the rational map $Q$. This allows us to lift the dynamical structure of $R_F$ (which is a mating of the hybrid class of $R$ and the external map $F_d$) to the correspondence plane, and deduce that $\G^{-1}$ is a mating of $R$ and the Hecke group $\mathcal{H}_{d+1}$ in the sense of Definition~\ref{mat}.

\subsubsection{B-involutions}\label{b_inv_subsec}

Throughout the rest of this section, we will use the notation $J(z)=1/z$.

\begin{definition}[B-involutions]\label{b_inv_def}
Let $Q$ be a rational map of degree $(d+1)$ and $\mathfrak{D}$ a Jordan disc such that 
\begin{enumerate}\upshape
\item $1, -1\in\partial\mathfrak{D}$,
\item $J(\partial\mathfrak{D})=\partial\mathfrak{D}$,
\item $\mathfrak{P}\subset\partial\mathfrak{D}$ be a finite set such that $J(\mathfrak{P})=\mathfrak{P}$ and $\partial\mathfrak{D}\setminus\mathfrak{P}$ is a union of non-singular real-analytic curves, 
\item for $z\in\partial\mathfrak{D}$, we have $Q'(z)= 0\ \iff\ z\in\mathfrak{P}$, and
\item $Q\vert_{\overline{\mathfrak{D}}}$ is injective; let $\cU:=Q(\mathfrak{D})$.
\end{enumerate}
We set $\mathfrak{S}:=Q(\mathfrak{P})\in\partial\cU$, and call the meromorphic map 
$$
S:=Q\circ J\circ\left(Q\vert_{\overline{\mathfrak{D}}}\right)^{-1}:\overline{\cU}\to\widehat{\C}
$$ 
the \emph{B-involution} associated with $\cU$.
\end{definition}

The name \emph{B-involution} is motivated by the following key property: the map $S$ induces an orientation-reversing self-involution on the boundary $\partial\cU$ of its domain of definition.

Since $Q$ has no critical point on $\partial\mathfrak{D}\setminus\mathfrak{P}$, the B-involution $S$ extends to a conformal involution on a neighbourhood of $\partial\cU\setminus\mathfrak{S}$ that sends points in $\cU$ to the exterior of $\cU$.
Moreover, $S: S^{-1}(\cU)\to\cU$ is a proper (branched) covering map of degree $d$, and $S: S^{-1}(\Int{\cU^c})\to\Int{\cU^c}$ is a (branched) degree $(d+1)$ covering map.

\subsubsection{Description of the conformal mating $R_F$ as a B-involution}\label{R_F_b_inv_subsec}

We now state and prove the key technical lemma for the construction of the correspondence $\G$ from the map $R_F$; i.e., we show that the map $R_F:\overline{\cU}\to\widehat{\C}$ is a B-involution associated with a degree $d+1$ polynomial $Q$. The principal feature of $R_F$ that goes into the proof of this statement is that $R_F$ restricts to the boundary $\partial\cU$ of its domain of definition as an orientation-reversing self-involution. Roughly speaking, the Riemann sphere where the desired correspondence lives is constructed by welding two copies of $\overline{\cU}$, where $R_F\vert_{\partial\cU}$ is used as the boundary identification map. 

The subtlety in carrying out this task stems from the fact that $R_F$ is not analytic in a neighborhood of the point $\pmb{s}\in\partial\cU$, which corresponds to the parabolic fixed point $1$ of $F_d$ (see Theorem~\ref{farey_mating_thm}). Thus, it needs to be justified that the conformal welding process indeed yields a Riemann sphere (in other words, the point $\pmb{s}$ corresponds to a single point, and not a hole, on the welded object).
We will present two ways of handling this issue: an analytic approach which exploits quasisymmetry properties of $F_d$ near the point $\pmb{s}$, and a softer (somewhat more geometric) approach which determines the type of the welded Riemann surface by constructing a special meromorphic function on it.

\begin{lemma}\label{conf_mating_description_lem}
The conformal mating $R_F:\overline{\cU}\to\widehat{\C}$ of $R$ and $F_d:\cD_1\to\overline{\D}$ is a B-involution. Further, the map $Q$ can be chosen to be a polynomial of degree $d+1$ and $\mathfrak{P}$ can be taken to be the singleton $\{1\}$. 
\end{lemma}
\begin{proof}
The proof of the lemma is based on certain analytic quality of the map $R_F\vert_{\partial\cU}$ near the point $\pmb{s}$. Specifically, we will show that a quasiconformal uniformization $\D\to\cU$ conjugates the map $R_F\vert_{\partial\cU}$ to a quasisymmetric involution of the circle $\mathbb{S}^1$. The proof of this fact requires analytic control on various steps in the construction of $R_F$; more precisely, one needs to investigate the intermediate maps appearing in the proofs of Lemmas~\ref{qc_conj_fd_bd_lem},~\ref{conf_mating_lem} and Theorem~\ref{farey_mating_thm}.

We recall the notation $\mathfrak{H}_1:=\D\setminus\Int{\cD_1}$ (see Section~\ref{farey_map_subsec}). The first step in the proof of Lemma~\ref{qc_conj_fd_bd_lem}, applied to map $F_d$, constructs a closed Jordan disc $\mathfrak{T}$ from $\mathfrak{H}_1$, that touches $\mathbb{S}^1$ only at $1$ and whose boundary $\partial\mathfrak{T}$ comprises a part of $\partial\mathfrak{H}_1$ and a pair of straight line segments emanating from $1$.
The global quasiconformal map $\mathfrak{g}$ of Lemma~\ref{qc_conj_fd_bd_lem} carries $\mathfrak{T}$ onto $\overline{\mathcal{P}}$ (where $\mathcal{P}$ is an appropriate attracting petal for $R$ constructed in the proof of the same lemma), and takes $1$ to $1$.

The conformal map $\psi_R:\D\to \mathcal{A}(R)$, that conjugates $B_d$ to $R$ and sends $1$ to $\infty$, carries the petal $\mathcal{P}$ to some petal in $\mathcal{A}(R)$. Note that $\psi_R$ can be factorized as the composition of a Fatou coordinate for $\mathcal{P}$ with the inverse of a Fatou coordinate for $\psi_R(\mathcal{P})$. It follows from the asymptotics of Fatou coordinates that near $1$, the curve $\partial\psi_R(\mathcal{P})$ is the union of two smooth arcs meeting at a positive angle. In particular, $\partial\psi_R(\mathcal{P})$ is a quasicircle. Hence, $\psi_R:\mathcal{P}\to\psi_R(\mathcal{P})$ extends to a global quasiconformal map $\widecheck{\psi}_R$. Therefore, the global quasiconformal map $\widecheck{\psi}_R\circ\mathfrak{g}$ carries $\mathfrak{T}$ onto $\overline{\psi_R(\mathcal{P})}$. Hence, the global quasiconformal map $\kappa\circ\widecheck{\psi}_R\circ\mathfrak{g}$ (where $\kappa$ is the global quasiconformal homeomorphism of Lemma~\ref{conf_mating_lem}) sends $\overline{\mathfrak{H}_1}$ onto $\widehat{\C}\setminus\cU$. We also note that $\psi_R$ and $\widecheck{\psi}_R$ agree on $\mathfrak{g}(\partial\mathfrak{H}_1)$.

The explicit asymptotic development of $F_d$ near $1$ furnished in Section~\ref{para_asymp_subsec} shows that $F_d\vert_{\partial\mathfrak{H}_1}$ is a quasisymmetry. Let us choose a Riemann uniformization $\phi:\D\to\widehat{\C}\setminus\overline{\mathfrak{H}_1}$ whose homeomorphic boundary extension sends $1$ to $1$. By \cite[Theorem~3.11]{Pom}, the map $\phi$ is of the form 
$$
\zeta\mapsto 1+c(\zeta-1)^2+o((\zeta-1)^2),\ \textrm{ for some } c\in\C^*,
$$ 
near $1$. It follows that $\phi^{-1}\circ F_d\circ\phi$ is a quasisymmetry on the unit circle. Combining this with the discussion of the previous paragraph and the construction of $R_F$, we conclude that there exists a quasiconformal map $\phi_1:=\kappa\circ\widecheck{\psi}_R\circ\mathfrak{g}\circ\phi:\D\to\cU$ whose homeomorphic boundary extension conjugates $R_F\vert_{\partial\cU}$ to an orientation-reversing quasisymmetric involution $\widecheck{J}:\mathbb{S}^1\to\mathbb{S}^1$ (see Figure~\ref{seqn_of_maps_fig}).
\begin{figure}[h!]
\captionsetup{width=0.96\linewidth}
\begin{tikzpicture}
\node[anchor=south west,inner sep=0] at (0,0) {\includegraphics[width=0.96\textwidth]{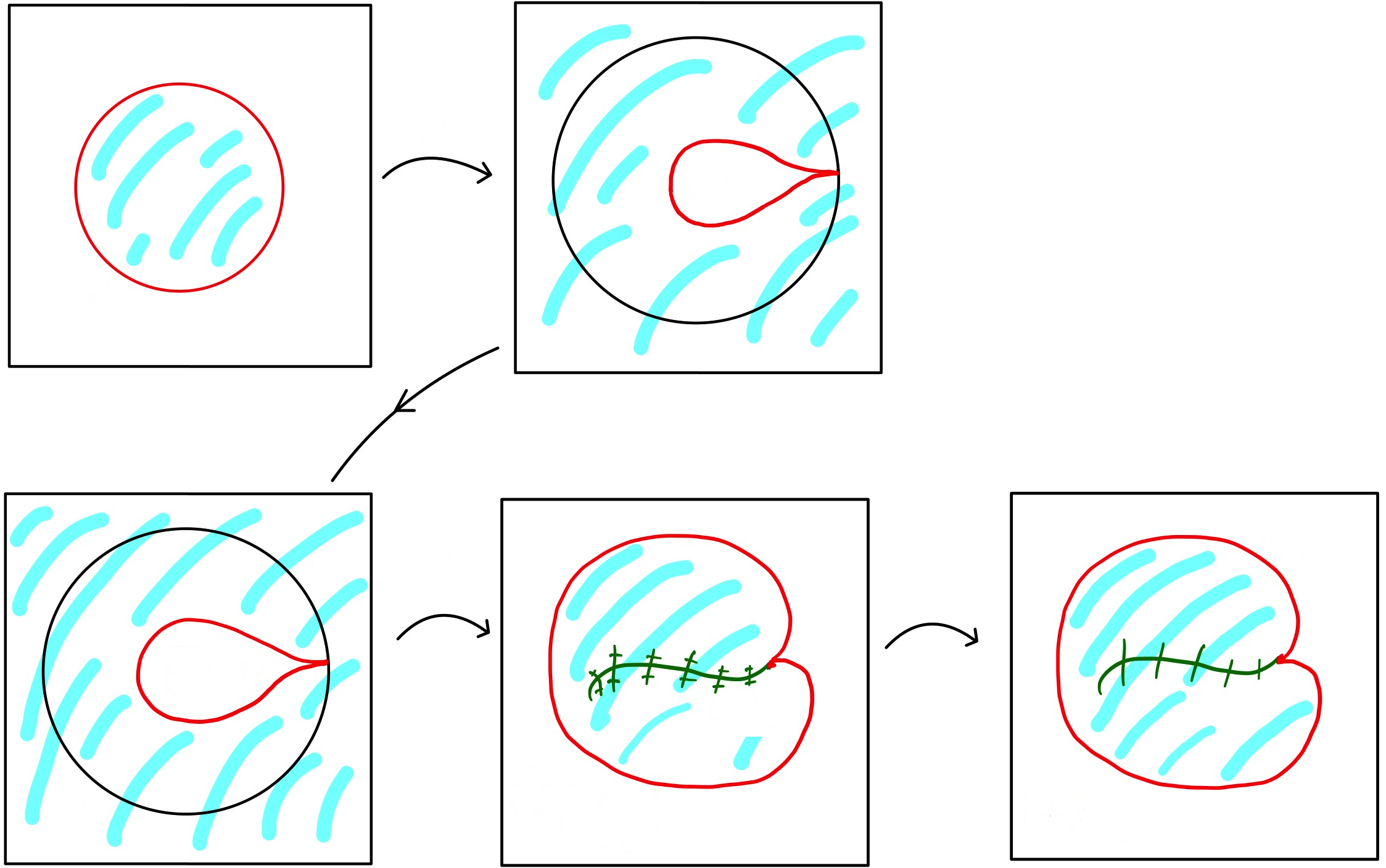}}; 
\node at (0.72,0.54) {$\mathbb{S}^1$};
\node at (1.6,5.92) {\begin{small}$\D$\end{small}};
\node at (0.9,5) {$\mathbb{S}^1$};
\node at (1.84,1.8) {\begin{small}$\mathfrak{g}(\mathfrak{H}_1)$\end{small}};
\node at (3.92,6.4) {$\phi$};
\node at (5.28,4.75) {$\mathbb{S}^1$};
\node at (6.36,6) {$\mathfrak{H}_1$};
\node at (3.84,4.5) {$\mathfrak{g}$};
\node at (3.92,2.56) {$\widecheck{\psi}_R$};
\node at (8.2,2.36) {$\kappa$};
\node at (5.4,0.4) {\begin{footnotesize}$\widecheck{\psi}_R(\mathfrak{g}(\mathfrak{H}_1))$\end{footnotesize}};
\node at (6.45,1.42) {\begin{tiny}$K(R)$\end{tiny}}; 
\node at (9.6,0.4) {\begin{footnotesize}$\widehat{\C}\setminus\overline{\cU}$\end{footnotesize}};
\node at (11.4,2.8) {$\partial\cU$};
\end{tikzpicture}
\caption{Illustrated is the definition of the quasiconformal map $\phi_1=\kappa\circ\widecheck{\psi}_R\circ\mathfrak{g}\circ\phi:\D\to\cU$.}
\label{seqn_of_maps_fig}
\end{figure}

Thus, the map $\widecheck{J}$ extends as a quasiconformal homeomorphism of $\widehat{\C}$ mapping $\D$ onto $\D^*=\widehat{\C}\setminus\overline{\D}$ with $\widecheck{J}^{\circ 2}=\mathrm{identity}$.

Define a quasiregular map 
\begin{align*}
\widecheck{Q}:\widehat{\C}\to\widehat{\C}, \hspace{1cm}
\widecheck{Q}=
\begin{cases}
\phi_1\quad \mathrm{on}\quad \overline{\D},\\
R_F\circ \phi_1\circ \widecheck{J}\quad \mathrm{on}\quad \D^*.
\end{cases}
\end{align*}

Set $\mu:=(\widecheck{Q})^*(\mu_0)$. Let $\pmb{\Phi}$ be a quasiconformal homeomorphism of $\widehat{\C}$ with $\pmb{\Phi}^*(\mu_0)=\mu$. Then $Q:=\widecheck{Q}\circ\pmb{\Phi}^{-1}$ is a quasiregular map of the Riemann sphere preserving the standard complex structure. Hence, $Q$ is rational.

By construction, $\widecheck{Q}\circ\widecheck{J}=R_F\circ\widecheck{Q}$ on $\D$. As $\mu_0$ is an $R_F$-invariant Beltrami coefficient, it follows that $\mu$ is $\widecheck{J}$-invariant. Hence, $\pmb{\Phi}\circ\widecheck{J}\circ\pmb{\Phi}^{-1}$ is a M{\"o}bius involution, and thus can be chosen to be $J(z)=1/z$.

We set $\mathfrak{D}:=\pmb{\Phi}(\D)$. Note that $Q$ carries $\overline{\mathfrak{D}}$ injectively onto $\overline{\cU}$. The fact that $\widecheck{J}$ preserves $\mathbb{S}^1$ implies that $\partial\mathfrak{D}$ is $J$-invariant.

By Theorem~\ref{farey_mating_thm}, the map $R_F\vert_{\partial\cU}$ is topologically conjugate to $F_d\vert_{\partial\mathfrak{H}_1}$, and hence has two fixed points (one of which is $\pmb{s}$). As $J\vert_{\partial\mathfrak{D}}$ is conjugate to $R_F\vert_{\partial\cU}$, it now follows that $\partial\mathfrak{D}$ contains the fixed points $\pm 1$ of $J$. After possibly pre-composing $Q$ with $z\mapsto -z$, we can assume that $Q(1)=\pmb{s}$.

We will now show that $R_F= Q\circ J\circ(Q\vert_{\overline{\mathfrak{D}}})^{-1}$.
To this end, let us assume that $z\in\overline{\mathfrak{D}}$. Then, 
$$
R_F(Q(z))=R_F(\widecheck{Q}(\pmb{\Phi}^{-1}(z)))= \widecheck{Q}(\widecheck{J}(\pmb{\Phi}^{-1}(z))).
$$ 
On the other hand, 
$$
Q(J(z))=\widecheck{Q}(\pmb{\Phi}^{-1}(\pmb{\Phi}(\widecheck{J}(\pmb{\Phi}^{-1}(z)))))=\widecheck{Q}(\widecheck{J}(\pmb{\Phi}^{-1}(z))).
$$
Hence, $R_F\circ Q= Q\circ J$ on $\overline{\mathfrak{D}}$.

The conformal conjugacy between $R_F$ and $F_d$ (on appropriate sets) shows that $R_F$ admits local holomorphic extensions around all points of $\partial\cU\setminus\{\pmb{s}\}$, but does not admit such an extension around $\pmb{s}$. The description of $R_F$ in terms of $Q$ and $J$ now implies that $1$ is the unique critical point of $Q$ on $\partial\mathfrak{D}$. Since $\partial\cU\setminus\{\pmb{s}\}$ is a non-singular real-analytic curve, we conclude that $\partial\mathfrak{D}\setminus\{1\}$ is also a non-singular real-analytic curve.

Finally, we need to argue that $Q$ can be chosen to be a polynomial. By construction, $\widehat{\C}\setminus K(R_F)$ is a simply connected domain containing exactly one critical value $v_0$ of $R_F$ such that $v_0\in\widehat{\C}\setminus\overline{\cU}$. Moreover, $R_F^{-1}(v_0)$ is a singleton that is mapped to $v_0$ by $R_F$ with local degree $d+1$ (recall that the map $F_d$ has a unique critical point of multiplicity $d$ with associated critical value $0\in\Int{\mathfrak{H}_1}$).
It now follows from the above description of $R_F$ that $Q^{-1}(v_0)$ is also a singleton. We set $Q^{-1}(v_0)=\{c_0\}$, and note that $c_0\in\mathfrak{D}^*:=\widehat{\C}\setminus\overline{\mathfrak{D}}$. After possibly post-composing $Q$ with a M{\"o}bius map and pre-composing it with a M{\"o}bius map that commutes with $J$, we can assume that $c_0=v_0=\infty$. With these normalizations, the map $Q$ is a degree $d+1$ polynomial.
\end{proof}

\begin{remark}
We refer the reader to \cite[\S 14]{LLM24} for a different proof of Lemma~\ref{conf_mating_description_lem} using a conformal welding construction.
\end{remark}

\subsubsection{Dynamics of B-involutions and associated correspondences}\label{gen_const_subsec}

Let $Q$ be a rational map of degree $(d+1)$ and $\mathfrak{D}$ a Jordan disc satisfying the conditions of Definition~\ref{b_inv_def}.

We recall the notation $\mathfrak{S}:=Q(\mathfrak{P})\in\partial\cU$, and $S:=Q\circ J\circ\left(Q\vert_{\overline{\mathfrak{D}}}\right)^{-1}:\overline{\cU}\to\widehat{\C}.$

Note that $\partial\cU\setminus\mathfrak{S}$ is a union of non-singular real-analytic curves. 
We define the \emph{fundamental tile}
$$
T^0(S):=\widehat{\C}\setminus\left(\cU\cup\mathfrak{S}\right)
$$ 
and the \emph{escaping/tiling set} 
$$
T^\infty(S):=\bigcup_{n\geq 0} S^{-n}(T^0(S)).
$$
For any $n\geq0$, the connected components of $S^{-n}(T^0(S))$ are called \emph{tiles} of rank $n$. Two distinct tiles have disjoint interior. Further, the boundary of the rank zero tile (namely, $T^0(S)$) in $T^\infty(S)$ is contained in the boundary of the rank one tiles.

The \emph{non-escaping set} $K(S)$ is defined as the complement of the tiling set in $\widehat{\C}$. We further set 
$$
\Omega:=Q^{-1}(T^\infty(S)),\ \mathrm{and}\ \cK:=Q^{-1}(K(S)).
$$ 
The common boundary of the non-escaping set $K(S)$ and the tiling set $T^\infty(S)$ is called the \emph{limit set} of $S$, denoted by $\Lambda(S)$.
Finally, recall the notation $\mathfrak{D}^*=\widehat{\C}\setminus\overline{\mathfrak{D}}$.
 
\begin{proposition}\label{tiling_open_prop}
The tiling set $T^\infty(S)$ is an open set. The non-escaping set $K(S)$ is closed.
\end{proposition} 
\begin{proof}
Let $E^k$ stand for the union of the tiles of rank $0$ through $k$. 
By definition, if $z\in T^\infty(S)$ belongs to the interior of a rank $k$ tile, then $z\in\Int{E^k}$. We claim that if $z\in T^\infty(S)$ belongs to the boundary of a rank $k$ tile, then $z$ lies in $\Int{E^{k+1}}$. Indeed, as the boundary of the rank zero tile (in $T^\infty(S)$) is contained in the boundary of the rank one tiles, it follows that the outer boundary of the rank $k$ tiles (viewed from $T^0(S)$) is contained in the inner boundary of the rank $k+1$ tiles.
Hence, 
$$
T^\infty(S)=\bigcup_{k\geq 0}\Int{E^k};
$$
i.e., $T^\infty(S)$ is an increasing union of open sets.
\end{proof}

We define the $d:d$ holomorphic correspondence $\G \subset\widehat{\C}\times\widehat{\C}$ as 
\begin{equation}
(z,w)\in\G\iff \frac{Q(w)-Q(J(z))}{w-J(z)}=0.
\label{corr_eqn}
\end{equation}

The next result shows that the correspondence $\G$ is obtained by lifting $S$ and its backward branches.
\begin{proposition}\label{inverse_lift_corr_holo}
\noindent 1) For $z\in\overline{\mathfrak{D}}$, we have that $(z,w)\in\G \iff Q(w)=S(Q(z)), w\neq J(z)$. 
\smallskip

2) For $z\in\mathfrak{D}^*$, we have that $(z,w)\in\G\ \implies\ S(Q(w))=Q(z), w\neq J(z)$.
\end{proposition}
\begin{proof}
Let us first consider $z\in\overline{\mathfrak{D}}$. For such $z$, we have $S(Q(z))=Q(J(z))$. Hence, for $z\in\overline{\mathfrak{D}}$, 
$$
(z,w)\in\G \iff Q(w)=Q(J(z))=S(Q(z)), w\neq J(z).
$$ 

Now let $z\in\mathfrak{D}^*$. For such $z$, we have that $S(Q(J(z)))=Q(z)$. Therefore, for $z\in\mathfrak{D}^*$, 
\begin{align*}
(z,w)\in\G & \iff Q(w)=Q(J(z)), w\neq J(z)\\
& \implies S(Q(w))=S(Q(J(z)))=Q(z), w\neq J(z).
\end{align*}
\end{proof}

The tiles of rank $n$ in $\Omega$ are defined as $Q$-preimages of tiles of rank $n$ in $T^\infty(S)$. If $Q$ has no critical value in a simply connected rank $n$ tile (of $T^\infty(S)$), then it lifts to $(d+1)$ rank $n$ tiles in $\Omega$ (each of which is mapped univalently under $Q$). 

\begin{figure}[h!]
\captionsetup{width=0.96\linewidth}
\begin{tikzpicture}
\node[anchor=south west,inner sep=0] at (0,0) {\includegraphics[width=0.8\textwidth]{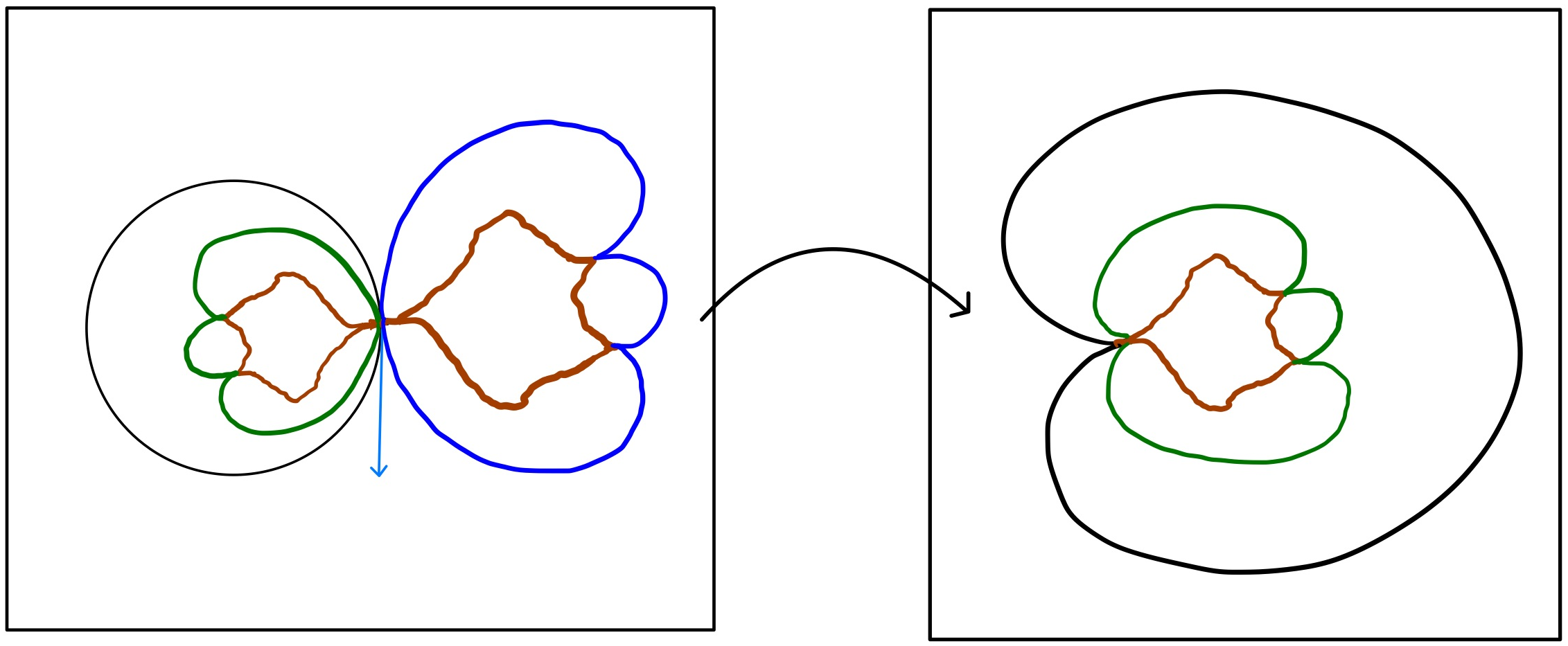}}; 
\node at (0.72,1.1) {$\partial\mathfrak{D}$};
\node at (2.45,0.9) {$1$};
\node at (1.2,3.6) {\begin{footnotesize}$Q^{-1}(T^0(S))$\end{footnotesize}};
\node at (5.36,2.8) {$Q$};
\node at (7,0.45) {$\partial\cU$};
\node at (8.25,0.96) {\begin{footnotesize}$S^{-1}(T^0(S))$\end{footnotesize}};
\node at (9.2,3.75) {$T^0(S)$};
\node at (1.92,2) {\begin{footnotesize}$\cK_+$\end{footnotesize}};
\node at (3.36,2.2) {\begin{footnotesize}$\cK_-$\end{footnotesize}};
\node at (7.84,2) {\begin{tiny}$K(S)$\end{tiny}};
\end{tikzpicture}
\caption{The map $Q$ mediates between the $S-$plane and the $\G-$plane.}
\label{two_planes_fig}
\end{figure}

\begin{proposition}\label{correspondence_partition_holo}
1) Each of the sets $\Omega$ and $\cK$ is completely invariant under the correspondence $\G$. More precisely, if $(z,w)\in\G$, then $$z\in\Omega\iff w\in\Omega,$$ and $$z\in\cK\iff w\in\cK.$$

2) $J(\Omega)=\Omega$, and $J(\cK)=\cK$.
\end{proposition}
\begin{proof}
This is immediate from the definitions of the sets (cf. \cite[Proposition~5.1]{MM23}).
\end{proof}

Let us set $\cK_+:=\cK\cap\overline{\mathfrak{D}}$ and $\cK_-:=\cK\setminus\mathfrak{D}$. 

\begin{lemma}\label{lifted_ne_top_lem}
\noindent\begin{enumerate}
\item $\cK_-\cap\partial\mathfrak{D}=\cK_+\cap\partial\mathfrak{D}=\cK_-\cap\cK_+=\mathfrak{P}$.

\item $\cK_+=J(\cK_-)$.

\item $Q$ carries $\cK_+$ (respectively, $\cK_-$) homeomorphically (respectively, as a degree $d$ branched cover) onto $K(S)$.
\end{enumerate}
\end{lemma}
\begin{proof}
1) By definition, $\cK_\pm\cap\partial\mathfrak{D}=\{z\in\partial\mathfrak{D}: Q(z)\in K(S)\}$. By construction, $\partial\cU=Q(\partial\mathfrak{D})$ meets $K(S)$ precisely at the finite set $\mathfrak{S}$. Hence, $\cK_\pm\cap\partial\mathfrak{D}=(Q\vert_{\partial\mathfrak{D}})^{-1}(\mathfrak{S})=\mathfrak{P}$. Since $\cK_+\cap\cK_-\subset\partial\mathfrak{D}$, it now follows that $\cK_+\cap\cK_- =\mathfrak{P}$.

2) The $J-$invariance of $\cK$ implies that $J(\cK\cap\mathfrak{D})=\cK\setminus\overline{\mathfrak{D}}$. By hypothesis, $J(\mathfrak{P})=\mathfrak{P}$. The result now follows from these facts and the description of $\cK_\pm\cap\partial\mathfrak{D}$ given in the previous part.

3) As $Q$ is a homeomorphism from $\overline{\mathfrak{D}}$ onto $\overline{\cU}$ and $K(S)\subset\overline{\cU}$, it follows that $\cK_+=Q^{-1}(K(S))\cap\overline{\mathfrak{D}}=(Q\vert_{\overline{\mathfrak{D}}})^{-1}(K(S))$. Hence, $Q$ carries $\cK_+$ homeomorphically onto $K(S)$. Since $Q$ is a global branched covering of degree $d+1$, it now follows that it maps $\cK_-=Q^{-1}(K(S))\setminus\mathfrak{D}$ as a degree $d$ branched cover onto $K(S)$.
\end{proof}

\begin{proposition}\label{basic_dynamics_corr_prop_holo}
1) $\cK_-$ is forward invariant, and hence, $\cK_+$ is backward invariant under $\G$.

2) $\G$ has a forward branch carrying $\cK_+$ onto itself with degree $d$, and this branch is conformally conjugate to $S:K(S)\to K(S)$.

The remaining forward branches of $\G$ on $\cK$ carry $\cK_+$ onto $\cK_-$.

3) $\G$ has a backward branch carrying $\cK_-$ onto itself with degree $d$, and this branch is conformally conjugate to $S:K(S)\to K(S)$. 
\end{proposition}
\begin{proof}
The proof of \cite[Proposition~2.6]{LMM23} applies verbatim to the holomorphic setting.
\end{proof}

\subsubsection{Correspondences as mating}\label{corr_mating_subsec}

Note that a B-involution $S:\overline{\cU}\to\widehat{\C}$ restricts to the pinched polynomial-like map $(S\vert_{\overline{S^{-1}(\cU)}}, \overline{S^{-1}(\cU)}, \overline \cU)$ with filled Julia set $K(S)$. Via the uniformizing rational map $Q:\mathfrak{D}\to\cU$, this pinched polynomial-like restriction of $S$ gives rise to a natural pinched polynomial-like restriction for the $d:1$ forward branch of the correspondence $\G$ carrying $\cK_+$ onto itself.
\smallskip

\noindent\textbf{Convention:} We will identify a B-involution with the above choice of pinched polynomial-like restrictions when discussing hybrid conjugacies.

\begin{theorem}\label{mating_exists_gen_thm}
Let $R\in\pmb{\mathcal{B}}_d$. Then, there exists a polynomial map $Q$ of degree $d+1$ and a closed Jordan disc $\overline{\mathfrak{D}}$ satisfying the conditions
\begin{enumerate}\upshape
\item $1, -1\in\partial\mathfrak{D}$,
\item $J(\partial\mathfrak{D})=\partial\mathfrak{D}$,
\item $\partial\mathfrak{D}\setminus\{1\}$ is a non-singular real-analytic curve, 
\item $Q'(1)=0$ and $Q'(z)\neq 0$ for $z\in\partial\mathfrak{D}\setminus\{1\}$, and
\item $Q\vert_{\overline{\mathfrak{D}}}$ is injective,
\end{enumerate}
such that the associated holomorphic correspondence $\F:=\G^{-1}$ on $\widehat{\C}$, where $\G$ is defined by the algebraic curve of Equation~\eqref{corr_eqn},
is a mating of the Hecke group $\mathcal{H}_{d+1}$ and the rational map $R$.
\end{theorem}
\begin{proof}
Let $R_F:\overline{\cU}\to\widehat{\C}$ be the map produced by Lemma~\ref{conf_mating_lem}, and $Q,\mathfrak{D}$ be as in Lemma~\ref{conf_mating_description_lem}.
Let $\G$ be the associated correspondence of bi-degree $d$:$d$ (see Section~\ref{gen_const_subsec}).

In light of Definition~\ref{mat}, Lemma~\ref{lifted_ne_top_lem}, and Propositions~\ref{correspondence_partition_holo} and~\ref{basic_dynamics_corr_prop_holo}, we only need to show that the $d$ branches of $\G$ on $\Omega$ are conformally conjugate to the actions of the generators $\beta_1,\cdots,\beta_d$ of $\mathcal{H}_{d+1}$ on $\D$ (see Section~\ref{hecke_map_subsec}).

Recall that $R_F$ is a B-involution associated with the domain $\cU$ such that $T^0(R_F)=\widehat{\C}\setminus\left(\cU\cup\pmb{s}\right)$, $K(R_F)=\kappa(K(R))$, and $T^\infty(R_F)=\kappa(\mathcal{A}(R))$ (where $\kappa$ is the quasiconformal homeomorphism of Lemma~\ref{conf_mating_lem}). 

The polynomial $Q:(\Omega,\infty)\to (T^\infty(R_F),\infty)$ is a branched cover of degree $(d+1)$ with a critical point of multiplicity $d$. By the Riemann-Hurwitz formula, $\Omega$ is a simply connected domain. Hence, 
$$
Q:\Omega\setminus\{\infty\}\to T^\infty(R_F)\setminus\{\infty\}
$$ 
is a $(d+1)$-to-1 covering map between topological annuli, and is thus a regular cover with deck transformation group isomorphic to $\Z/(d+1)\Z$.

Let $\tau$ be a generator of the above deck transformation group. Then, 
$$
\tau:\Omega\setminus\{\infty\}\to\Omega\setminus\{\infty\}
$$ 
is a biholomorphism such that $\tau(z)\to \infty$ as $z\to \infty$. We extend $\tau$ to a biholomorphism $\tau$ of $\Omega$ by setting $\tau(\infty)=\infty$. Then, the $d$ forward branches of the correspondence $\G$ on $\Omega$ are given by the conformal automorphisms $\tau\circ J,\cdots,\tau^{\circ d}\circ J$. 

We will now show that there exists a conformal map $\D\longrightarrow\Omega$ that conjugates the standard generators $\sigma, \rho$ of the Hecke group $\mathcal{H}_{d+1}$ to the conformal automorphisms $J,\tau$.

Let $\mathfrak{X}: \left(\D,0\right)  \to \left(T^\infty(R_F),\infty\right)$ be the conformal conjugacy between $F_d$ and $R_F$. Since the maps 
$$
\theta_1:(\D,0)\rightarrow(\D,0)\quad \textrm{and}\quad Q:(\Omega,\infty)\to(T^\infty(R_F),\infty)
$$ 
are degree $(d+1)$ branched coverings that are fully branched over $0,\infty$ (respectively), we can lift $\mathfrak{X}$ to a biholomorphism $\widetilde{\mathfrak{X}}:(\D,0)\to(\Omega,\infty)$. We normalize $\widetilde{\mathfrak{X}}$ so that it maps $C_1$ onto $\partial\mathfrak{D}\setminus \{1\}$. By construction, $\widetilde{\mathfrak{X}}$ maps $\Pi$ onto $\cT^0:=Q^{-1}(T^0(R_F))$.
\[
  \begin{tikzcd}
  \left(\D,0\right)    \arrow{d}[swap]{\mathrm{\theta_1}} \arrow{r}{\widetilde{\mathfrak{X}}} &  \left(\Omega,\infty\right) \arrow{d}{Q} \\
   \left(\D,0\right)   \arrow{r}{\mathfrak{X}}  & \left(T^\infty(R_F),\infty\right)
  \end{tikzcd}
\] 

As $\theta_1, Q$ conjugate $\sigma\vert_{C_1}, J\vert_{\partial\mathfrak{D}}$ to $F_d\vert_{\partial\mathfrak{H}_1}, R_F\vert_{\partial\cU}$ (respectively), and $\mathfrak{X}$ conjugates $F_d\vert_{\partial\mathfrak{H}_1}$ to $R_F\vert_{\partial\cU}$, it follows from the above construction that $\widetilde{\mathfrak{X}}$ conjugates $\sigma\vert_{C_1}$ to $J\vert_{\partial\mathfrak{D}}$. By the identity principle, $\widetilde{\mathfrak{X}}$ conjugates $\sigma\vert_{\D}$ to $J\vert_{\Omega}$.

On the other hand, the map $\widetilde{\tau}:=\widetilde{\mathfrak{X}}^{-1}\circ\tau\circ\widetilde{\mathfrak{X}}$ is a finite order conformal automorphism of $\D$ fixing the origin. Hence, $\widetilde{\tau}$ is a rigid rotation (around $0$) of order $d+1$. After possibly replacing $\tau$ with some iterate of it, we can assume that $\widetilde{\tau}\equiv \rho : z\mapsto \omega z$, where $\omega=e^{\frac{2\pi i}{d+1}}$.

It follows that $\widetilde{\mathfrak{X}}$ conjugates the generators $\sigma$ and $\rho$ of $\mathcal{H}_{d+1}$ to $J$ and $\tau$, respectively.
\end{proof}

\subsection{Algebraic correspondence from the pinched polynomial-like map $R_H$}\label{RtoF}

In this section we shall prove Theorem~\ref{mainthm} by a direct geometric construction of the $d:d$ correspondence $\F$ in the statement. This correspondence $\F$ will be defined on a Riemann sphere $\widetilde{S}$ which double covers the Riemann sphere $\widehat{\C}$ on which the map $R_H$ was defined in Theorem~\ref{hecke_mating_thm} (we use different names, $\widehat{\C}$ and $\widetilde{S}$, for the two Riemann spheres to avoid confusion between them). Our basic strategy for the construction (in subsection~\ref{corr_construction}) will be as follows: we shall lift $R_H$ to a $d:1$ map on one sheet of the cover, and lift $R_H^{-1}$ to a $1:d$ correspondence on the other sheet, and then `fill in the missing branches' to complete the definition of the $d:d$ correspondence $\F$ on the whole of $\widetilde{S}$. On the respective sheets the two lifts will give us copies $\mathcal K_-$ and $\mathcal K_+$ of the filled Julia set $K(R_H)\subset \widehat{\C}$, and on the complement of $\mathcal K_- \cup \mathcal K_+$ the correspondence $\F$ will be conjugate to $\bigcup_{j=1}^d \alpha_j$. Finally in subsection~\ref{right_form} we shall show that $\F$ has the form $J\circ Cov_0^P$ for some polynomial $P$, completing the proof of Theorem~\ref{mainthm}.

As we will be dealing with double covers in this section, it will be convenient to use a model of the action of the Hecke group $\mathcal H_{d+1}$ on the (open) unit disc $\D$ which puts the fixed point of the involution
$\sigma$ at the origin, and the fixed point of the rotation $\rho$ on the negative real axis in $\D$. We will use the notation illustrated in Figure \ref{disc_notation}, which depicts the same map, $\D \to \D/\langle \sigma\rangle$,
$\theta_2: z \to -z^2$ as the right-hand column of Figure \ref{farey_fig}, Section \ref{hecke_map_subsec}, but with the involution $\sigma^M$ and the rotation $\rho^M$ of the right-hand column of Figure \ref{farey_fig} now denoted by $\sigma$ and $\rho$, the (open) left-hand half of $\D$ now denoted by 
$\D_-$, and with the sets $\alpha_j^{-1}(D_-)$ (where $\alpha_j=\sigma\circ\rho^j,\ j=1\ldots d$), now denoted by $A_j$. 
The map $\theta_2$ sends $\D_-$ conformally onto $\D\setminus [0,1)$ (Figure \ref{disc_notation}). Let $A_j':=\theta_2(A_j)$. 
The Hecke map is the $d:1$ map 
$ H_d: \bigcup_{j=1}^d A_j' \to \D\setminus [0,1),$
defined by $\theta_2\circ \alpha_j \circ \theta_2^{-1}$ on each $A_j'$. Thus $H_d$ maps each $A_j'$ conformally onto $\D\setminus [0,1)$. It extends to a continuous map on the union of the closures of the sets $A_j'$,
but this extension is $2:1$ from the inner boundary arc of each $A_j$ onto the interval $[0,1]\subset \overline \D$.
\begin{figure}[h!]
\captionsetup{width=0.96\linewidth}
\begin{tikzpicture}
\node[anchor=south west,inner sep=0] at (0,0) {\includegraphics[width=0.8\textwidth]{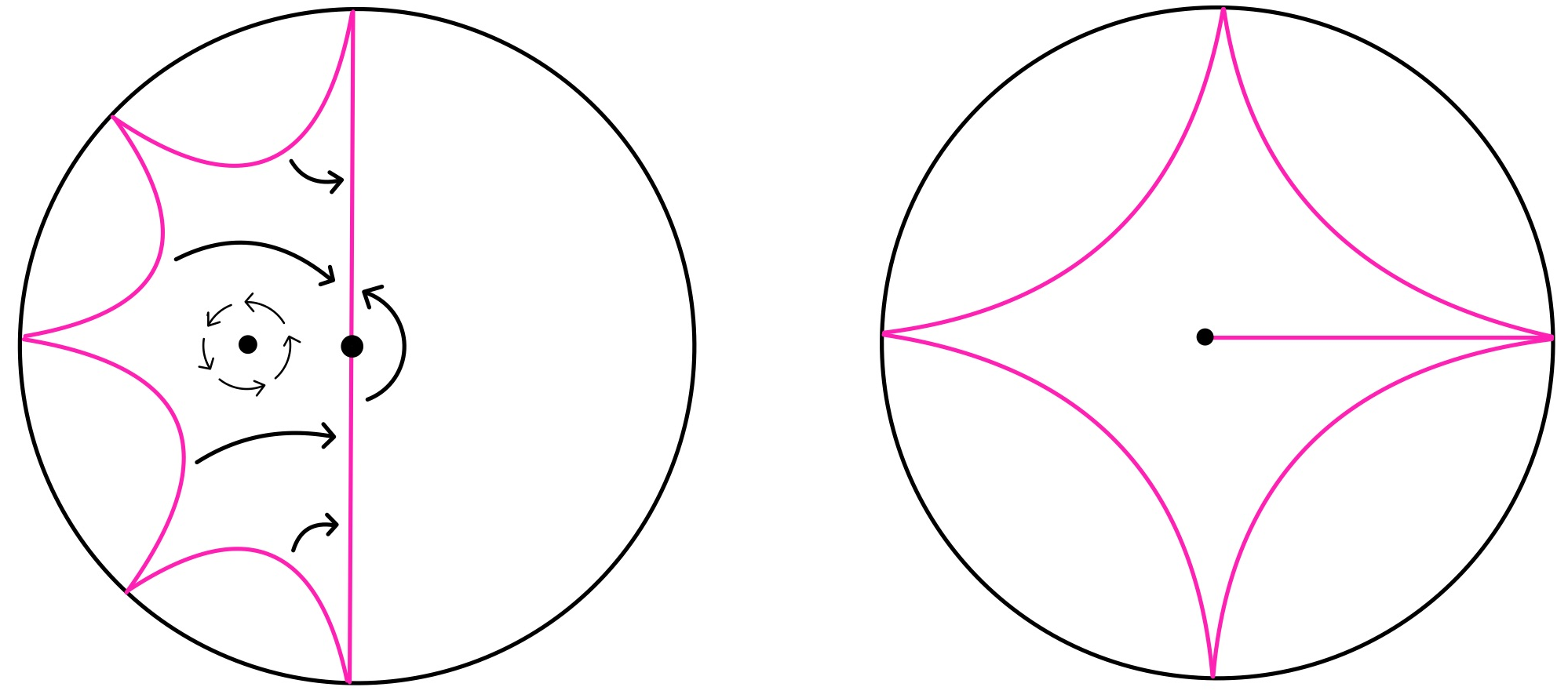}}; 
\node at (1.56,0.6) {$A_1$};
\node at (0.66,1.66) {$A_2$}; 
\node at (0.66,2.84) {$A_3$}; 
\node at (1.56,3.8) {$A_4$}; 
\node at (2.8,2.25) {$\sigma$};
\node at (1.92,2) {\begin{tiny}$\rho$\end{tiny}};
\node at (1.88,1.2) {\begin{tiny}$\alpha_1$\end{tiny}};
\node at (1.64,1.54) {\begin{tiny}$\alpha_2$\end{tiny}};
\node at (1.5,3.05) {\begin{tiny}$\alpha_3$\end{tiny}};
\node at (1.96,3.25) {\begin{tiny}$\alpha_4$\end{tiny}};
\node at (7.66,2.14) {$0$};
\node at (10.18,2.32) {$1$};
\node at (8.8,1.2) {$A_1'$};
\node at (6.75,1.2) {$A_2'$};
\node at (6.75,3.6) {$A_3'$};
\node at (8.8,3.6) {$A_4'$};
\end{tikzpicture}
\caption{Left, the Hecke map on $\D$; right, its quotient under $\sigma$.}
\label{disc_notation}
\end{figure}

We refer the reader to Definition~\ref{corr_defi} for the notion of the covering correspondence $Cov^f$, and the deleted covering correspondence $Cov_0^f$, of a 
rational map $f$, since we will use this terminology repeatedly in the current section. 
We shall use the same terms when $f$ is a holomorphic map from a proper subset $U$ of the Riemann sphere onto the whole sphere, in which case $Cov^f$ and $Cov_0^f$ will denote 
the obvious multi-valued functions from $U$ to itself.

\subsubsection{Constructing the holomorphic correspondence $\F$}\label{corr_construction}

According to Theorem~\ref{hecke_mating_thm}, translated into the notation we are using in this section (with
$\cV$ denoting the same set as in the statement of Theorem~\ref{hecke_mating_thm}), there exists a continuous map $R_H:\cV\to\widehat{\C}$, which is holomorphic on $\Int{\cV}$ and satisfies the following properties. 
\begin{enumerate}[leftmargin=10mm]
\item On a pinched neighbourhood of $K(R_H)\subset\cV$, the map $R_H$ is hybrid equivalent to $R$.
\item There exists a conformal map $\psi\equiv\psi_{R_H}:\widehat{\C}\setminus K(R_H)\to \D$ that conjugates $R_H:\cV\setminus K(R_H)\to\widehat{\C}\setminus K(R_H)$ to $\displaystyle H_d:\bigcup_{j=1}^d \left(\overline{A_j'}\cap\D\right)\to \D$.
\end{enumerate}

Let $\widetilde{S}$ be
the Riemann sphere which double covers $\widehat{\C}$, ramified at the parabolic point of $R_H$ and at $\infty$.
Topologically, $\widetilde{S}$ can be viewed as made of two copies 
$\widetilde{S}_-$ and $\widetilde{S}_+$, of the base sphere $\widehat \C$ cut along 
an arc joining the two ramification points: these copies, $\widetilde{S}_-$ and $\widetilde{S}_+$,
are glued together along opposite edges of the cut to make $\widetilde{S}$.  The computer plot in Figure \ref{double_covers_figure} is drawn using a coordinate $z$ in which the parabolic point is $z=0$ 
(corresponding to $1\in \D$ via $\psi$), and the other branch point, $z=\infty$, corresponds to $0\in \D$. With this coordinate $z$ on $\widetilde{S}$ and the appropriate choice of coordinate on $\widehat{\C}$, the projection $\pi:\widetilde{S} \to \widehat{\C}$ is the map $z \to z^2$.

\begin{figure}[h!]
\captionsetup{width=0.96\linewidth}
\begin{tikzpicture}
\node[anchor=south west,inner sep=0] at (0,0) {\includegraphics[width=0.5\textwidth]{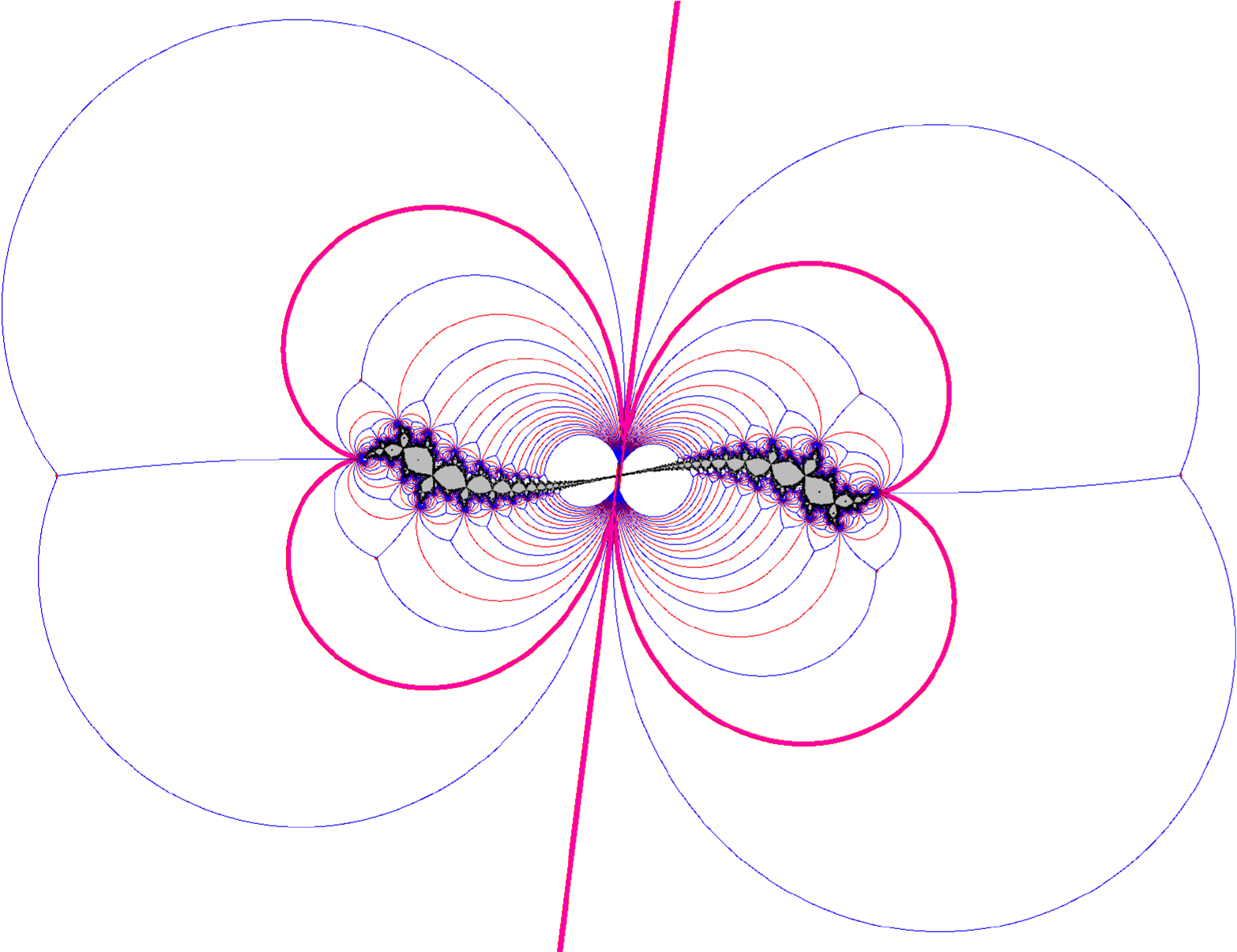}};
\node[anchor=south west,inner sep=0] at (7.5,0) {\includegraphics[width=0.36\textwidth]{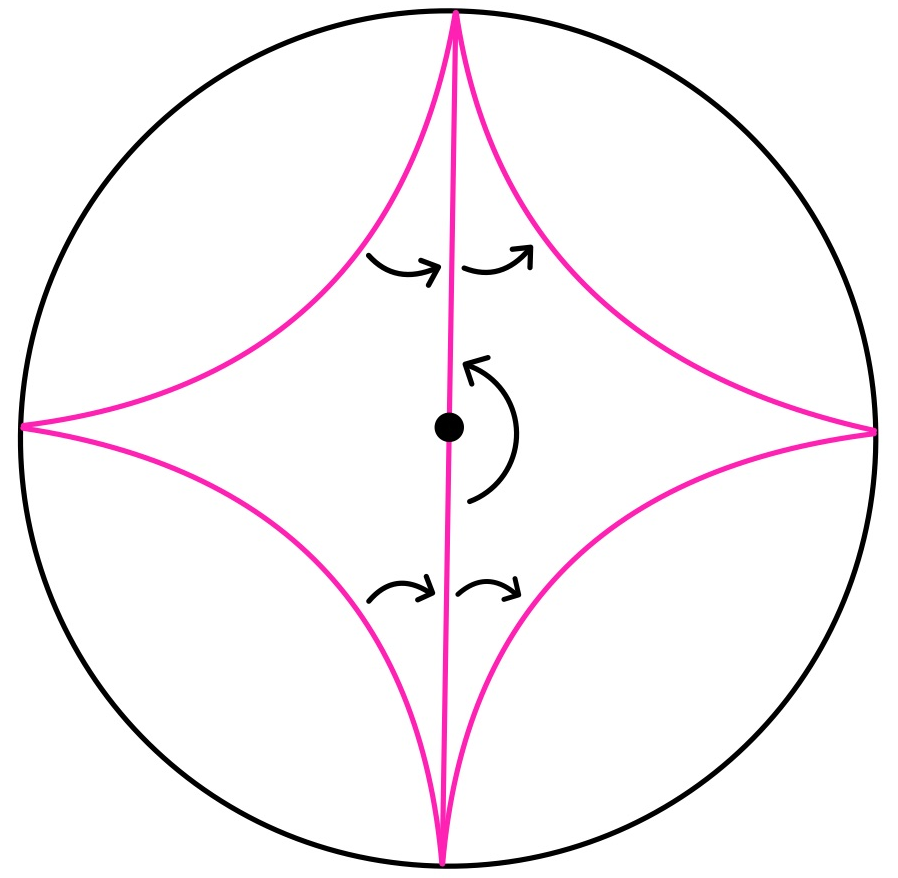}}; 
\node at (10.25,2.25) {\begin{tiny}$\sigma$\end{tiny}};
\node at (9.6,1.84) {\begin{tiny}$C_1$\end{tiny}};
\node at (8.8,2.28) {\begin{scriptsize}$\Pi$\end{scriptsize}};
\node at (10.8,2.28) {\begin{scriptsize}$\sigma(\Pi)$\end{scriptsize}};
\node at (9.96,1.32) {\begin{tiny}$\alpha_1$\end{tiny}};
\node at (9.6,1.32) {\begin{tiny}$\alpha_1$\end{tiny}};
\node at (10,3.3) {\begin{tiny}$\alpha_2$\end{tiny}};
\node at (9.6,3.24) {\begin{tiny}$\alpha_2$\end{tiny}};
\end{tikzpicture}
\caption{The double cover $\widetilde S$ and the conformal isomorphism $\widetilde \psi: \widetilde S\setminus(\cK_-\cup \cK_+)\to \D$ covering 
$\psi: \widehat{\C}\setminus K(R_H) \to \D/\langle\sigma\rangle$ (see text).
The plot is of the correspondence mating between the Douady rabbit and $\mathcal H_3$, the modular group.}
\label{double_covers_figure}
\end{figure}

Let $J$ denote the non-trivial deck transformation of the projection $\pi:\widetilde{S} \to \widehat\C$: in the computer plot Figure  \ref{double_covers_figure} this is the involution $J(z)=-z$ 
of $\widetilde{S}$ which exchanges $\widetilde{S}_-$ with $\widetilde{S}_+$and has fixed points $z=0$ and $z=\infty$. 
Set $\cK_-$ to be the lift of $K(R_H)\subset \widehat \C$ to $\widetilde{S}_-$ and $\cK_+$ to be the lift of $K(R_H)$ to 
$\widetilde{S}_+$. Thus $\cK_+=J(\cK_-)$.

The conformal isomorphism $\psi:\widehat \C\setminus K(R_H) \to \D$ lifts (under the branched coverings 
$\theta_2:\D\to\D$
and $\pi: \widetilde{S} \setminus (\cK_-\cup \cK_+)\to\widehat \C\setminus K(R_H)$) to a conformal isomorphism
$$
\widetilde\psi: \widetilde{S} \setminus (\cK_-\cup \cK_+) \to \D
$$
conjugating $J$ on $\widetilde{S} \setminus (\cK_-\cup \cK_+)$ to $\sigma$ on $\D$.

Let $\widetilde \cV_-$ denote the lift of $\cV$ to $\widetilde{S}_-$; i.e.,
$$
\widetilde \cV_-=\cK_-\cup \widetilde \psi^{-1}(A_1\cup\ldots\cup A_d)\subset \widetilde{S}_-.
$$ 
Similarly set $\widetilde\cV_+:=J(\widetilde\cV_-)$.
Observe that $R_H$ lifts (via $\pi$) to a holomorphic map $\widetilde R_H$ from $\widetilde \cV_-$ onto $\widetilde{S}_-$, 
and that $Cov_0^{\widetilde R_H}$ is a $d-1:d-1$ correspondence from $\widetilde \cV_-$ to itself. 

In the computer plot in Figure  \ref{double_covers_figure}, $\widetilde{S}_-$ and $\widetilde{S}_+$ are the parts of $\widetilde S$ to the left and right of the (thickened) straight red line through the origin, and
their respective subsets $\widetilde \cV_-$ and $\widetilde \cV_+$ are the regions bounded by thickened red arcs (in the plot each boundary is made up of $2$ arcs: it will be made up of $d$ arcs when $d>2$). 
The set $\widetilde \cV_-\setminus \cK_-$ is mapped by $\widetilde \psi$ to $\D_-\setminus \Pi$, and  $\widetilde \cV_+\setminus \cK_+$
is mapped to $\D_+\setminus\sigma(\Pi)$ (note that the origin in $\widetilde S$ is split by $\widetilde\psi$, mapping to both $i\in \overline\D$ and $-i\in \overline\D$, and $\infty \in \widetilde S$ 
is mapped by $\widetilde\psi$ to $0\in\D$).

Define the $d:d$ correspondence $\F: \widetilde{S} \rightarrow \widetilde{S}$ to be:  
$$ \F =\left\{
\begin{array}{cl}
\widetilde R_H  &\mbox{mapping  }   \widetilde \cV_-\xrightarrow{d:1}  \widetilde{S}_-  \\
J \circ Cov_0^{\widetilde R_H}  &\mbox{mapping  }   \widetilde \cV_-\xrightarrow{d-1:d-1}   \widetilde \cV_+ \\
J\circ (\widetilde R_H)^{-1}\circ J  &\mbox{mapping  }   \widetilde{S}_+ \xrightarrow{1:d}  \widetilde \cV_+  \\
\widetilde \psi^{-1}\circ \{\alpha_1,\dots,\alpha_d\}\circ \widetilde\psi  &\mbox{mapping  } \widetilde{S}_-\setminus \widetilde \cV_- \xrightarrow{d:d}  \widetilde{S}_+\setminus \widetilde \cV_+.\\
\end{array}\right.
$$
Notice that $J$ conjugates $\F$ on $\widetilde{S}_-$ to $\F^{-1}$ on $\widetilde{S}_+$: this is true for points in $\widetilde \cV_-$ by the definition of $\F$, but it is also true for points $z\in \widetilde{S}_-\setminus \widetilde \cV_-$
since $\sigma$ conjugates $\alpha_j$ to $\alpha_{d+1-j}^{-1}$ on $\D$.

It is straightforward to check that on the boundaries of $\widetilde \cV_-$, $\widetilde \cV_+$ and $\widetilde{S}_-$ the correspondence $\F$ is continuous, and that the various branches of $\F$ are holomorphic. By conformal removability of piecewise analytic arcs, the correspondence $\F$ is holomorphic on the entire sphere $\widetilde{S}$.
By construction $\F$ is a $d:d$ correspondence and is a mating between $R$ and $\mathcal{H}_{d+1}$ in the sense of Definition~\ref{mat}. This completes the proof of the first statement in Theorem~\ref{mainthm}. 

\begin{remark}
The `obvious' way to define the $d:d$ correspondence $\F$ on $\widetilde{S}$ is to separately define the `obvious' $d:d$ correspondence $\mathcal K_-$ to $\mathcal K_+$ and the `obvious' $d:d$ correspondence $\widetilde{S}\setminus(\mathcal K_- \cup \mathcal K_+) \to \widetilde{S}\setminus(\mathcal K_- \cup \mathcal K_+)$ and then show that these match on the boundary $(\partial \mathcal K_- \cup \partial\mathcal K_+)$. However this approach is problematic if the boundary is not locally connected. Defining $\F$ on the pinched neighborhoods $\widetilde \cV_\pm$ of $\mathcal K_\pm$ overcomes this problem. 
\end{remark}

We now turn to the proof of the second statement of Theorem~\ref{mainthm}.

\subsubsection{Constructing a polynomial $P$ such that $\F=J\circ Cov_0^P$}\label{right_form}

\begin{prop}\label{formofF}
The correspondence $\F$ has the form $J\circ Cov_0^P$ for some polynomial $P$ of degree $d+1$. 
\end{prop}

\begin{proof}
As in the previous subsection, we shall take the action of the Hecke group $\mathcal H_{d+1}$ on the disc $\D$ to be that in which the involution $\sigma$ fixes $0\in \D$, 
the parabolic group element $\alpha_1=\sigma\rho$ fixes $-i\in \overline \D$, and the parabolic element $\alpha_d=\sigma\rho^{-1}$ fixes $+i\in \overline \D$ (Figure \ref{disc_notation}).

Let $\Delta_\rho\subset \D$ denote the fundamental domain for $\rho$ bounded by the geodesics $L$ and $\rho(L)$ in $\D$, from the fixed point of $\rho$ to $-i$ and $+i$
(see the right-hand picture in Figure \ref{Qfigure}). 
Let $\Delta_\sigma:=\D_-$ (the left-hand half of $\D$, a fundamental domain
for $\sigma$), and let $T:=\Delta_\sigma \cap \Delta_\rho$, a fundamental domain for $\mathcal H_{d+1}$: the set
$\{g(T):g\in \mathcal H_{d+1}\}$ defines a tessellation of $\D$, with one tile for each group element $g$.

Define $\Delta_P$ to be the subset $\widetilde\psi^{-1}(\Delta_\rho)\cup \cK_+$ of the Riemann sphere $\widetilde S$.
In Figure \ref{Qfigure}, $\tilde\psi^{-1}(T)$ has boundary the outer pair of thick blue arcs on the left of the computer plot, together with the thick red straight line through 
the origin ($\infty$ is on this boundary: it maps to $0\in \D$ under $\tilde\psi$). Note that $\Delta_P=\widetilde{S}_+\cup \widetilde\psi^{-1}(T)$ and that it is simply-connected.

\begin{figure}[h!]
\captionsetup{width=0.96\linewidth}
\begin{tikzpicture}
\node[anchor=south west,inner sep=0] at (0,0) {\includegraphics[width=0.42\textwidth]{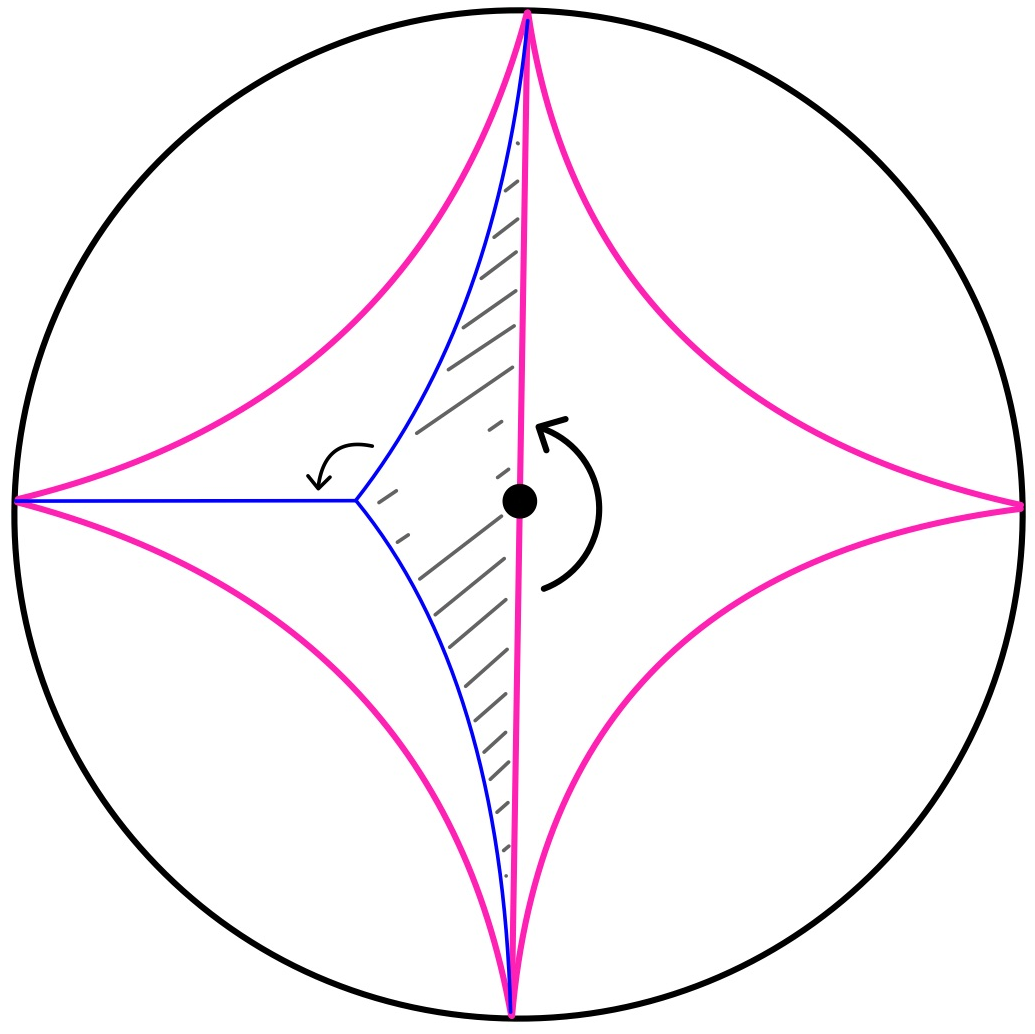}};
\node[anchor=south west,inner sep=0] at (6,0) {\includegraphics[width=0.5\textwidth]{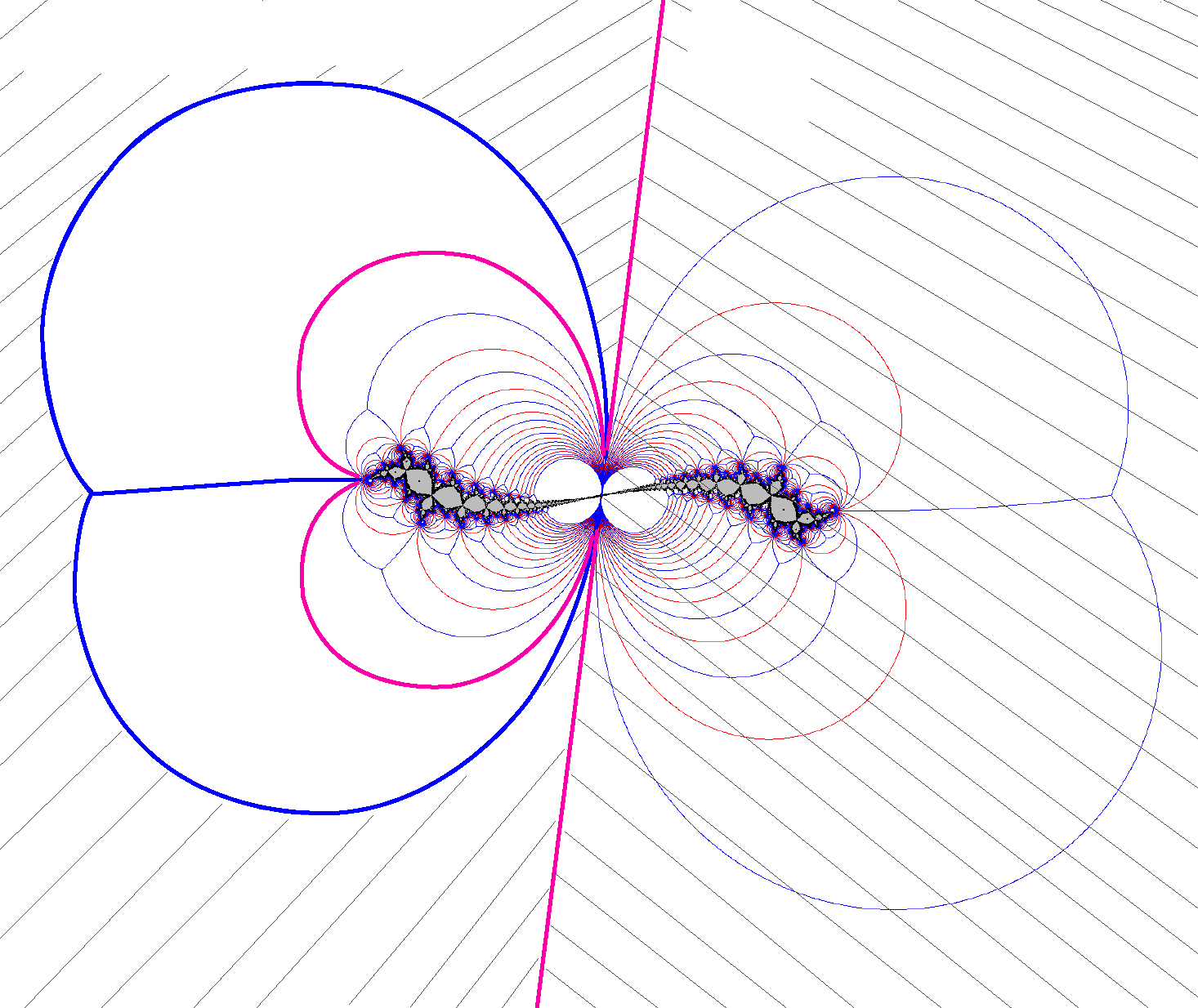}}; 
\node at (3.24,2.66) {\begin{scriptsize}$\sigma$\end{scriptsize}};
\node at (2.1,2.1) {\begin{scriptsize}$L$\end{scriptsize}};
\node at (2.32,2.8) {\begin{scriptsize}$T$\end{scriptsize}};
\node at (1.69,3.1) {\begin{tiny}$\rho$\end{tiny}};
\node at (2,3.54) {\begin{tiny}$\rho(L)$\end{tiny}};
\node at (7.4,5.16) {\begin{scriptsize}$\widetilde{\psi}^{-1}(T)$ ($\subset\widetilde{S}_-$)\end{scriptsize}};
\node at (10,4.96) {\begin{scriptsize}$\widetilde{S}_+$\end{scriptsize}};
\end{tikzpicture}
\caption{On the left, the tile $T=\Delta_\sigma\cap\Delta_\rho$, and on the right, $\Delta_P=\widetilde\psi^{-1}(T)\cup \widetilde S_+$. We construct a degree $3$ (in general $d+1$) rational map $P$ from 
$\widetilde S=\widehat \C$ onto the sphere $\Sigma=\Delta_P/_ \sim $ (where $\sim$ is the boundary identification $\widetilde\psi^{-1}(\rho)$), such that $\Delta_P$ is a fundamental domain for $Cov^P$, 
i.e. modulo boundaries $P$ maps $\Delta_P$ bijectively onto $\Sigma$, and maps $\widehat \C\setminus \Delta_P$ onto $\Sigma$ as a double (in general $d$-fold) branched-covering.}
\label{Qfigure}
\end{figure}

Define an equivalence relation on the boundary of $\Delta_P$ by $\widetilde\psi^{-1}(z)\sim \widetilde\psi^{-1}(\rho z)$ for all $z\in L$ and denote the quotient of  $\Delta_P$ under this identification 
(a sphere) by $\Sigma$.
We define a degree $d+1$ surjective holomorphic map $P:\widetilde S \to \Sigma$ by putting together the following maps:

\begin{enumerate}
\item
$identity/_\sim: \Delta_P=\widetilde{S}_+\cup \widetilde\psi^{-1}(T) \to \Sigma=(\widetilde{S}_+\cup \widetilde\psi^{-1}(T))/_\sim$. 

\item
for $1\le j \le d$,\ \  $\widetilde\psi^{-1}\rho^{-j}\widetilde\psi: \widetilde\psi^{-1}(\rho^j(T)) \to \widetilde\psi^{-1}(T)/_\sim$

(together these $\widetilde\psi^{-1}\rho^{-j}\widetilde\psi$ make up a $d:1$ map 

$\widetilde\psi^{-1}(\rho(T))\cup\ldots\cup\widetilde\psi^{-1}(\rho^d(T))\to \widetilde\psi^{-1}(T)/_\sim$);

\item 
$J\circ \widetilde R_H: \widetilde \cV_- \to \widetilde{S}_+$, also a $d:1$ map 

(which further restricts to a $d:1$ map $\cK_- \to \cK_+$).
\end{enumerate}

These definitions match on boundaries, and form a well-defined degree $(d+1)$ holomorphic branched-covering map of spheres $P: \widetilde S  \to \Sigma$. 
Choosing coordinates so that $\infty$ corresponds to the fixed 
point of $\rho$, the map $P$ becomes a polynomial. To verify that $\F=J\circ Cov_0^P$, we examine the correspondence $Cov_0^P$ for the map $P$ we have just defined.

For any $z\in \widetilde\cV_-$, we have
$Cov_0^P(z)=\{Cov_0^{\widetilde R_H}(z)\} \cup \{J(\widetilde R_H(z))\}$.
(The first set in this union is a subset of $\widetilde \cV_-$, generically $(d-1)$ points, and the second set, which comes from the branch of $P$ which is the identity on $\widetilde S_+$,  is a single point in 
$\widetilde{S}_+$.) It follows from this expression that
$$J\circ Cov_0^P(z) =\{J(\zeta):\ \zeta \in Cov_0^{\widetilde R_H}(z)\} \cup \{\widetilde R_H(z)\}.$$
But this is precisely how we defined the image of $z\in \widetilde\cV_-$ under the correspondence $\F$.
Similar checks for $z$ in each of the other regions in our definition of the correspondence $\F$ confirm that $J\circ Cov_0^P(z)=\F(z)$ for all $z \in \widetilde S$, completing the proof of the Proposition and hence that of Theorem~\ref{mainthm}.
\end{proof}

\begin{remark}

{\ \ \ }
\begin{enumerate}[leftmargin=8mm]
\item
The critical points of $P$ are the points $z$ where the cardinality
of $Cov_0^P(z)$ is strictly less than $d$, namely:
\begin{enumerate}

\item
the critical points of $\widetilde R_H$ (that is, the lift to $\widetilde S_-$ of the critical points of $R$ in $\cK(R)$); these become critical points of $P$, with the same multiplicities;

\item
the fixed point of $\widetilde\psi^{-1}\circ \rho\circ \widetilde\psi$; this is a critical point of $P$ of maximal multiplicity;

\item
the parabolic fixed point $z_0$ of $\F$; this is a simple critical point of $P$ since it is fixed by both $\widetilde R_H$ and $J$ and therefore appears in both of the sets  
$\{Cov_0^{\widetilde R_H}(z_0)\}$ and $ \{J(\widetilde R_H(z_0))\}$
making up $Cov_0^P(z_0)$, thereby reducing by one the cardinality of their union.

\end{enumerate}

\item
By construction $P$ is injective (modulo boundaries) on $\Delta_P\subset \widetilde S$. 
So we can push $\F$ restricted to domain and codomain $\Delta_P$ down to $\Sigma=\Delta_P/_\sim$. 
Here it becomes a $d:1$ map from the subset $P(\widetilde S_+)$ of $\Sigma$ onto the whole sphere $\Sigma$, a map readily identified 
as the mating $R_F$ we proved in Theorem~\ref{farey_mating_thm}  to exist between the rational map $R\in\pmb{\mathcal{B}}_d$ and the Farey map $F$. 
\end{enumerate}
\end{remark}

\appendix

\section{Dictionary between Sections~\ref{corr_from_b_inv_sec} and~\ref{RtoF}}\label{dictionary_sec}

The two proofs of the main theorem given in Sections~\ref{corr_from_b_inv_sec} and~\ref{RtoF} are independent of each other. However, key roles are played by the same spaces and maps. The following table is a post facto dictionary between the two sections, where the objects on the left and the right are the same up to a M{\"o}bius change of coordinates.

\begin{center}
\begin{tabular}{ || c || c ||  }
 \hline

\begin{footnotesize} \textbf{Spaces and maps in Section~\ref{corr_from_b_inv_sec}} \end{footnotesize}  & \begin{footnotesize} \textbf{Spaces and maps in Section~\ref{RtoF}} \end{footnotesize}  \\ \hline

\begin{footnotesize} $J$  \end{footnotesize}  & \begin{footnotesize}  $J$   \end{footnotesize}       \\ \hline 
\begin{footnotesize} $Q$  \end{footnotesize}  & \begin{footnotesize} $P$     \end{footnotesize}     \\  \hline 
\begin{footnotesize}  $\overline{\mathfrak{D}}$  \end{footnotesize} & \begin{footnotesize} $\widetilde{S}_+$  \end{footnotesize} \\ \hline
\begin{footnotesize}  $\overline{\mathfrak{D}^*}= J(\overline{\mathfrak{D}})$  \end{footnotesize}                  &       \begin{footnotesize} $\widetilde{S}_- = J(\widetilde{S}_+)$    \end{footnotesize}              \\ \hline
 \begin{footnotesize}   Riemann sphere $\widehat{\C}$ on which $R_F$ is defined \end{footnotesize}   & \begin{footnotesize} Riemann sphere $\Sigma:=\Delta_P/\sim$    \end{footnotesize}     \\ \hline
 \begin{footnotesize}   $\overline{\mathcal{U}}= Q(\overline{\mathfrak{D}})\subset \widehat{\C}$  \end{footnotesize}  &    \begin{footnotesize} $P(\widetilde{S}_+)\subset \Sigma$ \end{footnotesize}    \\ \hline
 \begin{footnotesize} $R_F=Q\circ J \circ \left(Q\vert_{\overline{\mathfrak{D}}}\right)^{-1}:\overline{\mathcal{U}}\to\widehat{\C}$ \end{footnotesize}  &  \begin{footnotesize} $P\circ J \circ \left(P\vert_{\widetilde{S}_+}\right)^{-1}: P(\widetilde{S}_+)\to\Sigma$  \end{footnotesize}  \\
 \begin{footnotesize} (noting that $Q$ is injective on $\overline{\mathfrak{D}}$) \end{footnotesize} &  \begin{footnotesize} (noting that $P$ is injective on $\widetilde{S}_+$) \end{footnotesize}  \\ \hline 
\begin{footnotesize} (Open) fundamental tile $\Int{T^0(S)}=\widehat{\C}\setminus\overline{\mathcal{U}}$ \end{footnotesize} & 
\begin{footnotesize} Tile $P(\widetilde\psi^{-1}T)=\Sigma\setminus P(\widetilde{S}_+)$  \end{footnotesize}  \\ \hline 
\begin{footnotesize}  Correspondence $\mathcal{G}^{-1}$, where \end{footnotesize}  & 
\begin{footnotesize} Correspondence $\cF=J\circ \mathrm{Cov}_0^P$   \end{footnotesize}  \\ 
\begin{footnotesize} $\mathcal{G}=\mathrm{Cov}_0^Q\circ J=\left(Q^{-1}\circ Q\circ J\right) \setminus\ J$  \end{footnotesize} & 
\begin{footnotesize} $=J\circ\left(P^{-1}\circ P\ \setminus\ \mathrm{Id}\right)$  \end{footnotesize} \\ \hline
\begin{footnotesize} Figure~\ref{two_planes_fig}  \end{footnotesize}  & \begin{footnotesize} Figure~\ref{Qfigure}  \end{footnotesize}  \\ \hline
\end{tabular} 
\end{center}

\end{document}